%% file: Article.tex
\documentclass[12pt]{amsart}
\usepackage{hyperref,subcaption,float,enumitem,cleveref}
\newcommand{\gH}{H_{\infty,\infty}}

\usepackage{cite}

\newcommand\tr{(t,r)}

\newcommand\Z{\mathbb{Z}}

\usepackage{amsmath,amsthm,amssymb,mathtools,url,graphicx,pdfpages,tikz,rotating,subcaption,fullpage}

\usepackage{listings}
\usepackage{xcolor}
\usepackage{subcaption}
\definecolor{codegreen}{rgb}{0,0.6,0}
\definecolor{codegray}{rgb}{0.5,0.5,0.5}
\definecolor{codepurple}{rgb}{0.58,0,0.82}
\definecolor{backcolour}{rgb}{0.95,0.95,0.92}

\lstdefinestyle{mystyle}{
    backgroundcolor=\color{backcolour},   
    commentstyle=\color{codegreen},
    keywordstyle=\color{magenta},
    numberstyle=\tiny\color{codegray},
    stringstyle=\color{codepurple},
    basicstyle=\ttfamily\footnotesize,
    breakatwhitespace=false,         
    breaklines=true,                 
    captionpos=b,                    
    keepspaces=true,                 
    numbers=left,                    
    numbersep=5pt,                  
    showspaces=false,                
    showstringspaces=false,
    showtabs=false,                  
    tabsize=2
}

\newtheorem{theorem}{Theorem}[section]

\newtheorem{lemma}[theorem]{Lemma}
\newtheorem{corollary}[theorem]{Corollary}

\theoremstyle{definition}
\newtheorem{definition}[theorem]{Definition}

\newtheorem{remark}[theorem]{Remark}

\newtheorem{question}{Question}
\newtheorem{problem}[question]{Problem}
\newtheorem{conjecture}[theorem]{Conjecture}

\numberwithin{equation}{section}

\title{$(t,r)$ Broadcast Domination Numbers and Densities of the Truncated Square Tiling Graph}
\author{Jillian Cervantes}
\author[Harris]{Pamela E. Harris}
\address[J.~Cervantes and P.~E.~Harris]{Department of Mathematical Sciences, University of Wisconsin-Milwaukee, Milwaukee, WI 53211}
\email{\textcolor{blue}{\href{mailto:cervan32@uwm.edu}{cervan32@uwm.edu}} and \textcolor{blue}{\href{mailto:peharris@uwm.edu}{peharris@uwm.edu}}}
\begin{document}

\begin{abstract}
For a pair of positive integer parameters $(t,r)$, a subset $T$ of vertices of a graph $G$ is said to $(t,r)$ broadcast dominate a graph 
$G$ if, for any vertex $u$ in $G$, we have $\sum_{v\in T, u\in N_t(v)}(t-d(u,v))\geq r$, where  
where $N_{t}(v)=\{u\in V:d(u,v)<t\}$ and $d(u,v)$ denotes the distance between $u$ and $v$. 
This can be interpreted as each vertex $v$ of $T$ sending $\max(t-\text{d}(u,v),0)$ signal to vertices within a distance of $t-1$ away from $v$. 
The signal is additive and we require that every vertex of the graph receives a minimum reception $r$ from all vertices in $T$.
For a finite graph the smallest cardinality among all $(t,r)$ broadcast dominating sets of a graph is called the $(t,r)$ broadcast domination number.
We remark that the $(2,1)$ broadcast domination number is the domination number and the $(t,1)$ (for $t\geq 1$) is the  distance domination number of a graph.

We study
a family of graphs that arise as a finite subgraph of the truncated square titling, which utilizes regular squares and octagons to tile the Euclidean plane. 
For positive integers $m$ and $n$, we let $H_{m,n}$ be the graph consisting of $m$ rows of $n$ octagons (cycle graph on $8$ vertices). 
For all $t\geq 2$, we provide 
lower and upper bounds for the
$(t,1)$ broadcast domination number for $H_{m,n}$ for all $m,n\geq 1$. 
We give exact $(2,1)$ broadcast domination numbers for $H_{m,n}$ when $(m,n)\in\{(1,1),(1,2),(1,3),(1,4),(2,2)\}$.
We also consider the infinite truncated square tiling, denoted $H_{\infty,\infty}$, and we provide constructions of infinite $(t,r)$ broadcasts for $(t,r)\in\{(2,1),(2,2),(3,1),(3,2),(3,3),(4,1)\}$. 
Using these constructions we give upper bounds on the density of these broadcasts i.e., the proportion of vertices needed to $(t,r)$ broadcast dominate this infinite graph. 
We end with some directions for future study.
\end{abstract}
\maketitle

\section{Introduction}

Introduced by Blessing, Insko, Johnson and Mauretour \cite{Blessing} the $(t,r)$ \textit{broadcast domination number} of a graph $G$, denoted $\gamma_{t,r}(G)$, depends on two positive integer parameters $(t,r)$ with $1\leq r\leq t$. The $(t,r)$ broadcast domination number is defined as the smallest cardinality of a $\tr$ broadcast, which is a subset of vertices $T\subseteq V(G)$ satisfying 
that for any
$u \in V(G)$, 
\begin{align}
f(u)\coloneqq\sum\limits_{\substack{v \in T \\ u \in N_t(v)}}(t-d(u,v))\geq r,\label{eq:f(u)}
\end{align} where $N_{t}(v)=\{u\in V:d(u,v)<t\}$.
If $T$ is a $\tr$ broadcast, we call $v\in T$ a \textit{broadcasting vertex} of strength $t$, or simply a broadcasting vertex, if the parameter $t$ is clear from context.
Note that when $t=r=1$, then $\gamma_{1,1}(G)=|V(G)|$.
When $t=2$ and $r=1$, the $\tr$ broadcast domination number of a graph is precisely the domination number, and when $t>2$ and $r=1$ it is the $(t-1)$-distance domination number.
Thus $\tr$ broadcast domination numbers generalize domination and distance domination numbers. We recall that determining a smallest dominating set for a graph is known to be an NP-hard problem (a proof is provided by Garey and Johnson in \cite{NPhard}).
Many have considered the domination problem for specific graph families including grid graphs, complete grid graphs, cross product of paths, and graphs with minimum degree ~\cite{cherifietal,completegridgraph,DomInGraphs,Chang,DomInGraphsWithMinDegree,Jacobson}.
For a survey on the subject of domination parameters and some generalizations we recommend the work of Haynes, Hedetniemi, and Slater~\cite{HHS}.

In their work, Blessing et al.~ \cite{Blessing} gave $(t,r)$ broadcast domination numbers for small grid graphs with $(t,r)\in\{(2,2),(3,1),(3,2),(3,3)\}$. They also provided bounds for these numbers when considering arbitrarily large grid graphs. 
Motivated by Blessing et al.'s work on arbitrarily large grid graphs,  Harris, Drews, Randolph \cite{Tim} introduce the notion of an optimal $(t,r)$ broadcast on the infinite grid $\Z\times\Z$, which is defined as the  smallest proportion of vertices (density) in a $(t,r)$ broadcast of the infinite grid graph. 
Drews et al. constructed optimal $(t,r)$ broadcasts for all $t\geq 2$, and $r=1$ and $r=2$. 
Harris, Luque, Reyes Flores, and Sepulveda \cite{triangles} constructed \textit{efficient} $(t,r)$ broadcasts on the infinite triangular grid graph, which are those broadcasts that minimize the excess signal received by non-broadcasting vertices. 
Harris et al. constructed efficient $(t,r)$ broadcast dominating sets for all $t\geq r\geq 1$ by using a geometric interpretation of the shape of the region of vertices receiving reception from a broadcasting vertex. 

In the present, we consider the infinite truncated square tiling (thought of as an infinite graph) which utilizes squares and octagons to tile the Euclidean plane.
We denote this infinite graph by $\gH$ and illustrate in Figure~\ref{fig:infinite}.
We also consider
a family of graphs that arise as a finite subgraph of the truncated square tiling, which we  denote by $H_{m,n}$, where $m$ and $n$ are positive integers, and $H_{m,n}$ consists of $m$ rows each with $n$ octagons. See Figure~\ref{fig:22a} for an illustration.

\begin{figure}[h]
    \centering
    \begin{subfigure}[t]{0.3\textwidth}
        \centering
\includegraphics[height=2in,trim=.5cm .5cm .5cm .5cm,clip]{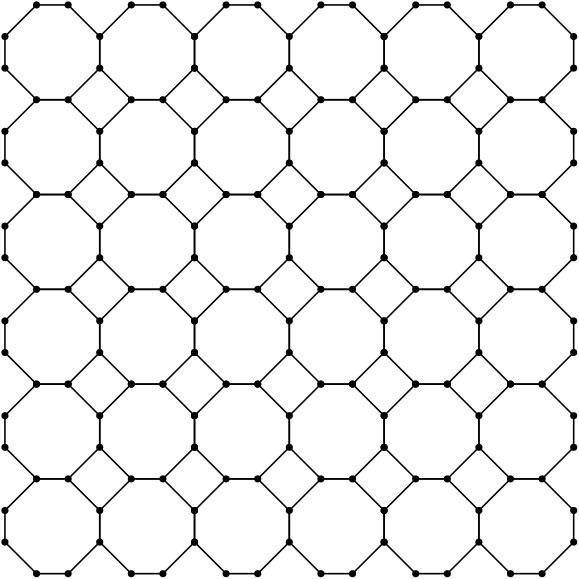}
    \caption{$\gH$}\label{fig:infinite}
    \end{subfigure}
\qquad\qquad
    \begin{subfigure}[t]{0.3\textwidth}
        \centering
\includegraphics[height=1.85in,trim=0in 0 0 0in,clip]{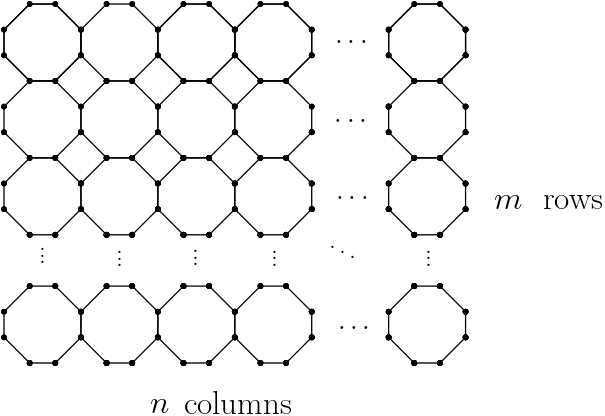}
        \caption{$H_{m,n}$}
        \label{fig:22a}
    \end{subfigure}
    \caption{The infinite truncated square graph and the finite graph $H_{m,n}$ consisting of $m$ rows each with $n$ columns of octagons.}
\end{figure}

In our first result, Theorem~\ref{thm:dom_1,n}, we 
establish the $(2,1)$ broadcast domination number (i.e. the domination number) for $H_{1,n}$ when $1\leq n\leq 4$, and for $H_{2,2}$. We illustrate these results in Figure~\ref{fig:dom_m=1} and Figure~\ref{fig:h_2_2_dom}, respectively. 
Using the domination numbers of these small graphs, in 
Theorems \ref{thm:initial_upper_bd} and \ref{2x2_bound} we provide initial upper bounds for the $(2,1)$ broadcast domination number of $H_{m,n}$ for any $m,n\geq 1$.
We improve these upper bounds in Theorem \ref{thm:better_upper_bd}.
In Theorem~\ref{thm:dom_low_bd} we establish a lower bound for the $(2,1)$ broadcast domination number of $H_{m,n}$ for any $m,n\geq 1$. Together these results show that for any $m,n\geq 1$ the $(2,1)$ broadcast  domination number of $H_{m,n}$ satisfies
    \begin{align*}
\displaystyle\Big\lceil \frac{2m+n(4m+2)}{4} \Big\rceil\leq \gamma_{2,1}(H_{m,n})\leq n(m+1) + m. 
    \end{align*}

Our results related to $\gH$ are as follows.
For $(t,r)\in\{(2,1),(2,2),(3,1),(3,2),(3,3),(4,1)\}$, we give
constructions of $(t,r)$ broadcasts for $\gH$.
Using these constructions we establish upper bounds on the density of these broadcasts,  defined as the proportion of vertices needed to $(t,r)$ broadcast dominate $\gH$. 
This gives upper bounds on the \emph{optimal density} of a $(t,r)$ broadcast of $\gH$,  defined as the minimum broadcast density over all $(t,r)$ broadcasts for $\gH$, which we denote by $\delta_{t,r}(\gH)$.

In Figure~\ref{fig:densities}, 
we illustrate each of the $(t,r)\in\{(t,r)\in\{(2,1),(2,2),(3,1),(3,2),(3,3),(4,1)\}$ broadcasts for $\gH$.
The circled vertices are broadcasting vertices, and 
each colored region depicts the $t-1$ neighborhood of a broadcasting vertex.
Throughout, we refer to the $t-1$ neighborhood of a broadcasting vertex as its \textit{reach}.
Whenever a non-broadcasting vertex $u$ lies within reach of multiple broadcasting vertices $v_1,v_2,\ldots,v_k$ the signal received from each broadcasting vertex $v_i$ is added together to compute the reception for $u$.
Based on the regularity of the $(2,2)$ and $(3,3)$ broadcasts presented in Theorem \ref{thm:density22} and Theorem \ref{thm:density33}, respectively, we establish upper bounds for the $(2,2)$ and $(3,3)$ broadcast domination number of the finite graph $H_{m,n}$ for all $m,n\geq 1$, see Corollaries \ref{cor:22bound} and  \ref{cor:33bound}, respectively.

\begin{figure}
    \centering
    \begin{subfigure}[t]{0.45\textwidth}
        \centering
      \includegraphics[width=2.3in,  height=2.3in,keepaspectratio, trim=0in 0 0 1.2in, clip]{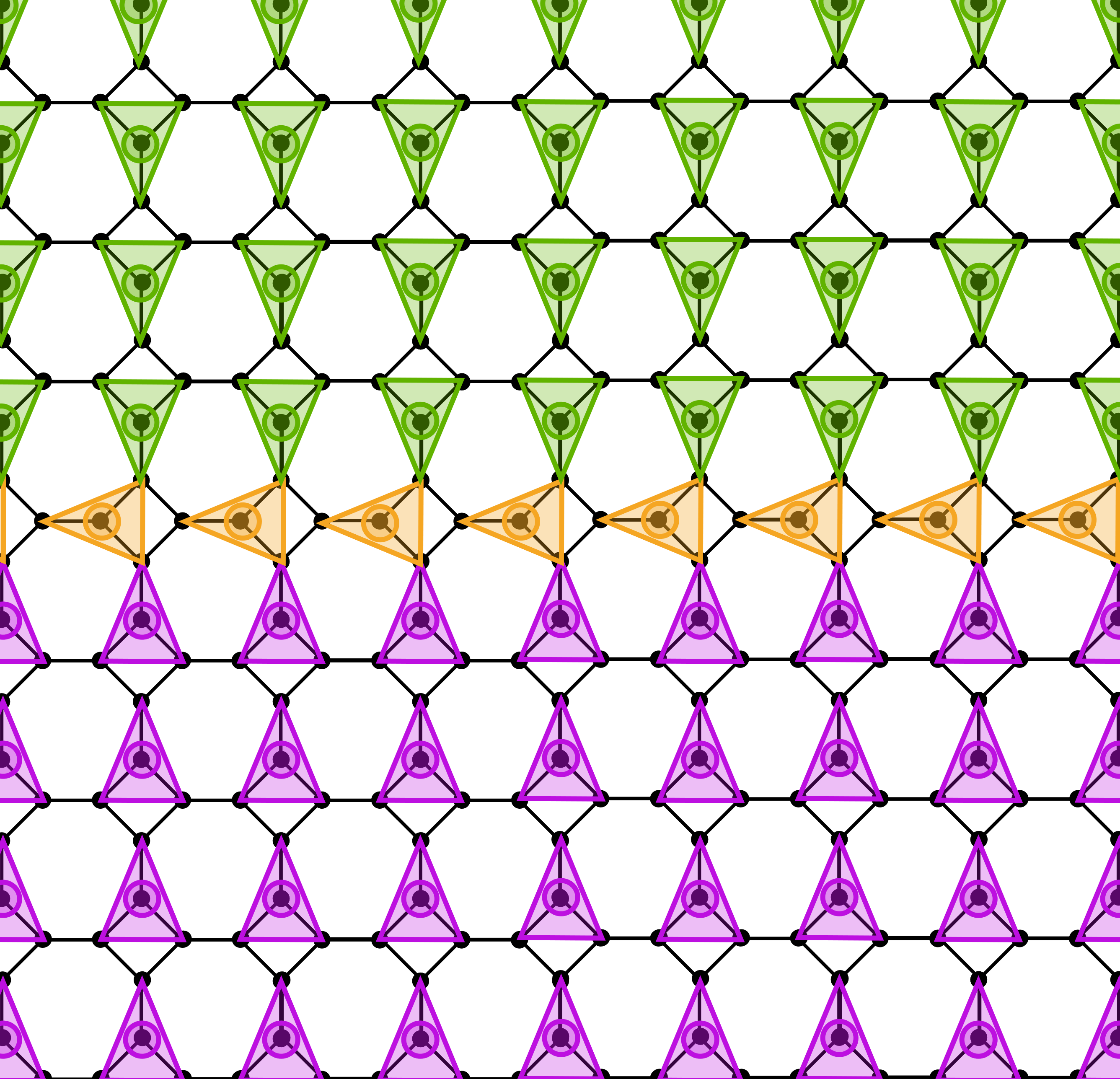}
    \caption{The $(2,1)$ broadcast on $H_{\infty,\infty}$ presented in Theorem~\ref{thm:density21} with density $\frac14$.}
    \label{fig:Delta21}
    \end{subfigure}
\hfill
    \begin{subfigure}[t]{0.45\textwidth}
        \centering
     \includegraphics[width=2.3in,  height=2.3in,keepaspectratio, trim=3.5in 0 0 0, clip]{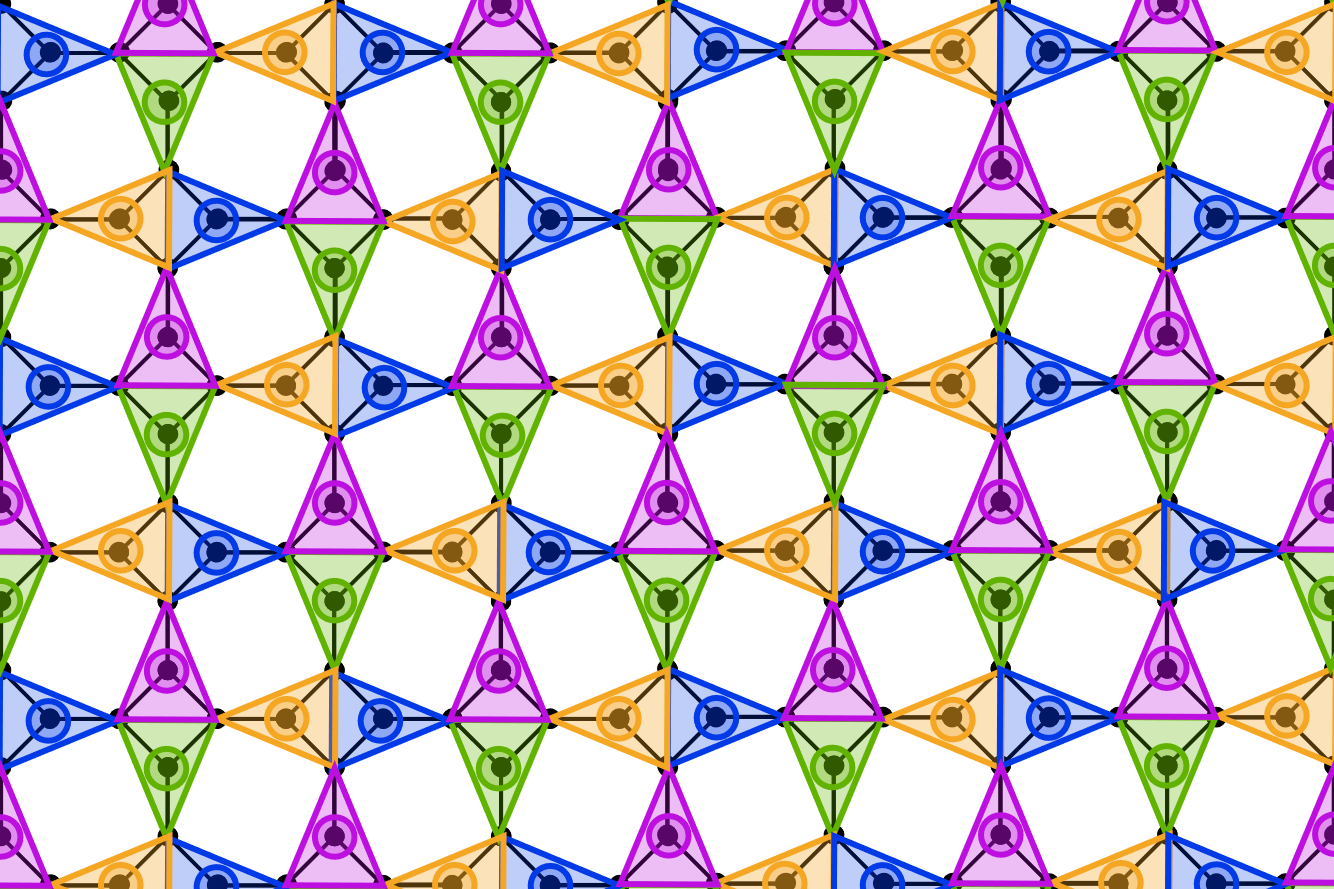}
    \caption{The $(2,2)$ broadcast on $H_{\infty,\infty}$ presented in Theorem~\ref{thm:density22} with density $\frac12$.}
    \label{fig:Delta22}
    \end{subfigure}
\\
    \begin{subfigure}[t]{0.45\textwidth}
        \centering
\includegraphics[width=2.3in, height=2.3in,keepaspectratio]{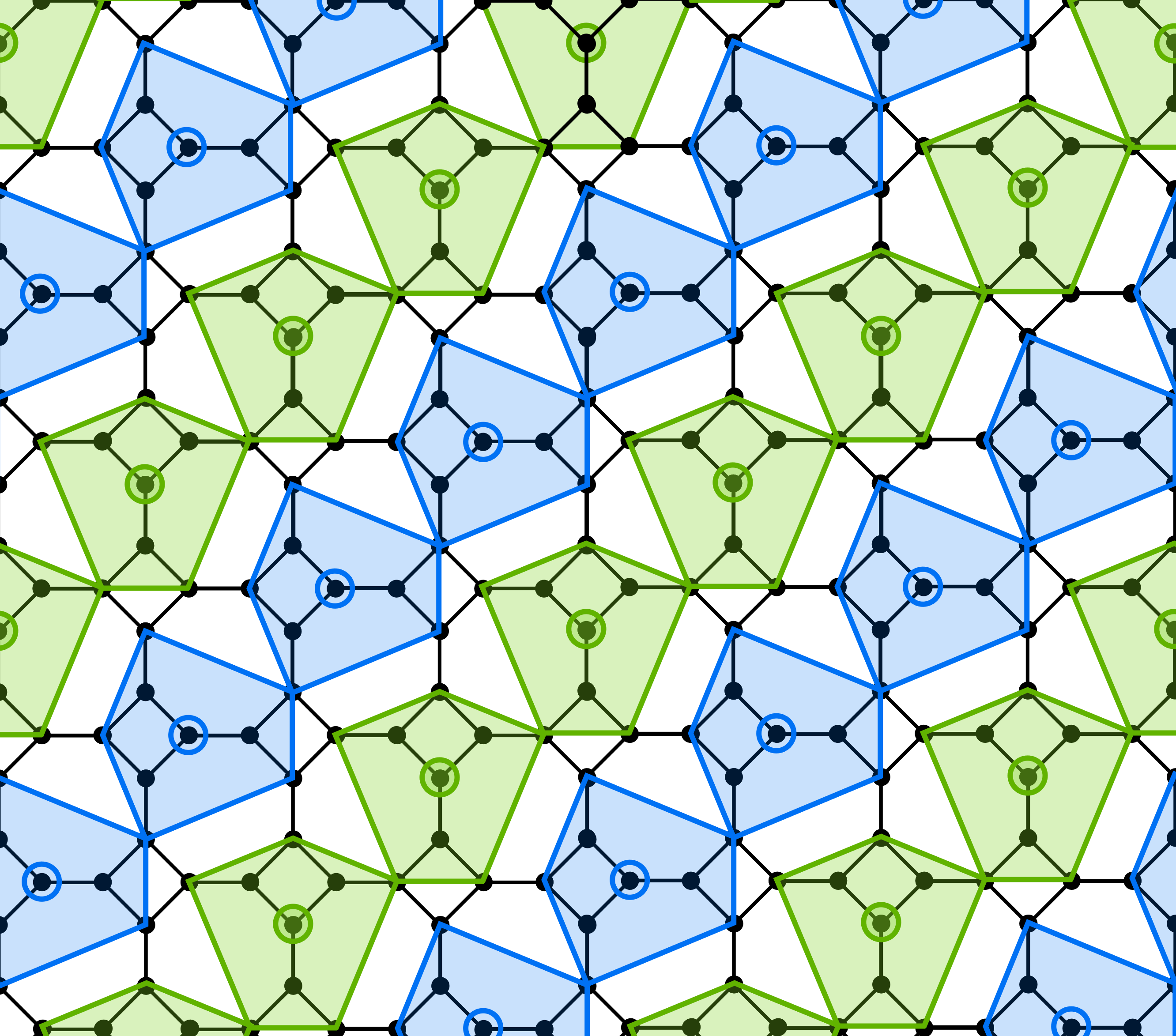}
    \caption{The $(3,1)$ broadcast on $H_{\infty,\infty}$ presented in Theorem~\ref{thm:density31} with density $\frac18$.}
    \label{fig:Delta31}
    \end{subfigure}
        \hfill
    \begin{subfigure}[t]{0.45\textwidth}
        \centering
        \includegraphics[width=2.3in,height=2.3in,keepaspectratio]{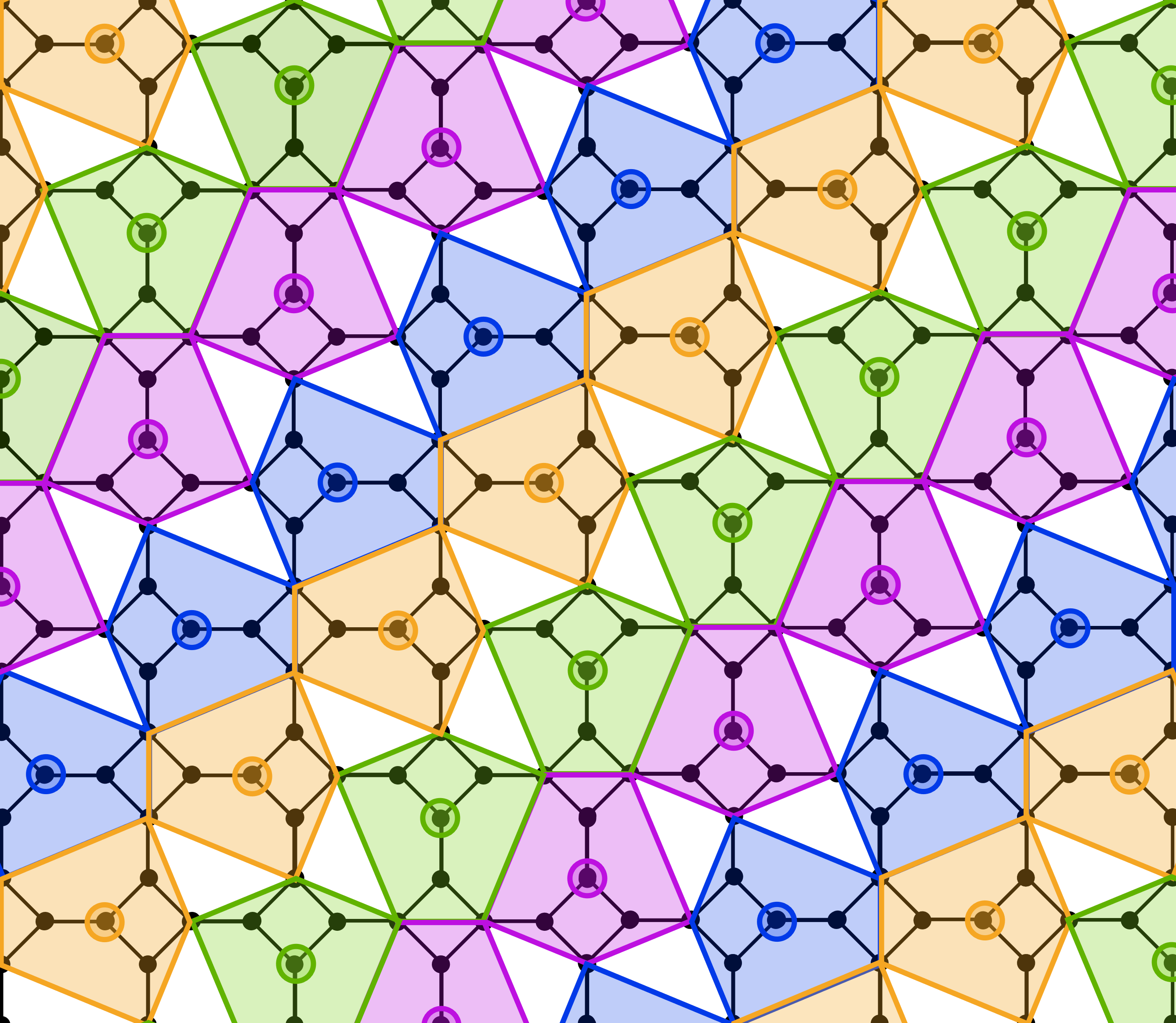}
    \caption{The $(3,2)$ broadcast on $H_{\infty,\infty}$ presented in Theorem~\ref{thm:density32} with density $\frac16$.}
    \label{fig:Delta32}
    \end{subfigure}
    \\
\begin{subfigure}[t]{0.45\textwidth}
        \centering
\includegraphics[width=2.3in,height=2.3in,keepaspectratio,trim=1.5in 0 0 0, clip]{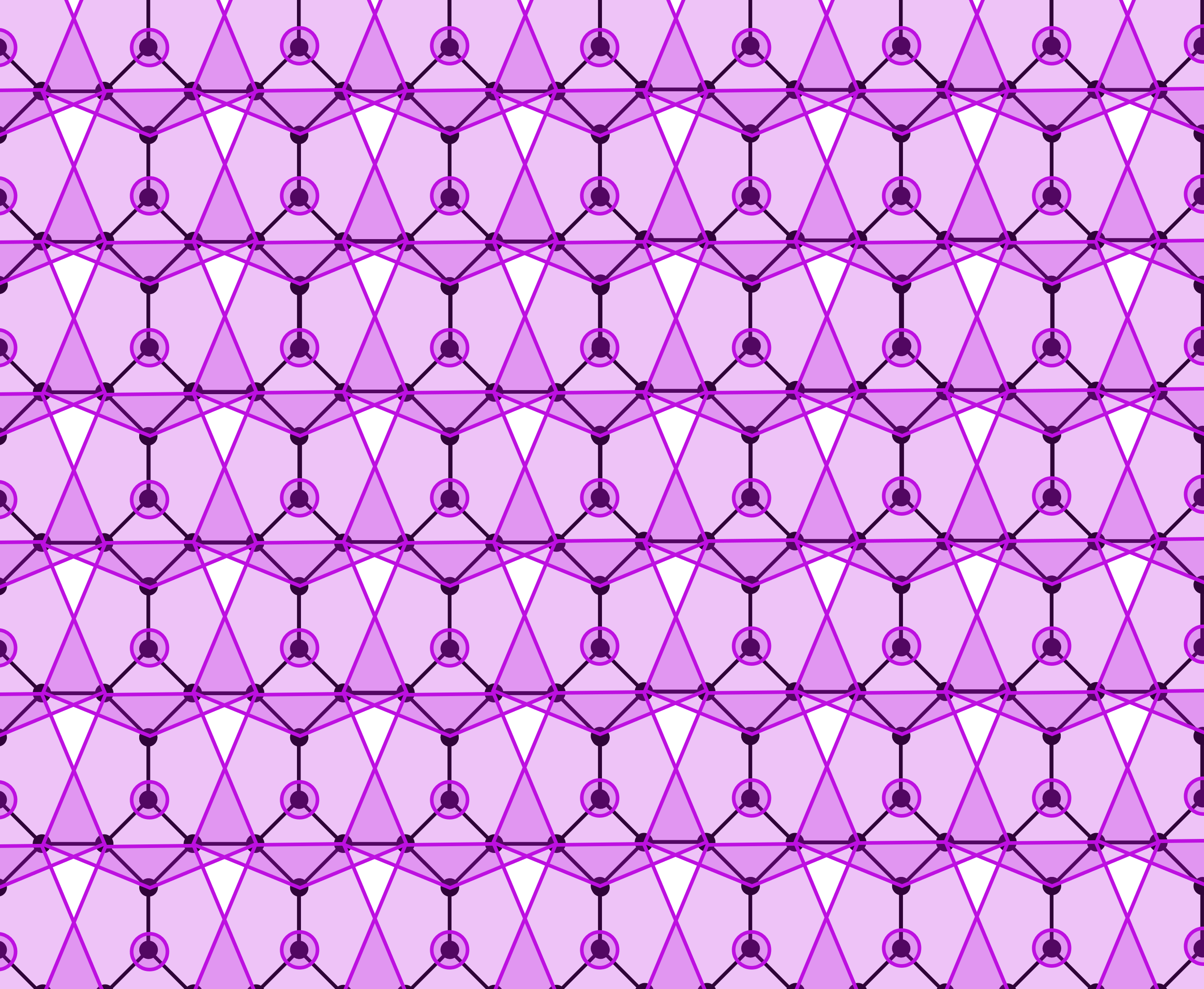}
    \caption{The $(3,3)$ broadcast on $H_{\infty,\infty}$ presented in Theorem~\ref{thm:density33} with density $\frac14$.}
    \label{fig:Delta33}
    
    \end{subfigure}
    \hfill
    \begin{subfigure}[t]{0.45\textwidth}
        \centering
    \includegraphics[width=2.3in,height=2.3in,keepaspectratio]{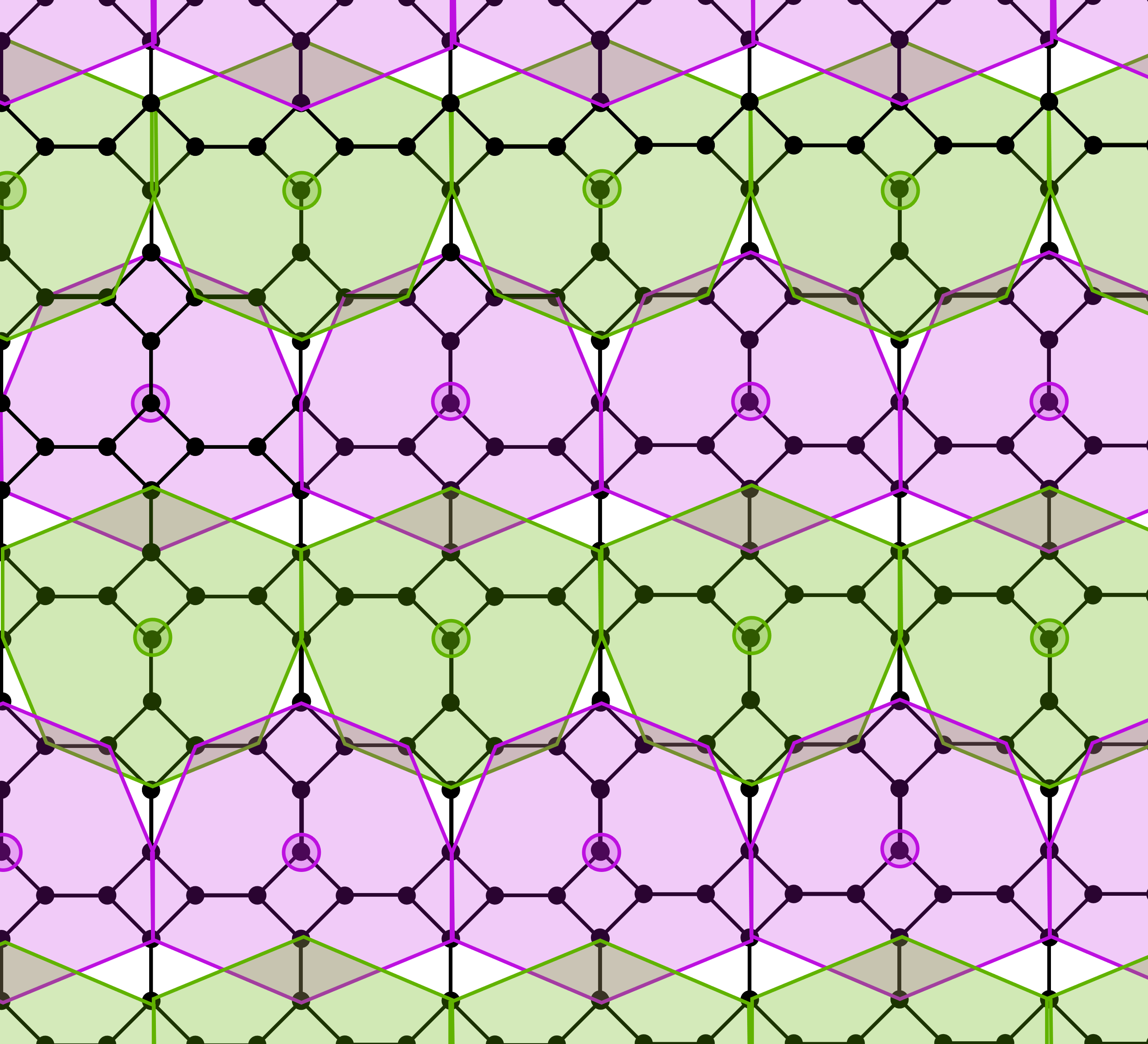}
    \caption{The $(4,1)$ broadcast on $H_{\infty,\infty}$ presented in Theorem~\ref{thm:density41} with density $\frac{1}{12}$.}
    \label{fig:Delta41}
    \end{subfigure}
    \caption{Broadcasts on $\gH$ for $(t,r)\in\{(2,2),(3,1),(3,2),(3,3),(4,1)\}$  along with the reach of broadcasting vertices.}\label{fig:densities}
\end{figure}

We conclude this introduction by remarking that this is the first paper describing $(t,r)$ broadcast dominating sets for the infinite truncated square tiling graph and $(t,r)$ broadcast domination numbers for the finite graphs $H_{m,n}$ with $m,n\geq 1$. 
Our work begins to answer the open problem posed by Harris, Insko, and Johnson in \cite[Research Project 13]{HIJ}, which was to study the $(t,r)$ broadcast domination number for this infinite graph. 
More work is needed to show that the broadcasts we present are optimal, and to establish optimal broadcasts in general. 
Thus, we provide a section with many open problems for further study, including conjectures for upper bounds on the density of $(4,r)$, with $2\leq r\leq 4$, broadcasts of $\gH$.

\begin{remark}
We have implemented code, which can be downloaded from GitHub repository \cite{cervantes2024github}. This code utilizes the Python package NetworkX \cite{hagberg2008networkx} to construct the graph $H_{m,n}$ by positioning the vertices and edges so that the illustrated graph is depicted as in Figure~\ref{fig:22a}. The code then allows us to test whether a given subset of vertices is a $(t,r)$ broadcast dominating set of $H_{m,n}$.
The user may also input the parameters $m,n$ (for small values) and a guess for the $(2,1)$ broadcast domination number. The program then tests all subsets of size one less than the input guess, to determine if there exists a dominating set of that smaller size.
\end{remark}
    
\section{Domination numbers and bounds for \texorpdfstring{$H_{m,n}$}{Hmn} graphs}

In this section we consider the graphs $H_{m,n}$ and provide bounds for the domination number of these graphs. Note that this is equivalent to giving bounds for the $(2,1)$ broadcast domination number. 
We henceforth use the notation $\gamma(H_{m,n})$ to denote $\gamma_{2,1}(H_{m,n})$. 

\begin{definition}
For positive integers $m$ and $n$, let $H_{m,n}$ denote a truncated square graph with $m$ rows consisting of $n$ octagons each. Let $V(H_{m,n})$ and $E(H_{m,n})$ denote the vertex and edge set of the graph $H_{m,n}$. We illustrate the graph $H_{m,n}$ in Figure~\ref{fig:22a}.
\end{definition}

Our first result establishes a count for the number of vertices of $H_{m,n}$.

\begin{lemma}\label{lem:m=1}
If $n\geq 1$, then $|V(H_{1,n})| = 6n+2$. 
\end{lemma}
\begin{proof}
We proceed by induction.
For $n=1$, it is evident that 
$    |V(H_{1,1})| = 8 = 6(1)+2$.
Assume $|V(H_{1,n})| = 6n+2$. We now show that $|V(H_{1,n+1})| = 6(n+1)+2$. 
The difference between $|V(H_{1,n})|$ and $|V(H_{1,n+1})|$ is $6$, since the additional octagon shares $2$ vertices with the octagon in the $n$th column.  
Then
$        |V(H_{1,n+1})| - |V(H_{1,n})| = 6$ and, induction hypothesis, we have that  
        $|V(H_{1,n+1})| - (6n+2) = 6 $. 
        Thus 
  $      |V(H_{1,n+1})| = 6 + 6n+2= 6(n+1) + 2$, as desired.
\end{proof}

\begin{theorem}\label{thm:enum_vertices}
If $m,n\geq 1$, then $|V(H_{m,n})|=2m+n(4m+2)$. 
\end{theorem}
\begin{proof}
    We prove this result by induction. 
    Fix an arbitrary $m \geq 1$. Then, for $n=1$, by Lemma~\ref{lem:m=1} we have that 
    $        |V(H_{m,1})| = 6m+2 = 2m + 1(4m + 2)
     $.
    Now, assume $|V(H_{m,n})| = 2m+n(4m+2)$ holds for $n\geq 1$.
    Notice in Figure~\ref{fig:hmn+1} that adding a column of octagons results in the addition of $6$ vertices (marked in red) for the first octagon, and another $4$ vertices (marked in blue) for each of the additional $m-1$ rows.
    Hence, by induction hypothesis, we have that 
$|V(H_{m,n+1})| = |V(H_{m,n})| + 6 + 4(m-1)= 2m+n(4m+2) + 6 + 4(m-1)= 2m + (n+1)(4m+2)$, 
which completes the proof.
\end{proof}
\begin{figure}[htp]
        \centering
    \includegraphics[width=3in]{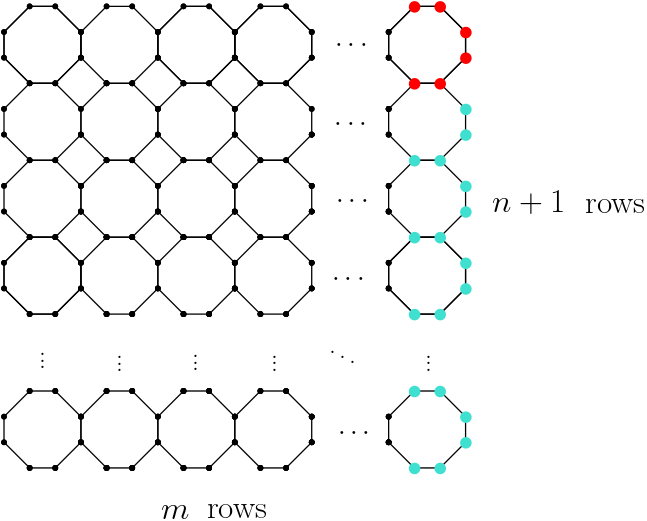}
        \caption{Adding one additional column of $n+1$ octagons to $H_{m,n}$ and create $H_{m,n+1}$.}
        \label{fig:hmn+1}
    \end{figure}

For small $m,n$ values (particularly when $mn \leq 4$), upper bounds for $\gamma(H_{m,n})$ are feasible to construct and confirm computationally. We provide such a result next.

\begin{theorem}\label{thm:dom_1,n}
    If $n \leq 4$, then
    \[ \gamma(H_{1,n}) =\begin{cases} 
      3 & \text{if } n = 1 \\
      5 & \text{if } n = 2 \\
      7 & \text{if } n = 3 \\
      9 & \text{if } n = 4.
   \end{cases}
\]

\end{theorem}
\begin{proof}
In Figures~\ref{fig:h11},~\ref{fig:h12},~\ref{fig:h13}, and~\ref{fig:h14} we provide illustrations showing that these graphs are dominated by the selected vertices. This  establishes that the claimed values are upper bounds for the domination numbers.     
Using the Python package NetworkX \cite{hagberg2008networkx} (see the accompanying supporting code in GitHub \cite{cervantes2024github}) we checked all possible subsets of vertices with cardinality one less than the constructed upper bound and confirmed that none of these subsets dominate $H_{1,n}$, for the values of $n$ considered. In doing so, we have established that these upper bounds are also lower bounds for $\gamma(H_{1,n})$ and thereby we have given exact domination numbers for these graphs as claimed.
\end{proof}

\begin{figure}
     \centering
     \begin{subfigure}[b]{0.45\textwidth}
         \centering
         \input{h_1_1_dom}
         \caption{$H_{1,1}$}
         \label{fig:h11}
     \end{subfigure}
\qquad
     \begin{subfigure}[b]{0.45\textwidth}
         \centering
         \input{h_1_2_dom}
         \caption{$H_{1,2}$}
         \label{fig:h12}
     \end{subfigure}
\\
     \vspace{30pt}
     \begin{subfigure}[b]{0.45\textwidth}
         \centering
         \input{h_1_3_dom}
         \caption{$H_{1,3}$}
         \label{fig:h13}
     \end{subfigure}
\qquad
     \begin{subfigure}[b]{0.45\textwidth}
         \centering
         \input{h_1_4_dom}
         \caption{$H_{1,4}$}
         \label{fig:h14}
     \end{subfigure}
        \caption{The circled vertices form $(2,1)$ broadcasts of $H_{1,n}$. for $1\leq n\leq 4$.}
        \label{fig:dom_m=1}
\end{figure}
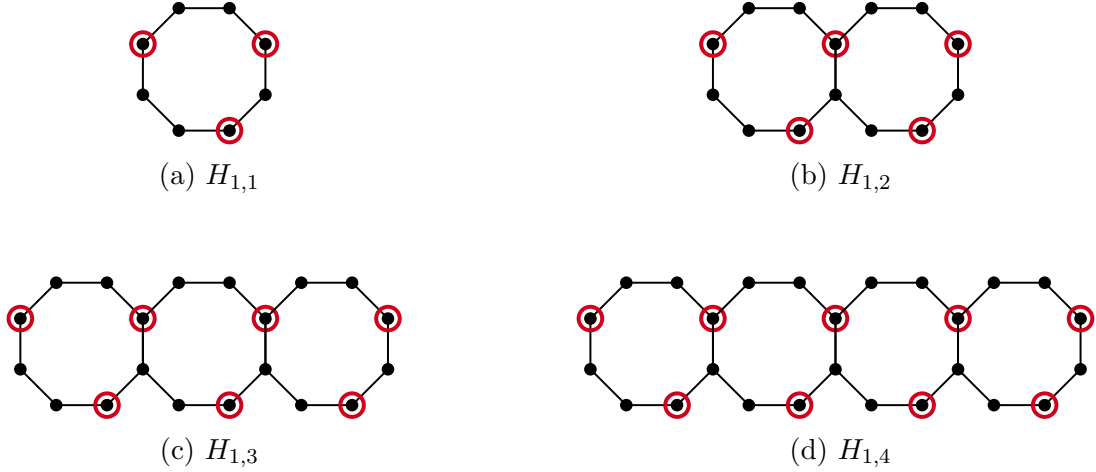

\begin{theorem}\label{thm:dom_22}
    The domination number of $H_{2,2}$ satisfies $\gamma(H_{2,2}) = 8$.
\end{theorem}
\begin{proof}
    Figure~\ref{fig:h_2_2_dom} shows a selection of 8 vertices that dominates $H_{2,2}$, proving that $\gamma(H_{2,2}) \leq 8$. 
    As with Theorem~\ref{thm:dom_1,n}, we used the Python package NetworkX \cite{hagberg2008networkx} to confirm that no subset of 7 vertices dominates $H_{2,2}$. We conclude that $\gamma(H_{2,2})$ is exactly~8.
\end{proof}

Using Theorem~\ref{thm:dom_1,n} we establish upper bounds for larger graphs by dividing graphs into subgraphs ($H_{1,1}, H_{1,2}, H_{1,3}, H_{1,4}$) and dominating the smaller subgraphs. The subgraphs are then used to ``tile'' and form the larger $H_{m,n}$ graphs. 

\begin{theorem}\label{thm:initial_upper_bd}
    If $m,n \geq 1$, then
    \[ \gamma(H_{m,n}) \leq \begin{cases} 
      \frac{9}{4}mn & \hspace{10pt}\text{if } n \equiv 0 \text{ mod } 4\\
      \frac{9}{4}m(n-1) +3m &  \hspace{10pt}\text{if } n \equiv 1 \text{ mod } 4 \\
      \frac{9}{4}m(n-2) +5m &  \hspace{10pt}\text{if } n \equiv 2 \text{ mod } 4 \\
      \frac{9}{4}m(n-3) + 7m &  \hspace{10pt}\text{if } n \equiv 3 \text{ mod } 4.
   \end{cases}
\]
\end{theorem}
\begin{proof}
We construct a dominating set using the result from Theorem \ref{thm:dom_1,n} by utilizing the subgraph $H_{1,4}$ to tile the graph $H_{m,n}$. We have four cases to consider based the residue of $n$ when reduced modulo $4$.
\begin{itemize}[leftmargin=.25in]
    \item If $n \equiv 0\mod 4$, then $H_{m,n}$ can be dominated by separating the graph into $m\left(\frac{n}{4}\right)$ ``copies" of $H_{1,4}$. By Theorem \ref{thm:dom_1,n}, $H_{1,4}$ is dominated by $9$ vertices. Hence, the graph $H_{m,n}$ can be dominated with $\frac{9}{4}mn$ vertices.
\item If $n \equiv 1\mod 4$, then we consider the first $n-1$ columns of octagons in $H_{m,n}$ to be  $m\left(\frac{n-1}{4}\right)$ copies of $H_{1,4}$, and the final column of octagons as $m$ copies of $H_{1,1}$. 
The first $n-1$ columns are dominated using $\frac{9}{4}m(n-1)$ vertices, and the final column is dominated using $3m$ vertices, per Theorem \ref{thm:dom_1,n}. Thus $H_{m,n}$ is dominated by $\frac{9}{4}m(n-1) + 3m$ vertices.

\item If $n \equiv 2\mod 4$, we consider the first $n-2$ columns of $H_{m,n}$ as $m\left(\frac{n-2}{4}\right)$ copies of $H_{1,4}$, and the final two columns as $m$ copies of $H_{1,2}$. Hence, by Theorem \ref{thm:dom_1,n}, $H_{m,n}$ is dominated by $\frac{9}{4}m(n-2) + 5m$ vertices.

\item If $n \equiv 3\mod 4$, then the first $n-3$ columns of octagons make up $m\left(\frac{n-3)}{4}\right)$ copies of $H_{1,4}$, and the final $3$ columns make up $m$ copies of $H_{1,3}$. 
Since $H_{1,3}$ is dominated with $7$ vertices, we have that $H_{m,n}$ is dominated with $\frac94 m(n-3) + 7m$ vertices.
\end{itemize}
These cases establish the upper bounds in the theorem statement.
\end{proof}

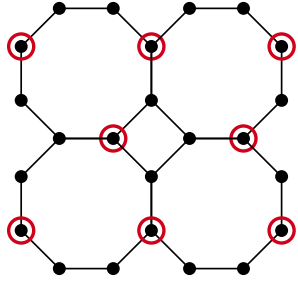
\begin{figure}
    \centering
    \resizebox{1.65in}{!}{
    \input{h_2_2_dom}}
    \caption{The eight circled vertices form a $(2,1)$ broadcast of $H_{2,2}$. }
    \label{fig:h_2_2_dom}
\end{figure}

Utilizing the domination number for the graph $H_{2,2}$, with placement of broadcasting vertices, as illustrated in Figure~\ref{fig:h_2_2_dom}, we now give a bound for the domination number of $H_{m,n}$ for $m,n\geq 2$.

\begin{theorem}\label{2x2_bound}
    If $m,n \geq 2$, where $m = 2q_1 + r_1$ and $n = 2q_2 + r_2$, with $0\leq r_1,r_2\leq 1$ then 
    \[\gamma(H_{m,n}) \leq 8\displaystyle\Big\lceil \frac{(m+r_1)(n+r_2)}{4} \Big\rceil. \]
\end{theorem}
    \begin{proof}
        We construct a dominating set for $H_{m,n}$ using the result from Theorem~\ref{thm:dom_22}. Note that $r_1, r_2 \in \{0,1\}$, depending on the parity of $m$ and $n$, respectively. 
        If $m,n$ are even, $r_1=r_2=0$ and we construct a dominating set by considering $H_{m,n} = H_{m+r_1,n+r_2}$ as $\frac{(m+r_1)(n+r_2)}{4}$  copies of $H_{2,2}$. Since $\gamma(H_{2,2}) = 8$ (Theorem~\ref{thm:dom_22}), we multiply by $8$ and conclude \[\gamma(H_{m,n}) \leq 8\displaystyle\Big\lceil \frac{(m+r_1)(n+r_2)}{4} \Big\rceil.\]

        In the case that $m$ and $n$ are not both even, $|V(H_{m,n})| < |V(H_{m+r_1,n+r_2})|$. To construct a dominating set for $H_{m,n}$ using $\gamma(H_{2,2})$, we  consider dominating $H_{m+r_1,n+r_2}$ since we are giving an upper bound to the domination number of $H_{m,n}$. Then we analogously consider $H_{m+r_1,n+r_2}$ as $\frac{(m+r_1)(n+r_2)}{4}$  copies of $H_{2,2}$. Since $\gamma(H_{2,2}) = 8$, we multiply by $8$ and conclude \[\gamma(H_{m,n}) \leq 8\displaystyle\Big\lceil \frac{(m+r_1)(n+r_2)}{4} \Big\rceil.\qedhere\]
    \end{proof}

We now construct a better upper bound for the domination number of $H_{m,n}$.
In fact, asymptotically, the  bound in our next result (Theorem~\ref{thm:better_upper_bd}) is tighter than that of Theorem ~\ref{thm:initial_upper_bd}, as the leading terms (of the form $mn$) have coefficients $1$ and $\frac{9}{4}$, respectively.

\begin{theorem}\label{thm:better_upper_bd}
    If $m,n \geq 1$, then $\gamma(H_{m,n}) \leq n(m+1) + m$.
    \end{theorem}
    \begin{proof}
        The first row of octagons in $H_{m,n}$ can be dominated with $2n + 1$ vertices, as in Figure~\ref{fig:better_upper_bd}. The subsequent row of octagons is dominated by an additional $(n+1)$ vertices, and this construction is repeated for the following $(m-2)$ rows. Hence any graph $H_{mn,}$ can be dominated by $2n+1 + (m-1)(n+1) = n(m+1) + m$ vertices.
    \end{proof}

In Theorem~\ref{thm:initial_low_bd} we establish a lower bound for the $(t,1)$ domination number of the graph $H_{m,n}$. The specialization of $t=2$ is provided below.
\begin{theorem}\label{thm:dom_low_bd}
    If $m,n \geq 1$ then\[ \gamma(H_{m,n})\geq \displaystyle\Big\lceil \frac{2m+n(4m+2)}{4} \Big\rceil.\]
\end{theorem}
\begin{proof}
    By Theorem~\ref{thm:enum_vertices}, we have that the total number of vertices $|V(H_{,mn})| = 2m+n(4m+2)$. Notice every vertex in $H_{m,n}$ has  degree at most $3$, and so any broadcasting vertex dominates itself and up to $3$ neighboring vertices. Hence we divide the total number of vertices by $4$, the number of vertices that can be dominated by a single broadcasting vertex. 
\end{proof}

\begin{figure}[h]
    \centering
    \input{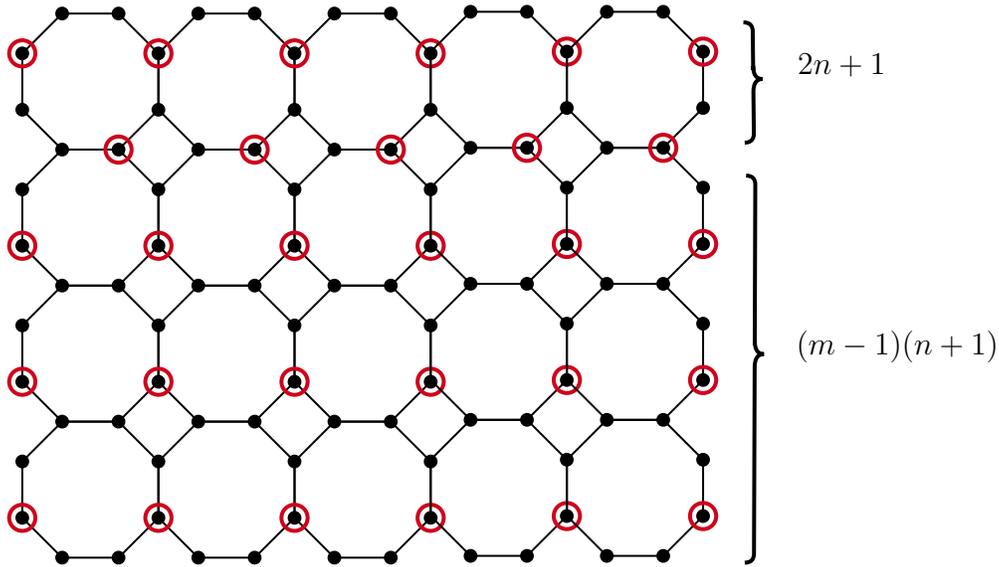}
    \caption{Construction of upper bound for $\gamma(H_{m,n})$.}
    \label{fig:better_upper_bd}
\end{figure}

\section{Bounds for the distance domination number of \texorpdfstring{$H_{m,n}$}{H	extunderscore{m,n}}}
Note that the $(t,1)$ broadcast domination number of $H_{m,n}$ is precisely the $(t-1)$-distance domination number of the graph. 
Thus 
knowing the number of vertices that are a distance  of $t-1$ away from a broadcasting vertex helps establish an initial lower bound for the $(t,1)$ broadcast domination number of $H_{m,n}$. 

The number of vertices at a fixed distance from a broadcasting vertex has been previously studied; we refer to the work of 
Goodman-Strauss and Sloane \cite{goodman-strauss2019}.
For a fixed vertex $v$, the  sequence $(a_n)_{n\geq 0}$ which counts the number of vertices a distance of $n$ away from $v$ in the infinite truncated square tiling graph is called the crystallographic \emph{coordination sequence} of a vertex. 
 The formula for the coordination sequence of a vertex is entry  \cite[\href{https://oeis.org/A008576}{A008576}]{OEIS} and is 
given by $a(0) = 1$ and thereafter $a(3k) = 8k$, $a(3k + 1) = 8k + 3$, $a(3k + 2) = 8k + 5$. 
The coordination sequence begins with
\[1, 3, 5, 8, 11, 13, 16, 19, 21, 24, 27, 29, 32, 35, 37, \ldots.\]
We adapt the definition of the coordination sequence for our purposes as follows.

\begin{definition}
   For $t\geq 2$, we let $c(t)$ be the $(t-1)$th entry in the coordination sequence. That is,  $c(t)$ denotes the number of vertices which are exactly a distance of $t-1$ away from a vertex $v\in \gH$.
\end{definition}

In what follows we let $P_v(t)$ be the subgraph of $H_{\infty,\infty}$, which consists of all vertices (and their incident edges) that are within a distance of $t-1$ away from a fixed vertex $v$. For an example, see Figure~\ref{fig:P(5)} illustrating for $P_v(5)$, and note that for a fixed $t$, the region created by $P_v(t)$ is not necessarily convex. 

\begin{figure}[h!]
    \centering
    \begin{subfigure}{.4\textwidth}
        \centering
\includegraphics[width=2in,trim=30cm 23cm 30cm 17cm,clip]{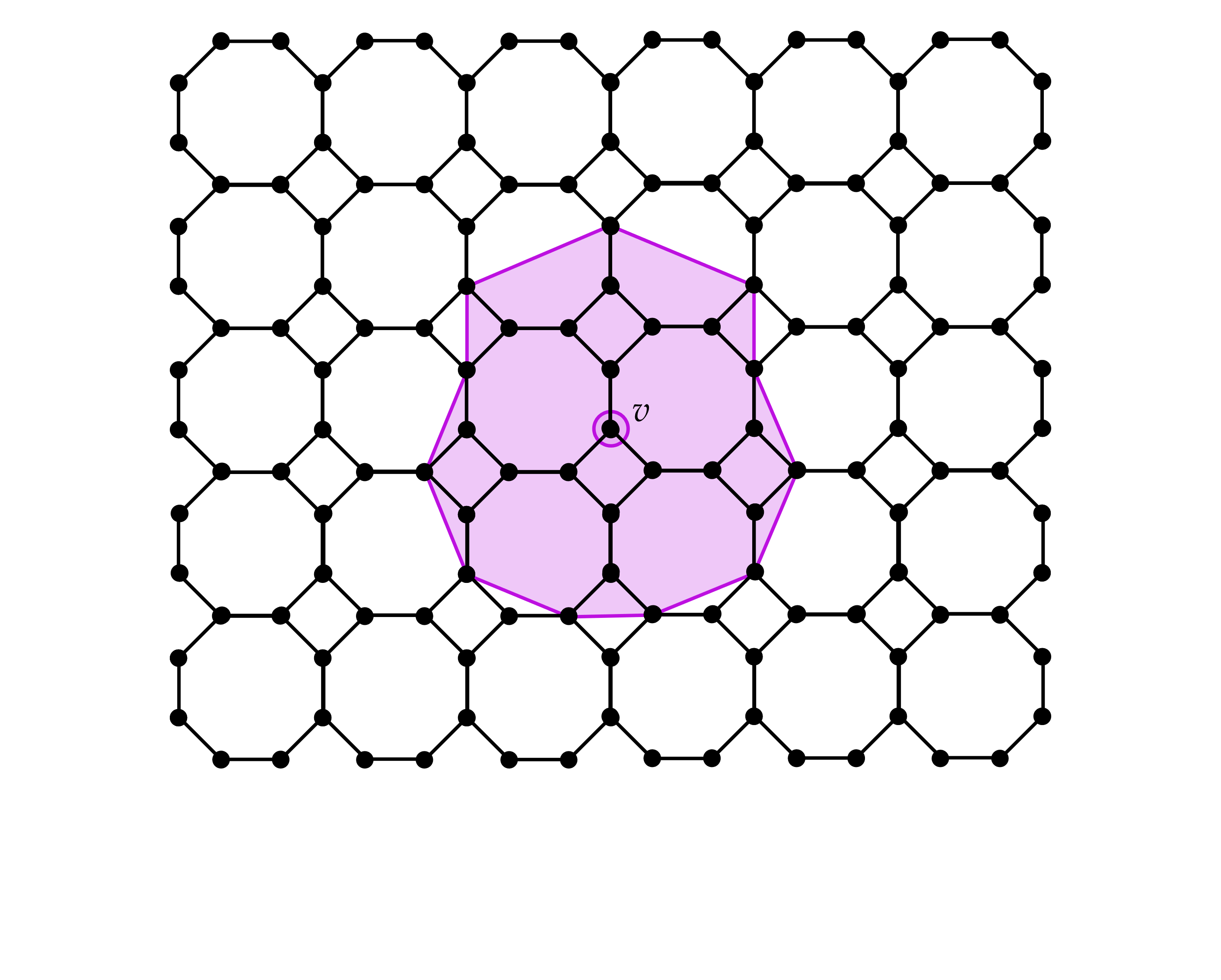}
    \caption{Area consisting of vertices at most a distance of $5$ from vertex~$v$.}
    \label{fig:P(5)}    
    \end{subfigure}
    \qquad
    \begin{subfigure}{.4\textwidth}
        \centering
        \includegraphics[width=2in,trim=30cm 23cm 30cm 17cm,clip]{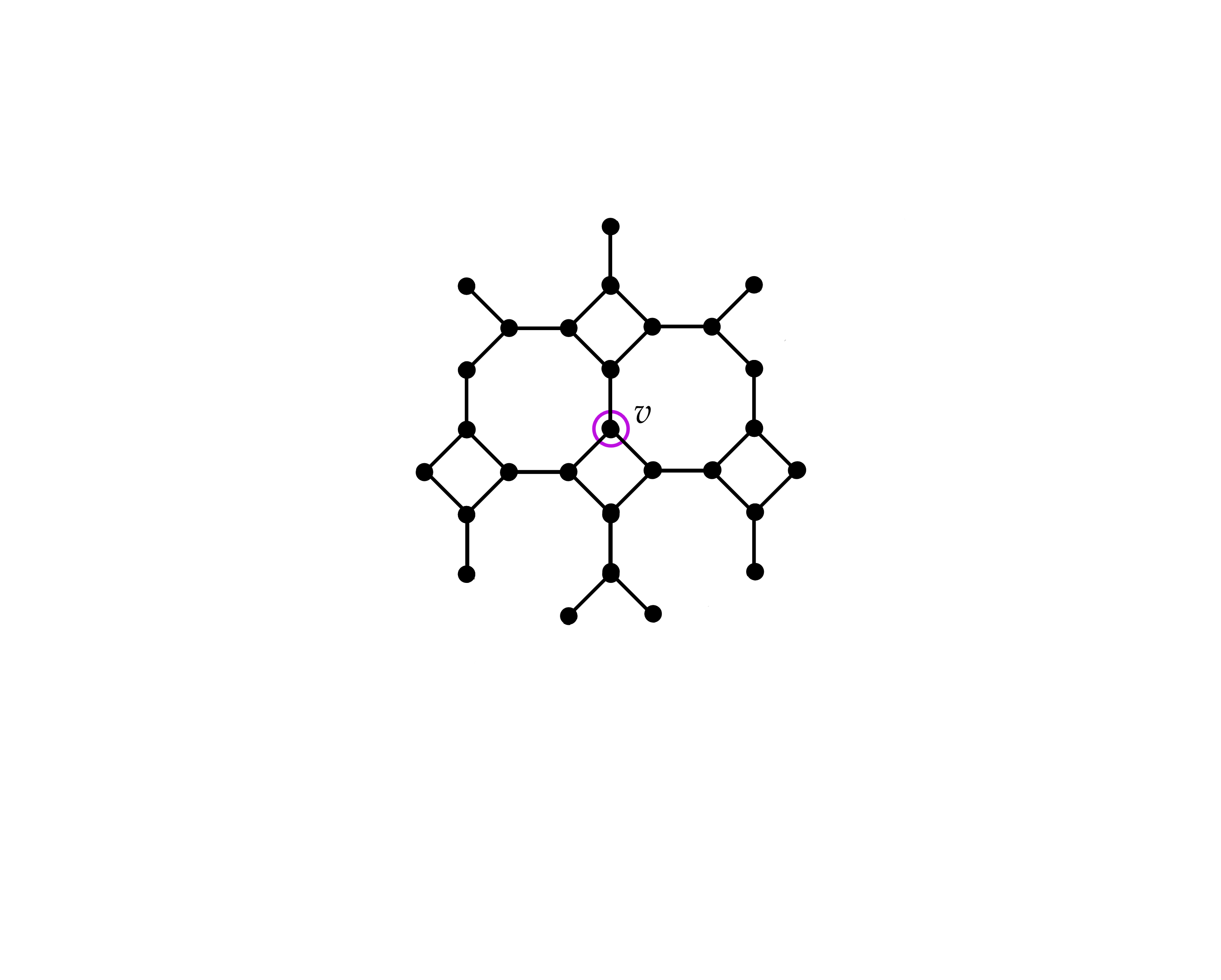}
        \caption{The subgraph $P_{v}(5)$.\linebreak}
    \end{subfigure}
    \caption{Illustrating vertices a maximum distance of $5$ away from a fixed vertex $v$, and the corresponding subgraph $P_v(5)$.}\label{fig:Pv5}
    \end{figure}

\begin{lemma}\label{lem:P(t)}
    Fix a vertex $v\in H_{\infty,\infty}$. If $t\geq 2$, then $\,\,|V(P_v(t))| = 1+\displaystyle \sum_{i=2}^{t} c(i)\,$.
    \begin{proof}
        The proof follows directly from the definition of $c(t)$. To enumerate the vertices in $P_v(t)$, we must count all vertices distance $j$ from vertex $v$, with $1 \leq j \leq t-1$. This corresponds to the sum of $c(i)$ for $2 \leq i \leq t$, or $\sum_{i=2}^t c(i)$. Lastly, we add $1$ to this sum to account for the vertex $v$ which is also in $P_v(t)$. Thus $|V(P_v(t))| = 1+ \sum_{i=2}^t c(i)$, as claimed.
    \end{proof}
\end{lemma}

    \begin{figure}
     \centering
     \begin{subfigure}[b]{0.3\textwidth}
         \centering
         \includegraphics[width=1.75in, trim=45cm 92cm 50cm 20cm, clip]{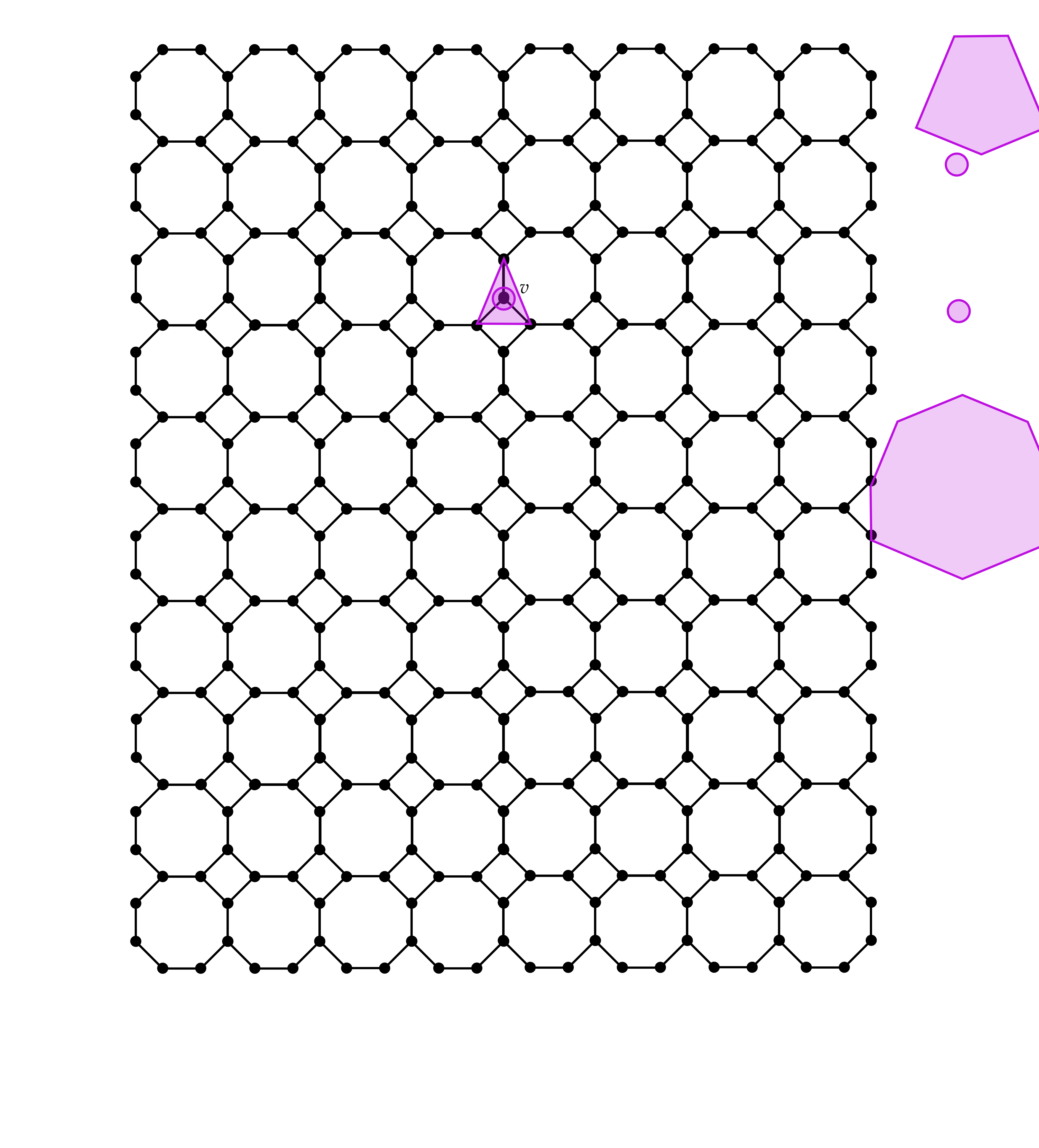}
         \caption{$c(2)=3$}
         \label{fig:c(2)}
     \end{subfigure}     
     \begin{subfigure}[b]{0.3\textwidth}
         \centering
         \includegraphics[width=1.75in, trim=45cm 92cm 50cm 20cm, clip]{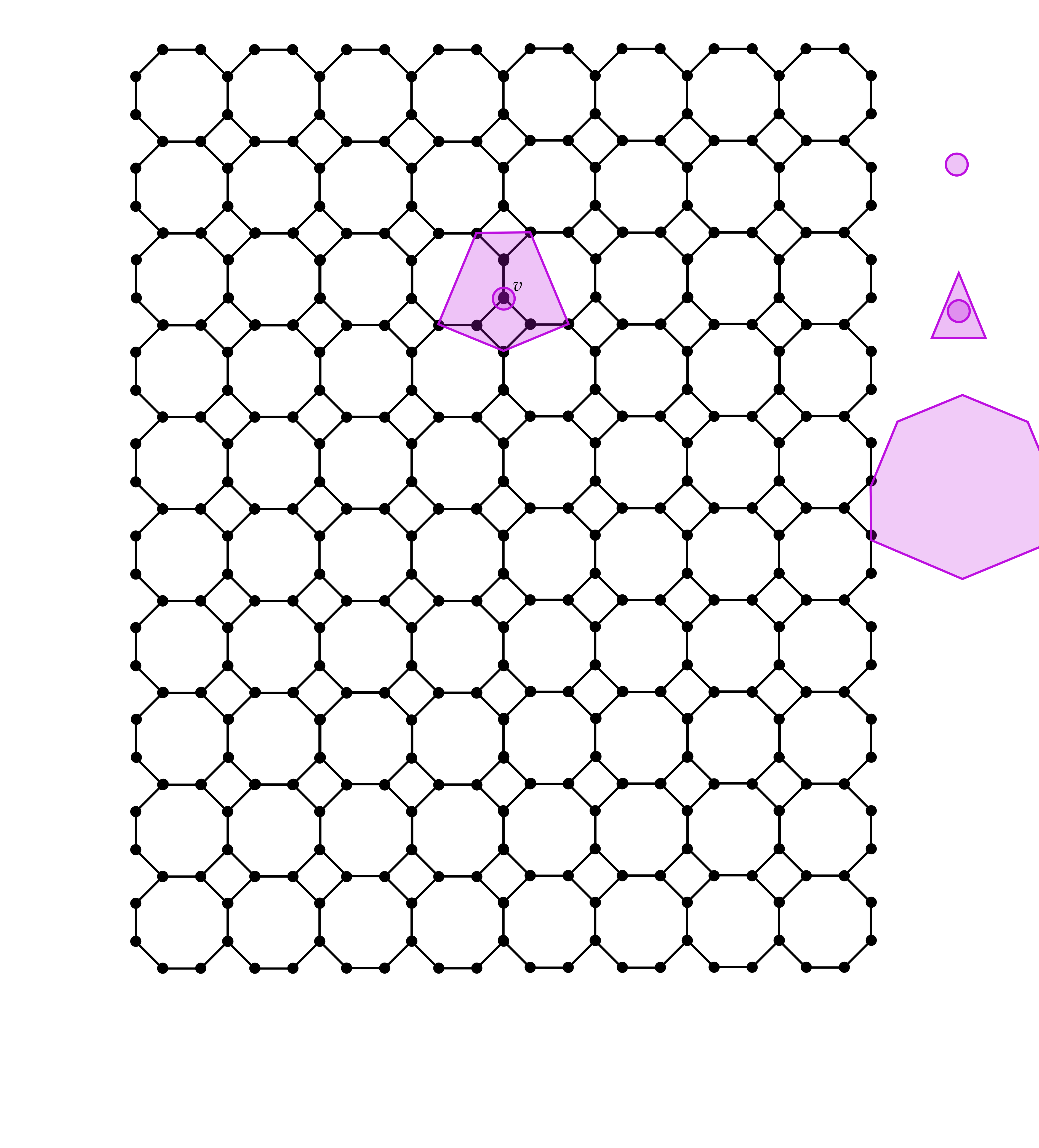}
         \caption{$c(3)=5$}
         \label{fig:c(3)}
     \end{subfigure}
     \begin{subfigure}[b]{0.3\textwidth}
         \centering
         \includegraphics[width=1.75in, trim=45cm 92cm 50cm 20cm, clip]{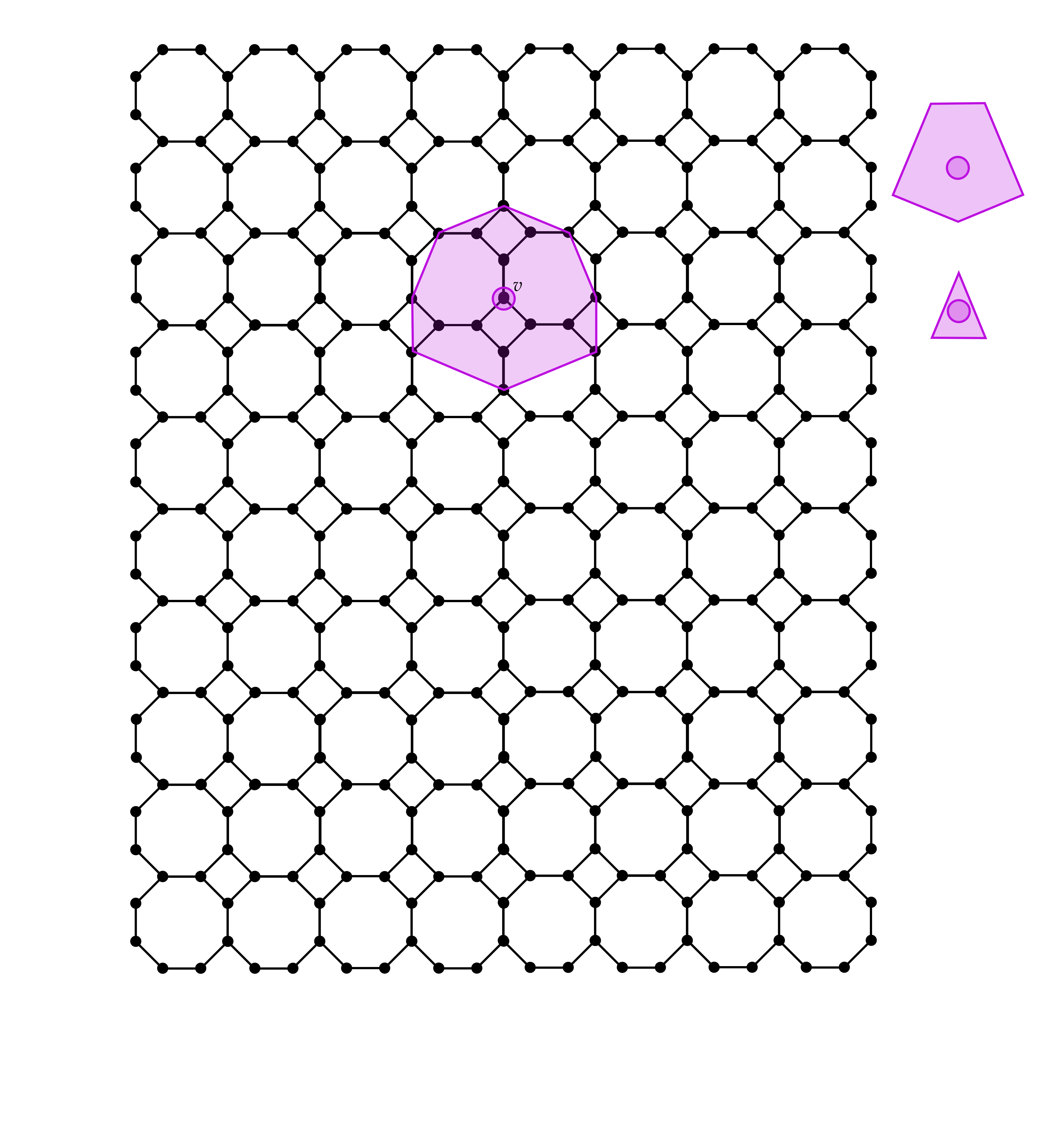}
         \caption{$c(4)=8$}
         \label{fig:c(4)}
     \end{subfigure}
     \hfill
     \hfill
        \caption{Figures to illustrate how we calculate $|P_v(4)| = 1+3+5+8=17$.}
        \label{fig:Pv(t)pf}
\end{figure}

We now give a lower bound for the distance domination number  $\gamma_{t,1}(H_{m,n})$.
\begin{theorem}\label{thm:initial_low_bd}
    If $m,n \geq 1$ and $t \geq 2$, then\[\gamma_{t,1}(H_{m,n}) \geq \displaystyle\Big\lceil \frac{2m+n(4m+2)}{|V(P_v(t))|} \Big\rceil.\]
    \end{theorem}
    \begin{proof}
        This lower bound follows from the fact that when $r=1$, each vertex $u$ in $H_{m,n}$ must have reception strength $f(u) \geq 1$, where $f(u)$ is as defined in equation \eqref{eq:f(u)}. We take the total number of vertices in $H_{m,n}$ and divide by the maximum number of vertices which receive reception from a broadcasting vertex of transmission strength $t$. (Note that we say maximum because if a broadcasting vertex is near an exterior edge of the graph, then $|V(P_v(t))|$ may be smaller.) We take the ceiling of this value to ensure an integer bound.
    \end{proof}

\section{Broadcast domination patterns for the truncated square tiling}\label{sec:densities}

In what follows we consider the infinite truncated square tiling and provide results on the proportions of vertices needed to dominate the graph. We begin by recalling the following definition.

\begin{definition}[Page 2 in \cite{DrewsHarrisRandolph}]
Given an infinite graph $G$, we refer to a $(t,r)$ broadcast dominating set $T$ as a $(t,r)$ \textit{broadcast} for $G$.    
The \textit{broadcast density} of the $(t,r)$ broadcast $T$ for $G$ is defined by the limit where we consider a finite portion of the graph $G$, denoted $H$, and compute
\[\lim_{n \to \infty} \frac{|T \cap V(H)|}{|V(H)|}.\] 
The \emph{optimal density} of a $(t,r)$ broadcast for $G$ is defined as the minimum broadcast density over all $(t,r)$ broadcasts. We denote the optimal density of a $(t,r)$ broadcast by $
\delta_{t,r}(G)$.
\end{definition}

We now establish upper bounds for the optimal density of $(t,r)$ broadcasts for $\gH$ when $(t,r)\in\{(2,1),(2,2),(3,1),(3,2),(3,3),(4,1)\}$.

\begin{theorem}\label{thm:density21}
    The optimal density of a $(2,1)$ broadcast for $\gH$ satisfies
    \[\delta_{2,1}(\gH)\leq \frac14.\]
\end{theorem}
\begin{proof}
    Fix a positive integer $n$. Then by Theorem~\ref{thm:enum_vertices}
    we know the number of vertices in $V(H_{n,n})=2n+n(4n+2)=4n^2+4n$. Then by Theorem~\ref{thm:better_upper_bd} we know $\gamma_{2,1}(H_{n,n})\leq n(n+1) + n=n^2+2n$.
    Thus \[\delta_{2,1}(\gH)\leq \lim_{n\to\infty}\frac{n^2+2n}{4n^2+4n}=\frac14.\qedhere\]
\end{proof}

Before providing our next results we need the following notation.
In order to describe $(t,r)$ broadcast dominating sets for $H_{\infty, \infty}$, we fix the truncated square lattice in $\mathbb{R}^2$ and establish a coordinate system by selecting an arbitrary $4$-cycle in $H_{\infty, \infty}$, which we refer to as the \emph{origin} of the graph. 
Note that the origin $(0,0)$ refers to a $4$-cycle of vertices, not to a vertex in the graph $H_{\infty, \infty}$, see Figure~\ref{fig:graph with origin}. 
Once we fix a $4$-cycle as the origin, we color its interior with orange.
Then we give a coordinate system $(x,y)\in\Z\times\Z$ to indicate the position of each $4$-cycle in the graph. 
We then refer to the vertices on a $4$-cycle by numbering them as on a clock naming the top vertex $0$, the right vertex $1$, the bottom vertex $2$, and the left vertex~$3$. See Figure \ref{fig:origin with vertices}.

\begin{figure}[h!]
    \centering
    \begin{subfigure}[t]{.4\textwidth}
    \centering
    
    \includegraphics[height=2in,trim=6cm .5cm .5cm .5cm,
    clip]{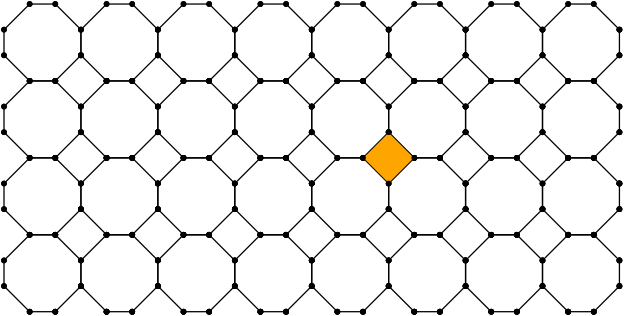}
    \caption{Identifying an origin of $\gH$.}
    \label{fig:graph with origin}
    \end{subfigure}
\qquad\qquad\qquad
    \begin{subfigure}[t]{.4\textwidth}
    \centering
    \includegraphics[height=1.75in,trim=2cm 2.25cm 2cm 2cm,clip]{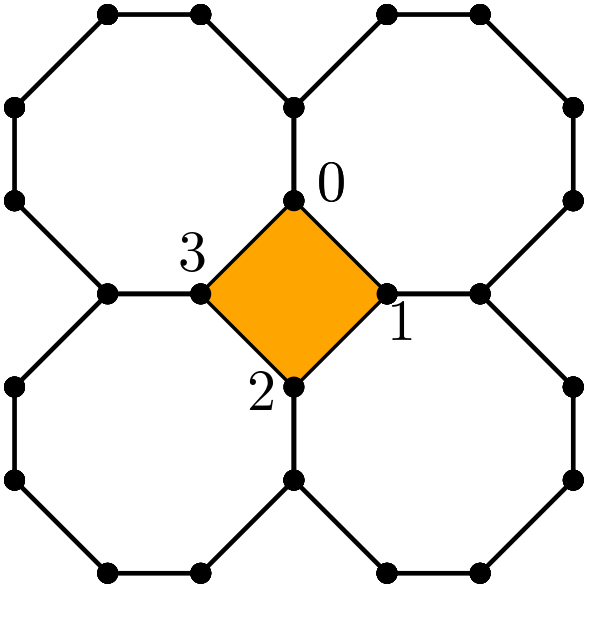}
    \caption{The vertices of the origin.}
    \label{fig:origin with vertices}
\end{subfigure}
\caption{The graph $\gH$ with a fixed $4$-cycle called the origin (highlighted in orange) and the vertices of the origin labeled $0,1,2,3$.}\label{fig:origin of H}    
\end{figure}

With the origin fixed, every vertex in $H_{\infty, \infty}$ can be described as a tuple $(a,(x,y))$, where $a\in\mathbb{Z}_4$, and $(x,y)\in\mathbb{Z}^2$.
For example, Figure~\ref{fig:examples} gives the location of the vertices $v=(1,(2,3))$ and $w=(1,(-1,1))$.

\begin{figure}[h!]
    \centering
    \includegraphics[width=3.5in,trim=8cm 4cm .25cm 2cm,clip]{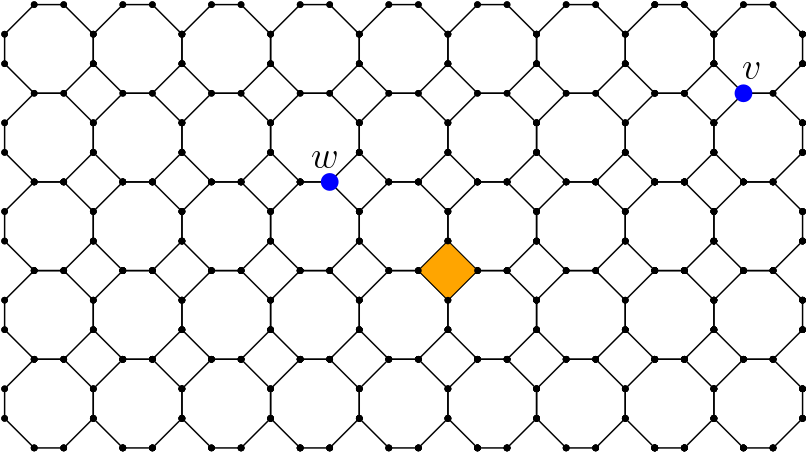}
    \caption{Position of the vertices $v=(1,(2,3))$ and $w=(1,(-1,1))$, relative to the fixed origin on $\gH$.}
    \label{fig:examples}
\end{figure}

In light of this notation, the set of vertices of $\gH$ are described by 
\[V(\gH)=\{(a,(x,y)): a\in\mathbb{Z}_4, (x,y)\in\mathbb{Z}^2\}.\]

\begin{theorem}\label{thm:density22}
The optimal density of a $(2,2)$ broadcast for $\gH$ satisfies
    \[\delta_{2,2}(H_{\infty,\infty})\leq \frac12.\]
\end{theorem}
\begin{proof}
    We construct a $(2,2)$ broadcast dominating set $T$ for $\gH$ consisting of the subset of vertices
    \[T = T_1\cup T_2\cup T_3 \cup T_4, \]
    where 
    \begin{align*}
    T_1&=\{(0, (x, x+2y)): x,y\in \Z\}, \\ 
    T_2&=\{(2, (x, x+2y)): x,y\in \Z\},  \\
    T_3&=\{(1, (x, x+2y + 1)): x,y\in \Z\}, \mbox{ and}\\
    T_4&=\{(3, (x, x+2y+1)): x,y\in \Z\}.
    \end{align*}
    We illustrate the positioning of the vertices in the set $T$ in Figure~\ref{fig:rowsfor22} and in this proof we establish that the density of this broadcast is $\frac12$. To begin we prove that $T$ is a $(2,2)$ broadcast of $\gH$.

\begin{figure}[h!]
    \centering
    \includegraphics[width=4in,trim=1.5cm 1.5cm 1.5cm 1.5cm,clip]{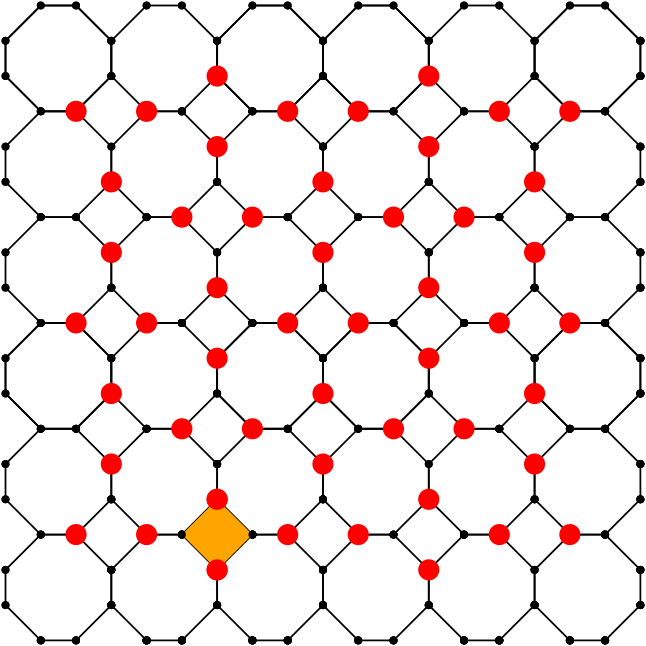}
    \caption{The position of the broadcasting vertices in the $(2,2)$ broadcast of Theorem~\ref{thm:density22}, which has density~$\frac{1}{2}$.}
    \label{fig:rowsfor22}
    
\end{figure}

Because each broadcasting vertex in $T$ receives reception $2$ from itself, we focus only on the vertices in the set $V(H_{\infty, \infty}) \setminus T$.
This set consists of the vertices in the set 
\[A_1\cup A_2,\]
where
\begin{align*}
A_1&=\{(a, (x, x+2y)) : a\in\{1,3\}, x,y \in \mathbb{Z}\}, \mbox{ and}\\
A_2&=\{(a, (x, x+2y+1)) : a\in\{0,2\}, x,y \in \mathbb{Z}\}.
\end{align*}
We now proceed via a case-by-case analysis to show that for each $1\leq i\leq 2$, if $v\in A_i$, then it receives reception of at least $2$.
\begin{figure}[h!]
    \centering
    \begin{subfigure}[t]{.4\textwidth}
    \centering

\includegraphics[width=1.35in,trim=1.5cm 1.3cm 1.5cm 1.5cm,clip]{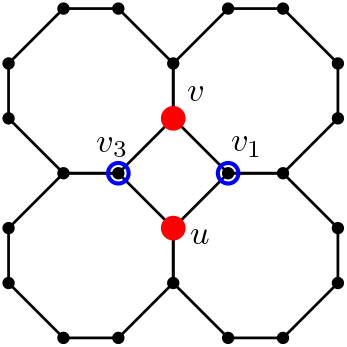}
    \caption{Let $(t,r)=(2,2)$. the vertices $v_1$ and $v_3$ each receive reception one from each of the broadcasting vertices $v$ and $u$. Hence they receive reception at least two and are dominated.}
    \label{fig:22pf1}
\end{subfigure}    
\qquad\qquad
    \begin{subfigure}[t]{.4\textwidth}
    \centering
    \includegraphics[width=1.35in,trim=2.5cm 2.5cm 2.5cm 2.5cm,clip]{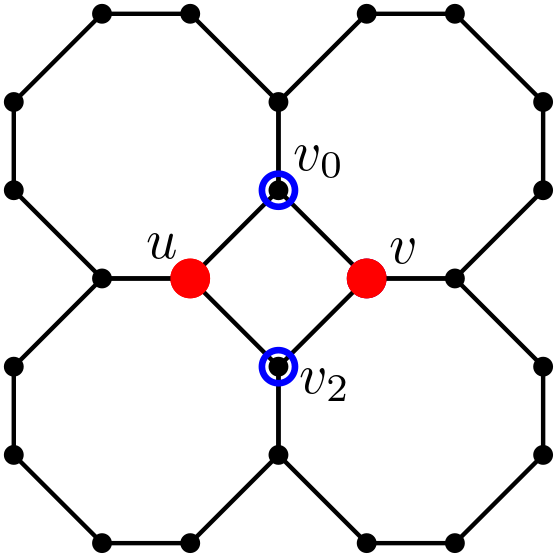}
    \caption{Let $(t,r)=(2,2)$. The vertices $v_0$ and $v_2$ each receive reception one from each of the broadcasting vertices $v$ and $u$. Hence they receive reception at least $2$ and are dominated.}
    \label{fig:22pf2}
\end{subfigure}
\caption{Illustrations for cases (1) and (2) in the proof of Theorem~\ref{thm:density22}.}
\end{figure}

\begin{enumerate}
\item Fix arbitrary integers $x,y\in\Z$. 
To begin consider 
vertices $v_1=(1,(x,x+2y))$ and $v_3 = (3,(x,x+2y))$ in the set $A_1$.
Observe that the vertices $v=(0,(x,x+2y))$ and $u=(2, (x, x+2y))$ are broadcasting vertices in $T$.
Then the vertices $v_1$ and $v_3$ lie within a distance of 1 from the broadcasting vertex $v=(0,(x,x+2y))$, and $v_1$ and $v_3$ lie within a distance of 1 from the broadcasting vertex $u=(2, (x, x+2y))$. Thus $v_1$ and $v_3$ receive reception 1 from both of the broadcasting vertices $u$ and $v$, see Figure~\ref{fig:22pf1} for an illustration.

Therefore $v_1$ and $v_3$ receive reception at least  $2$ and are dominated. As $x$ and $y$ were arbitrary, all of the vertices in the set $A_1$ are dominated.

\item Fix arbitrary integers $x,y\in\Z$. 
To begin consider 
vertices $v_0=(0,(x,x+2y+1))$ and $v_2 = (2,(x,x+2y+1))$ in the set $A_2$.
Observe that the vertices $v=(1,(x,x+2y+1))$ and $u=(3, (x, x+2y+1))$ are broadcasting vertices in $T$.

Then the vertices $v_0$ and $v_2$ lie within a distance of $1$ from the broadcasting vertex $v=(1,(x,x+2y+1))$, and $v_1$ and $v_3$ lie within a distance of $1$ from the broadcasting vertex $u=(3, (x, x+2y+1))$. Thus $v_0$ and $v_2$ receive reception $1$ from both of the broadcasting vertices $u$ and $v$, see Figure~\ref{fig:22pf2} for an illustration.

Therefore $v_0$ and $v_2$ receive reception at least $2$ and are dominated. As $x$ and $y$ were arbitrary, all of the vertices in the set $A_2$ are dominated.
\end{enumerate}
We have shown that $T$ is a $(2,2)$ broadcast for $\gH$.
We now want to determine the proportion of vertices in $\gH$ that lie in $T$.

To compute this proportion, we note that 
every $4$-cycle has $2$ broadcasting vertices, see Figure~\ref{fig:rowsfor22}. Thus, the overall proportion of broadcasting vertices in $T$ to vertices in the graph $\gH$ is given by $1/2$.
Thus $\delta_{2,2}(\gH)\leq \frac{1}{2}$, as claimed.
\end{proof}

\begin{theorem}\label{thm:density31}
The optimal density of a $(3,1)$ broadcast for $\gH$ satisfies
    \[\delta_{3,1}(H_{\infty,\infty})\leq \frac{1}{8}.\]
\end{theorem}

\begin{proof}
We construct a $(3,1)$ broadcast $T$ for $H_{\infty,\infty}$ consisting of the subset of vertices
\[T=\{(2, (x, x+4y)) : x,y \in \mathbb{Z}\} \cup \{(1, (x, x+4y+2)) : x,y \in \mathbb{Z}\}.\] 
We illustrate the positioning of the vertices in the set $T$ in Figure~\ref{fig:rowsfor31} and in this proof we establish that the density of this broadcast is $\frac18$. To begin we prove that $T$ is a $(3,1)$ broadcast for $\gH$.
\begin{figure}[h!]
    \centering
\includegraphics[width=3.5in,trim=1.5cm 1.5cm 1.5cm 1.5cm,clip]{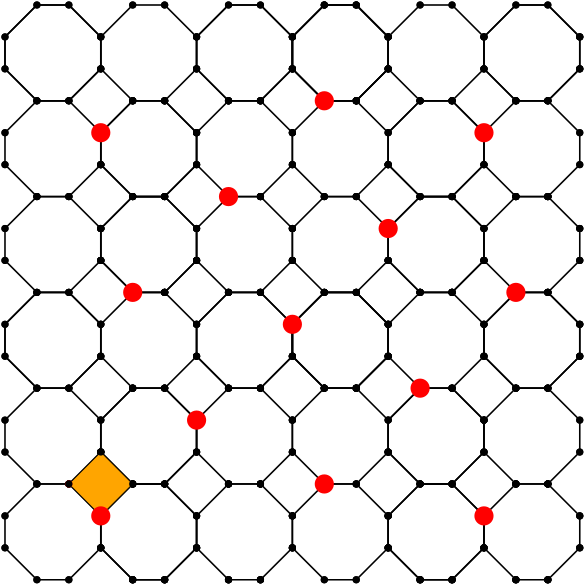}
    \caption{The position of the broadcasting vertices  in the $(3,1)$ broadcast for $\gH$ used in the proof of Theorem~\ref{thm:density31}, which has density $\frac{1}{8}$.}
    \label{fig:rowsfor31}
\end{figure}
Because each broadcasting vertex in $T$ receives reception $3$ from itself, we focus only on the vertices in the set $V(H_{\infty, \infty}) \setminus T$.
This set consists of the vertices in 
\[A_1\cup A_2\cup A_3\cup A_4,\]
where
\begin{align*}
A_1&=\{(a, (x, x+4y)) : a\in\{0,1,3\}, x,y \in \mathbb{Z}\},\\
A_2&=\{(a, (x, x+4y+2)) : a\in\{0,2,3\}, x,y \in \mathbb{Z}\},\\
A_3&= \{(a, (x, x+4y + 3)) :  a\in\{0,1,2,3\}, x,y \in \mathbb{Z}\}\mbox{, and}\\
A_4&= \{(a, (x, x+4y+1)) : a\in\{0,1,2,3\}, x,y \in \mathbb{Z} \}.
\end{align*}
We now proceed via a case-by-case analysis to show that for each $1\leq i\leq 4$, if $v\in A_i$, then it receives reception at least $1$.

\begin{figure}
    \centering
    \begin{subfigure}[t]{.4\textwidth}
    \centering
    \includegraphics[width=1.35in, trim=1.75cm 1.75cm 1.75cm 1.75cm,clip]{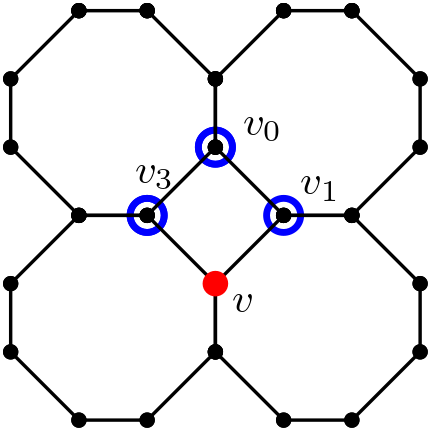}
    \caption{
    The vertices $v_0, v_1, v_3$ receive reception at least $1$ from broadcasting vertex $v$, hence they are dominated.}
    \label{fig:31pf1}
    \end{subfigure}
\qquad\qquad    
        \begin{subfigure}[t]{.4\textwidth}
    \centering
    \includegraphics[width=1.35in, trim=1.75cm 1.75cm 1.75cm 1.75cm,clip]{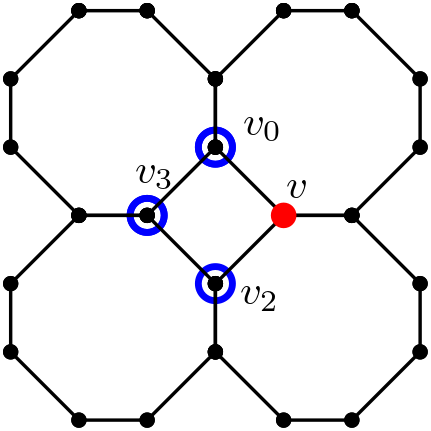}
    \caption{
    The vertices $v_0, v_2, v_3$ receive reception at least $1$ from broadcasting vertex $v$, hence they are dominated.}
    \label{fig:31pf2}
\end{subfigure}
\\
        \begin{subfigure}[t]{.4\textwidth}
    \centering
    \includegraphics[width=1.5in, trim=1.5cm 4cm 1cm 4cm,clip]{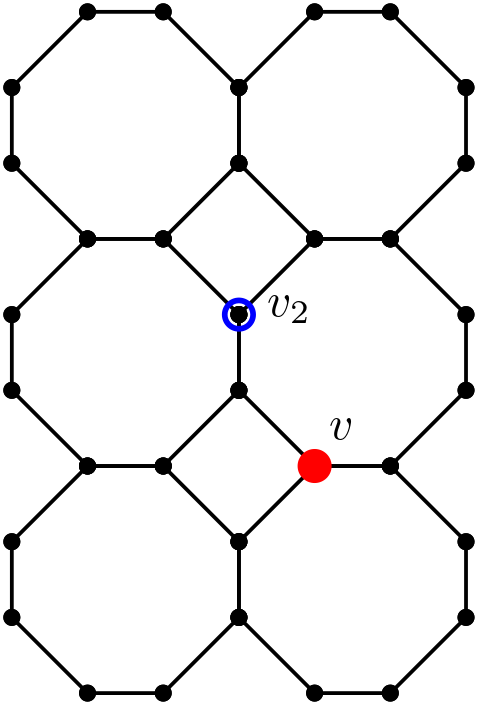}
    \caption{
    The vertex $v_2$ receives reception $1$ from broadcasting vertex $v$, hence $v_2$ is dominated.}
    \label{fig:success}
\end{subfigure}
\qquad\qquad
        \begin{subfigure}[t]{.4\textwidth}
    \centering
    \includegraphics[width=1.5in, trim=1cm 3cm 1cm 3cm,clip]{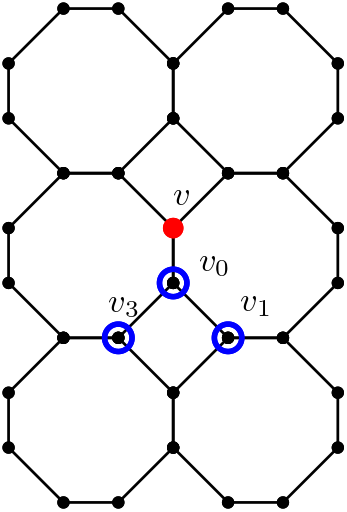}
    \caption{
    The distance from the broadcasting vertex $v$ to vertices $v_0$, $v_1$, and $v_3$ is at most $2$, hence these vertices are dominated by the broadcasting vertex~$v$.}
    \label{fig:success2}
\end{subfigure}
\\
        \begin{subfigure}[t]{.4\textwidth}
    \centering
    \includegraphics[width=2in, trim=3.25cm 3cm 3.5cm 3cm,clip]{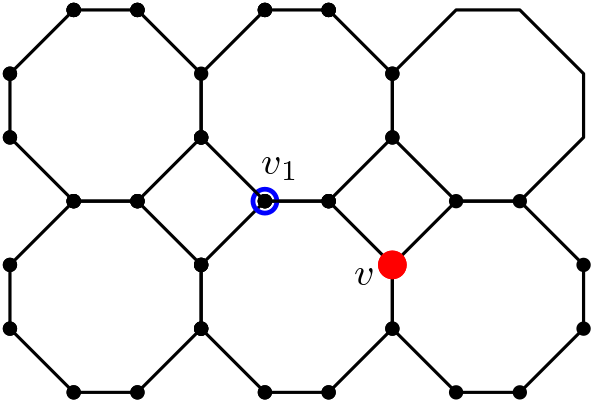}
    \caption{
    The vertex $v_1$ is a distance of $2$ from the broadcasting vertex $v$. Hence $v_1$ receives reception $1$ from $v$ and is dominated.}
    \label{fig:success3}
\end{subfigure}
\qquad
        \begin{subfigure}[t]{.4\textwidth}
    \centering
    \includegraphics[width=2in, trim=3.5cm 3.6cm 9cm 11.5cm,clip]{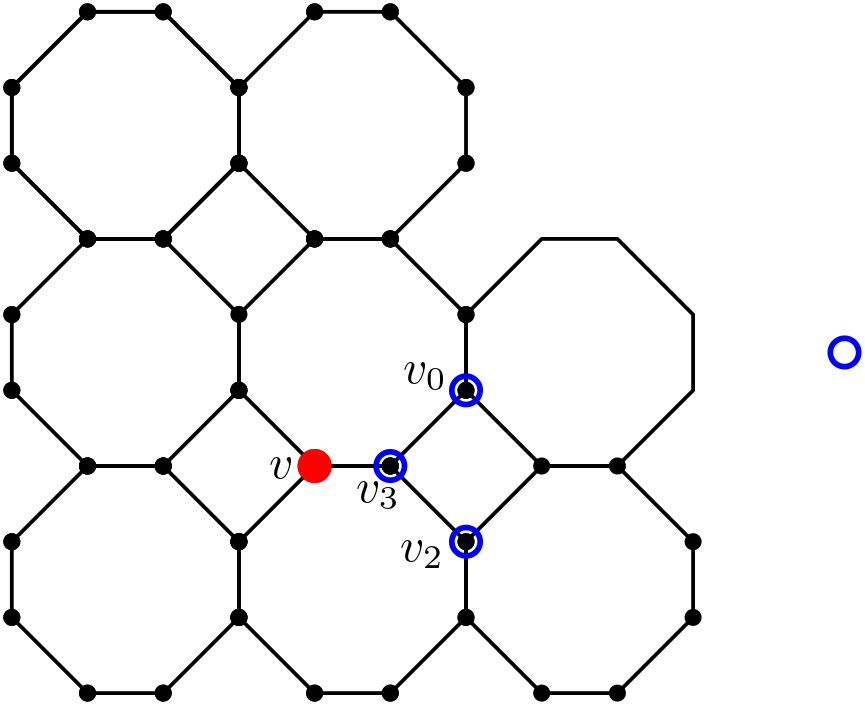}
    \caption{
    The distance from the broadcasting vertex $v$ to vertices $v_0$, $v_2$, and $v_3$ is at most $2$, hence these vertices are dominated by the broadcasting vertex $v$.}
    \label{fig:success4}
\end{subfigure}

    \caption{Illustrations with $(t,r)=(3,1)$ for cases ($1$)-($6$) in the proof of Theorem~\ref{thm:density31}.}
\end{figure}

\begin{enumerate}
    \item 
Fix arbitrary integers $x,y\in\Z$. Consider the vertices $v_0=(0,(x,x+4y))$, $v_1=(1,(x,x+4y))$, and $v_3=(3,(x,x+4y))$.
Note $v_0,v_1,v_3\in A_1=\{(a, (x, x+4y)) : a\in\{0,1,3\}, x,y \in \mathbb{Z}\}$. 
Now observe that there exists vertex $v=(2, (x, x+4y))$ in $T$, see Figure~\ref{fig:31pf1} for an illustration.

Then the vertices $v_0,v_1,v_3$ lie within a distance of  $2$ from the broadcasting vertex $v=(2,(x,x+4y))$, and thus receive reception at least 1 from that vertex. Thus, as $x,y$ were arbitrary, all of the vertices in $A_1$ are dominated.

\item Fix arbitrary integers $x,y\in\Z$. Consider the  vertices $v_0=(0,(x,x+4y+2))$, $v_2=(2,(x,x+4y+2))$, and $v_3=(3,(x,x+4y+2))$.
Note $v_0,v_2,v_3\in A_2=\{(a, (x, x+4y+2)) : a\in\{0,2,3\}, x,y \in \mathbb{Z}\}$. 
Now observe that there exists vertex $v=(1, (x, x+4y+2))$ in $T$, see Figure~\ref{fig:31pf2} for an illustration.

Then the vertices $v_0,v_2,v_3$ lie within a distance of $2$ from the broadcasting vertex $v=(1,(x,x+4y+2))$, and thus receive reception at least $1$ from that vertex. Thus, as $x,y$ were arbitrary, all of the vertices in $A_2$ are dominated.

\item Fix arbitrary integers $x,y\in\Z$. 
To begin consider 
a vertex $v_2=(2,(x,x+4y+3))$ in the set $A_3$.
Observe that the vertex $v=(1,(x,x+4y+2))$ is a broadcasting vertex in $T$.
Now notice that the distance from $v$ to $v_2$ is computed as follows: $v_2$ is the bottom vertex of a $4$-cycle, while $v$, the broadcasting vertex, is the right vertex of the $4$-cycle immediately below the $4$-cycle containing $v_2$, see Figure~\ref{fig:success}. 
Thus the distance   is exactly $2$, and $v_2$ receives signal $1$ from $v$, and hence is dominated.

We now consider the vertices $v_0=(0,(x,x+4y+3))$, $v_1=(1,(x,x+4y+3))$, and $v_3=(3,(x,x+4y+3))$. 
Consider the broadcasting vertex located at position 
$v=(2,(x,x+4y+4))=(2,(x,x+4(y+1)))$ which is in $T$.
Note that the broadcasting vertex $v$ is the bottom vertex in the $4$-cycle above the $4$-cycle that contains the vertices $v_0,v_1,v_3$, see Figure~\ref{fig:success2} for an illustration. 
Note that the distance between the broadcasting vertex $v$ and the vertices $v_0,v_1,$ and $v_3$ is at most $2$, hence these vertices receive signal $1$ from the broadcasting vertex $v$, and hence are dominated.

Thus, as $x,y$ were arbitrary,
all of the vertices in the set $A_3$ are dominated.

\item 
Fix arbitrary integers $x,y\in\Z$. 
To begin consider 
a vertex $v_1=(1,(x,x+4y+1))$ in the set $A_4$.
Observe that the vertex $v=(2,(x+1,x+1+4y))$ is a broadcasting vertex in $T$.
Now notice that the distance from $v$ to $v_1$ is computed as follows: $v_1$ is the right vertex of a $4$-cycle, while $v$, the broadcasting vertex, is the bottom vertex of the $4$-cycle immediately to the right of the $4$-cycle containing $v_1$, see Figure~\ref{fig:success3} for an illustration. 
Thus, the distance  is exactly $2$, and $v_1$ receives signal $1$ from the broadcasting vertex $v$, and hence is dominated.

We now consider the vertices $v_0=(0,(x,x+4y+1))$, $v_2=(2,(x,x+4y+1))$, and $v_3=(3,(x,x+4y+1))$. 
Consider the broadcasting vertex located at position 
$v=(1,(x-1,x+4y+1))=(1,(x-1,x-1+4y+2))$ which is in $T$.
Note that the broadcasting vertex $v$ is the right vertex in the $4$-cycle immediately to the left of the $4$-cycle that contains the vertices $v_0,v_2,v_3$, see Figure~\ref{fig:success4} for an illustration. 
Note that the distance between the broadcasting vertex $v$ and the vertices $v_0,v_2,$ and $v_3$ is at most $2$, hence these vertices receive signal $1$ from the broadcasting vertex $v$, and hence are dominated.

Thus, as $x,y$ were arbitrary,
all of the vertices in the set $A_4$ are dominated.

\end{enumerate}
Thus, $T$ is a $(3,1)$ broadcast for $\gH$.
Next we determine the proportion of vertices in $\gH$ that are in $T$.

To compute this proportion, we note that among any row of $4$-cycles containing the vertices of $\gH$, out of every two $4$-cycles, one vertex is a broadcasting vertex, see Figure~\ref{fig:rowsfor31}. Since every $4$-cycle consists of $4$ vertices, then we have that one of every eight vertices is selected to be a broadcasting vertex. 
Thus $\delta_{3,1}(\gH)\leq \frac{1}{8}$, as claimed.
\end{proof}

\begin{theorem}\label{thm:density32}
The optimal density of a $(3,2)$ broadcast for $\gH$ satisfies
    \[\delta_{3,2}(H_{\infty,\infty})\leq \frac{1}{6}.\]
\end{theorem}
\begin{proof}
We construct a $(t,r)$ broadcast $T$ for $H_{\infty,\infty}$ consisting of the subset of vertices 
\[T = T_1\cup T_2\cup T_3 \cup T_4, \]
    where 
    \begin{align*}
    T_1&=\{(3, (x, x+6y)): x,y\in \Z\}, \\ 
    T_2&=\{(1, (x, x+6y+2)): x,y\in \Z\},  \\
    T_3&=\{(0, (x, x+6y+3)): x,y\in \Z\}, \mbox{ and}\\
    T_4&=\{(2, (x, x+6y+5)): x,y\in \Z\}.
    \end{align*}

We illustrate the positioning of the vertices in the set $T$ in Figure~\ref{fig:rowsfor32} and in this proof we establish that the density of this broadcast is $\frac16$. To begin we prove that $T$ is a $(3,2)$ broadcast for $\gH$.

\begin{figure}[h!]
    \centering
\includegraphics[width=4in,trim=1.5cm 1.5cm 1.5cm 1.5cm,clip]{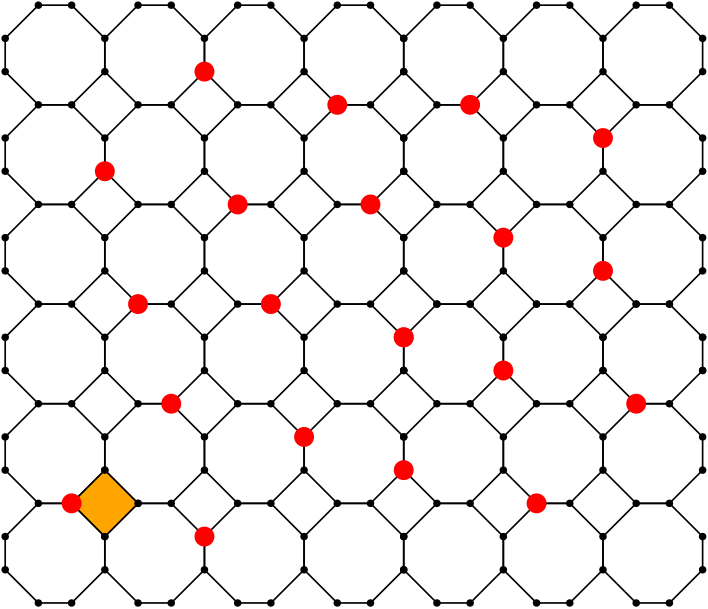}
    \caption{The position of the broadcasting vertices  in the $(3,2)$ broadcast for $\gH$ used in the proof of Theorem~\ref{thm:density32}, which has density $\frac{1}{6}$.}
    \label{fig:rowsfor32}
\end{figure}

Because each broadcasting vertex in $T$ receives reception $3$ from itself, we focus only on the vertices in the set $V(H_{\infty, \infty}) \setminus T$.
This set consists of the vertices in 
\[A_1\cup A_2\cup A_3\cup A_4 \cup A_5 \cup A_6,\]
where
\begin{align*}
A_1&=\{(a, (x, x+6y)) : a\in\{0,1,2\}, x,y \in \mathbb{Z}\},\\
A_2&=\{(a, (x, x+6y+2)) : a\in\{0,2,3\}, x,y \in \mathbb{Z}\},\\
A_3&=\{(a, (x, x+6y+3)) : a\in\{1,2,3\}, x,y \in \mathbb{Z}\},\\
A_4&=\{(a, (x, x+6y+5)) : a\in\{0,1,3\}, x,y \in \mathbb{Z}\},\\
A_5&= \{(a, (x, x+6y + 1)) :  a\in\{0,1,2,3\}, x,y \in \mathbb{Z}\}\mbox{, and}\\
A_6&= \{(a, (x, x+6y+4)) : a\in\{0,1,2,3\}, x,y \in \mathbb{Z} \}.
\end{align*}
We now proceed via a case-by-case analysis to show that for each $1\leq i\leq 6$, if $v\in A_i$, then it receives reception at least $2$.

\begin{figure}[h!]
    \centering
\begin{subfigure}[t]{.4\textwidth}
\centering
    \includegraphics[width=1.75in, trim=2.25cm 1.5cm 1.75cm 2cm,clip]{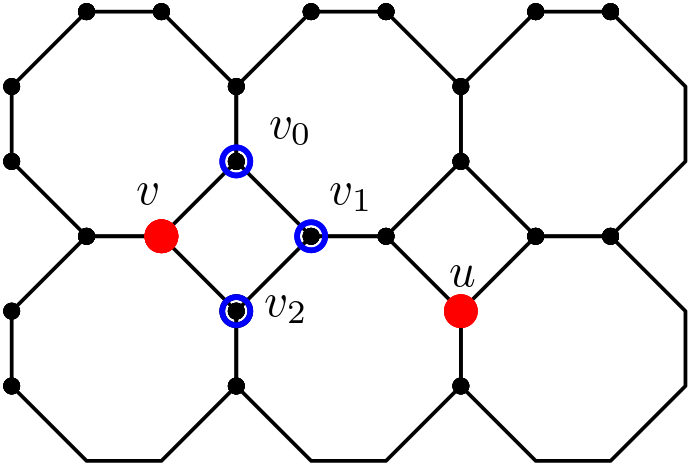}
\caption{The vertices $v_0$ and $v_2$ are dominated by the broadcasting vertex at vertex $v$, while vertex $v_1$ gets reception total equal to two, with reception one coming from each of the broadcasting vertices $u$ and $v$.}
    \label{fig:32pf1}
    \end{subfigure}
    \qquad
\begin{subfigure}[t]{.4\textwidth}
    \centering
    \includegraphics[width=1.75in,trim=2cm 2cm 2cm 2cm, clip]{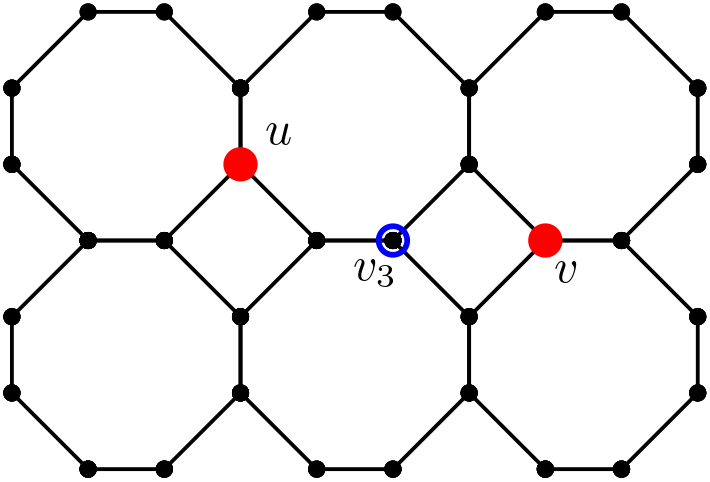}
    \caption{The vertex $v_3$ has reception total equal to two, with reception one coming from each of the broadcasting vertices $u$ and $v$.}
    \label{fig:32pf2}
\end{subfigure}\\
\begin{subfigure}[t]{.4\textwidth}
    \centering
\includegraphics[width=1in, trim=2cm 2cm 2cm 2cm, clip]{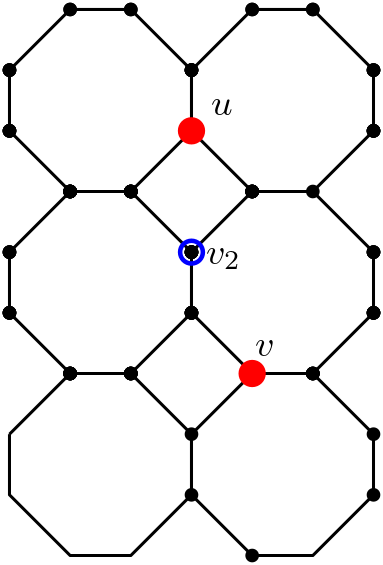}
    \caption{The vertex $v_2$ has reception total equal to two, with reception one coming from each of the broadcasting vertices $u$ and~$v$.}
    \label{fig:32pf3}
\end{subfigure}
\qquad
\begin{subfigure}[t]{.4\textwidth}
    \centering
\includegraphics[width=1in, trim=1.65cm 1.65cm 1.65cm 1.65cm, clip]{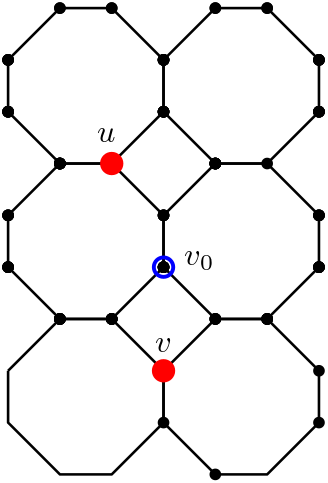}
    \caption{The vertex $v_0$ has reception total equal to two, with reception one coming from each of the broadcasting vertices $u$ and~$v$.}
    \label{fig:32pf4}
\end{subfigure}
\\
\begin{subfigure}[t]{.4\textwidth}
    \centering
\includegraphics[width=2in,trim=3cm 1.5cm 3cm 1.5cm, clip]{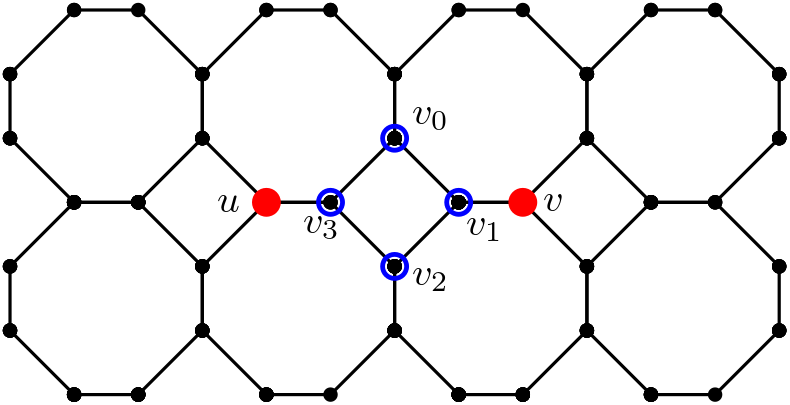}
    \caption{The vertex $v_1$ is dominated by the broadcasting vertex at vertex $v$, $v_3$ is dominated by the broadcasting vertex at vertex $u$, while vertices $v_0$ and $v_2$ get reception total equal to two, with reception one coming from each of the broadcasting vertices $u$ and $v$.}
    \label{fig:fig5for32}
\end{subfigure}
\qquad
\begin{subfigure}[t]{.4\textwidth}
    \centering
\includegraphics[width=1in,trim=.75cm 1.5cm .75cm 1.5cm, clip]
{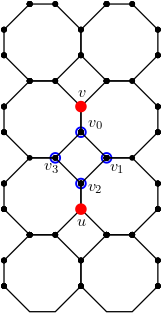}
    \caption{The vertex $v_0$ is dominated by the broadcasting vertex at vertex $v$, $v_2$ is dominated by the broadcasting vertex at vertex $u$, while vertices $v_1$ and $v_3$ get reception total equal to two, with reception one coming from each of the broadcasting vertices $u$ and $v$.}
    \label{fig:fig6for32}
\end{subfigure}

 \caption{Illustrations with $(t,r)=(3,2)$ for cases (1)-(6) in the proof of Theorem~\ref{thm:density32}.}
\end{figure}

\begin{enumerate}
    \item 
Fix arbitrary integers $x,y\in\Z$. 
First consider the vertices $v_0=(0,(x,x+6y))$ and $v_2=(2,(x,x+6y))$.
Note $v_0,v_2\in A_1=\{(a, (x, x+6y)) : a\in\{0,1,2\}, x,y \in \mathbb{Z}\}$. 
Now observe that there exists vertex $v=(3, (x, x+6y))$ in $T$. 
Then the vertices $v_0$ and $v_2$ lie a distance of $1$ from the broadcasting vertex $v=(3, (x, x+6y))$, and thus receive reception $2$ from that vertex. 
We now consider a vertex $v_1=(1,(x,x+6y))$. Note that $v_1 \in A_1$. 
Also notice that $v_1$ is a distance of $2$ from the broadcasting vertex $v=(3, (x, x+6y))$, and thus receives reception $1$ from the broadcasting vertex $v$. Hence, $v_1$ is dominated. 

Further observe that there exists vertex $u=(2,(x+1,x+6y))$ in $T$. 
This follows because if we substitute $x=x'-1$ and $y=y'+1$, we note that the location of the $4$-cycle at position \begin{align*}
    (x+1,x+6y)&=((x'-1)+1,(x'-1)+6(y'+1))=(x',x'+6y'+5).
\end{align*}
Then $(2,(x+1,x+6y))=(2,(x',x'+6y'+5))\in T_4\subseteq T$.

The distance from $u$ to $v_1$ is computed as follows: $v_1$ is the right vertex of a $4$-cycle, while $u$, the broadcasting vertex, is the bottom vertex of the $4$-cycle immediately to the right of the $4$-cycle containing $v_1$, see Figure~\ref{fig:32pf1} for an illustration. 
Thus the distance is exactly $2$, and $v_1$ receives signal $1$ from the broadcasting vertex $u$. Then $v_1$ receives reception at least $2$.
Thus, as $x,y$ were arbitrary, all of the vertices in $A_1$ are dominated.

   \item 
Fix arbitrary integers $x,y\in\Z$. 
First consider the vertices $v_0=(0,(x,x+6y+2))$ and $v_2=(2,(x,x+6y+2))$.
Note $v_0,v_2\in A_2=\{(a, (x, x+6y+2)) : a\in\{0,2,3\}, x,y \in \mathbb{Z}\}$. 
Now observe that there exists vertex $v=(1, (x, x+6y+2))$ in $T$. 
Then the vertices $v_0$ and $v_2$ lie a distance of $1$ from the broadcasting vertex $v=(1, (x, x+6y+2))$, and thus receive reception $2$ from that vertex. 
We now consider the vertex $v_3=(3,(x,x+6y+2))$. Note that $v_3 \in A_2$. 
Also notice that $v_3$ is a distance of $2$ from the broadcasting vertex $v=(1, (x, x+6y+2))$ (which is in the same $4$-cycle), and thus receives reception $1$ from the broadcasting vertex $v$. Hence $v_3$ is dominated. 

Further observe that the vertex $u=(0,(x-1,x+6y+2))$ is in $T$ since substituting $x=x'+1$ and $y=y'$, we have that 
\[
    (x-1,x+6y+2)=((x'+1)-1,(x'+1)+6y'+2)=(x',x'+6y'+3).
\]
Then $(0,(x-1,x+6y+2))=(0,(x',x'+6y'+3))\in T_3\subseteq T$.

The distance from $u=(0,(x-1,x+6y+2))$ to $v_3$ is computed as follows: $v_3$ is the left vertex of a $4$-cycle, while $u$, the broadcasting vertex, is the top vertex of the $4$-cycle immediately to the left of the $4$-cycle containing $v_3$, see Figure~\ref{fig:32pf2} for an illustration. 
Thus the distance is exactly $2$, and $v_3$ receives signal $1$ from the broadcasting vertex $u$. Then $v_3$ receives reception $1$ from the broadcasting vertex $v$ and reception $1$ from the broadcasting vertex $u$. As reception is additive, $v_3$ receives reception at least $2$, and is dominated.

Thus, as $x,y$ were arbitrary, all of the vertices in $A_2$ are dominated.

   \item 
Fix arbitrary integers $x,y\in\Z$. 
First consider the vertices $v_1=(1,(x,x+6y+3))$ and $v_3=(3,(x,x+6y+3))$.
Note $v_1,v_3\in A_3=\{(a, (x, x+6y+3)) : a\in\{1,2,3\}, x,y \in \mathbb{Z}\}$. 
Now observe that there exists vertex $v=(0, (x, x+6y+3))$ in $T$. 
Then the vertices $v_1$ and $v_3$ lie a distance of $1$ from the broadcasting vertex $v=(0, (x, x+6y+3))$, and thus receive reception $2$ from that vertex. 
We now consider the vertex $v_2=(3,(x,x+6y+2))$. Note that $v_2 \in A_3$. 
Also notice that $v_2$ is a distance of $2$ from the broadcasting vertex $v=(0, (x, x+6y+3))$ (which is in the same $4$-cycle), and thus receives reception $1$ from the broadcasting vertex $v$. 

Further observe that the vertex $u=(1,(x,x+6y+3-1))$ is in $T$ since substituting $x=x'$ and $y=y'+1$, we have that 
\begin{align*}
    (x,x+6y+3-1)&=(x',x'+6(y'+1)+3-1)
    =(x', x'+6(y'+1)+2).
\end{align*}
Then $(1,(x,x+6y+3-1))=(1,(x', x'+6(y'+1)+2)\in T_2\subseteq T$.

The distance from $u=(1,(x,x+6y+2))$ to $v_2$ is computed as follows: $v_2$ is the bottom vertex of a $4$-cycle, while $u$, the broadcasting vertex, is the right vertex of the $4$-cycle immediately below the $4$-cycle containing $v_2$, see Figure~\ref{fig:32pf3} for an illustration. 
Thus the distance is exactly $2$, and $v_2$ receives signal $1$ from the broadcasting vertex $u$. 
Then $v_2$ receives reception $1$ from the broadcasting vertex $v$ and reception $1$ from the broadcasting vertex $u$. As reception is additive, $v_2$ receives reception at least $2$, and is dominated.

Thus, as $x,y$ were arbitrary, all of the vertices in $A_3$ are dominated.

   \item 
Fix arbitrary integers $x,y\in\Z$. 
First consider the vertices $v_1=(1,(x,x+6y+5))$ and $v_3=(3,(x,x+6y+5))$.
Note $v_1,v_3\in A_4=\{(a, (x, x+6y+5)) : a\in\{0,1,3\}, x,y \in \mathbb{Z}\}$. 
Now observe that there exists vertex $v=(2, (x, x+6y+5))$ in $T$. 
Then the vertices $v_1$ and $v_3$ lie a distance of $1$ from the broadcasting vertex $v=(2, (x, x+6y+5))$, and thus receive reception $2$ from that vertex. 
We now consider the vertex $v_0=(3,(x,x+6y+2))$. Note that $v_0 \in A_4$. 
Also notice that $v_0$ is a distance of $2$ from the broadcasting vertex $v=(2, (x, x+6y+5))$ (which is in the same $4$-cycle), and thus receives reception $1$ from the broadcasting vertex $v$. 

Further observe that the vertex $u=(3,(x,x+6y+5+1))$ is in $T$ since substituting $x=x'$ and $y=y'-1$, we have that 
\begin{align*}
    (x,x+6y+5+1)&=(x',x'+6(y'-1)+5+1)
    =(x', x'+6y').
\end{align*}
Then $(3,(x,x+6y+6))=(3,(x', x'+6y'))\in T_1\subseteq T$.

The distance from $u=(3,(x,x+6y+6))$ to $v_0$ is computed as follows: $v_0$ is the top vertex of a $4$-cycle, while $u$, the broadcasting vertex, is the left vertex of the $4$-cycle immediately above the $4$-cycle containing $v_0$, see Figure~\ref{fig:32pf4} for an illustration. 
Thus, the distance is exactly $2$, and $v_0$ receives signal one from the broadcasting vertex~$u$. 

Then $v_0$ receives reception $1$ from the broadcasting vertex $v$ and reception $1$ from the broadcasting vertex $u$. As reception is additive, $v_0$ receives reception at least $2$, and is dominated.

Thus, as $x,y$ were arbitrary, all of the vertices in $A_4$ are dominated.

   \item 
Fix arbitrary integers $x,y\in\Z$. 
First consider the vertex $v_1=(1,(x,x+6y+1))$.
Note $v_1\in A_5=\{(a, (x, x+6y+1)) : a\in\{0,1,2,3\}, x,y \in \mathbb{Z}\}$. 
Now observe that the vertex $v=(3, (x+1,x+6y+1))$ is in $T$, since substituting $x=x'-1$ and $y=y'$, we have that 
\begin{align*}
    (x+1,x+6y+1)&=((x'-1)+1,(x'-1)+6y'+1)=(x', x'+6y').
\end{align*}
Then $(3,(x+1,x+6y+1))=(3,(x', x'+6y'))\in T_1\subseteq T$.
Notice $v_1$ is the right vertex of a $4$-cycle, and $v$ is the left vertex of the $4$-cycle immediately to the right, see Figure~\ref{fig:fig5for32} for an illustration. Then the vertex $v_1$ is a distance of $1$ from the broadcasting vertex $v=(3, (x+1,x+6y+1))$, and thus receives reception $2$ from that vertex.

We now consider the vertex $v_3=(3,(x,x+6y+1))$.
Note $v_3\in A_5=\{(a, (x, x+6y+1)) : a\in\{0,1,2,3\}, x,y \in \mathbb{Z}\}$. 
Now observe that the vertex $u=(1, (x-1,x+6y+1))$ is in $T$, since substituting $x=x'+1$ and $y=y'$, we have that 
\begin{align*}
    (x-1,x+6y+1)&=((x'+1)-1,(x'+1)+6y'+1)=(x', x'+6y'+2).
\end{align*}
Then $(3,(x-1,x+6y+1))=(3,(x', x'+6y'+2))\in T_2\subseteq T$.
Then the vertex $v_3$ is a distance of $1$ from the broadcasting vertex $u=(1, (x-1,x+6y+1))$, and thus receives reception $2$ from that vertex. 

We now consider the vertex $v_0=(0,(x,x+6y+1))$. Note that $v_0 \in A_5$. 
Also notice that $v_0$ is a distance of $2$ from the broadcasting vertex $v=(3, (x+1,x+6y+1))$, since $v_0$ is the top vertex in a $4$-cycle and $v$ is the left vertex of the $4$-cycle immediately to the right of the $4$-cycle containing $v_0$, see Figure~\ref{fig:fig5for32} for an illustration. 
Thus $v_0$ receives reception $1$ from the broadcasting vertex $v$.
Note additionally that $v_0$ is a distance of $2$ from the broadcasting vertex $u=(1, (x-1,x+6y+1))$, and $u$ is the right vertex in the $4$-cycle immediately to the left of the $4$-cycle containing $v_0$, see again Figure~\ref{fig:fig5for32} for an illustration. 
Then $v_0$ receives reception $1$ from the broadcasting vertex $v$ and reception $1$ from the broadcasting vertex $u$. As reception is additive, $v_0$ receives reception at least $2$, and is dominated.

Similarly, consider the vertex $v_2=(2,(x,x+6y+1))$. Note $v_2 \in A_5$. 
Also note that $v_2$ is a distance of $2$ from the broadcasting vertex $v=(3, (x+1,x+6y+1))$, since $v_2$ is the bottom vertex in a $4$-cycle and $v$ is the left vertex of the $4$-cycle immediately to the right of the $4$-cycle containing $v_2$, see Figure~\ref{fig:fig5for32} for an illustration. 
Thus $v_2$ receives reception $1$ from the broadcasting vertex $v$.
Note additionally that $v_2$ is a distance of $2$ from the broadcasting vertex $u=(1, (x-1,x+6y+1))$, and $u$ is the right vertex in the $4$-cycle immediately to the left of the $4$-cycle containing $v_2$,  see gain Figure~\ref{fig:fig5for32} for an illustration. 
Then $v_2$ receives reception $1$ from the broadcasting vertex $v$ and reception $1$ from the broadcasting vertex $u$. As reception is additive, $v_2$ receives reception at least $2$, and is dominated.

Thus, as $x,y$ were arbitrary, all of the vertices in $A_5$ are dominated.

\item
Fix arbitrary integers $x,y\in\Z$. 
First consider the vertex $v_0=(0,(x,x+6y+4))$.
Note $v_0\in A_6=\{(a, (x, x+6y+4)) : a\in\{0,1,2,3\}, x,y \in \mathbb{Z}\}$. 
Now observe that the vertex $v=(2, (x,x+6y+4+1))$ is in $T$, since substituting $x=x'$ and $y=y'-1$, we have that 
\begin{align*}
    (x,x+6y+4+1)&=(x',x'+6(y'-1)+4+1)=(x',x'+6(y'-1)+5).
\end{align*}
Then  $(2, (x,x+6y+5))=(2,(x',x'+6(y'-1)+5))\in T_4\subseteq T$.
Notice $v_0$ is the top vertex of a $4$-cycle, and $v$ is the bottom vertex of the $4$-cycle immediately above, see Figure~\ref{fig:fig6for32} for an illustration. 
Then the vertex $v_0$ is a distance of $1$ from the broadcasting vertex $v=(2, (x,x+6y+2))$, and thus receives reception $2$ from that vertex. 

We now consider the vertex $v_2=(2,(x,x+6y+4))$.
Note $v_2\in A_6=\{(a, (x, x+6y+4)) : a\in\{0,1,2,3\}, x,y \in \mathbb{Z}\}$. 
Now observe that the vertex $u=(0, (x,x+6y+4-1))$ is in $T$, since substituting $x=x'$ and $y=y'+1$, we have that 
\begin{align*}
    (x,x+6y+4-1)&=(x',x'+6(y'+1)+4-1)=(x',x'+6(y'+1)+3).
\end{align*}
Then $(3,(x,x+6y+4-1))=(3,(x',x'+6(y'+1)+3))\in T_3\subseteq T$.
Notice $v_2$ is the bottom vertex of a $4$-cycle, and $u$ is the top vertex of the $4$-cycle immediately below, see Figure~\ref{fig:fig6for32} for an illustration. 
Then the vertex $v_2$ is a distance of $1$ from the broadcasting vertex $u=(0, (x,x+6y+3))$, and thus receives reception $2$ from that vertex. 

We now consider the vertex $v_1=(1,(x,x+6y+4))$. Note that $v_1 \in A_6$. 
Also notice that $v_1$ is a distance of $2$ from the broadcasting vertex $v=(2, (x,x+6y+2))$, since $v_1$ is the right vertex in a $4$-cycle and $v$ is the bottom vertex of the $4$-cycle immediately above the $4$-cycle containing $v_1$, see Figure~\ref{fig:fig6for32} for an illustration. 
Thus $v_1$ receives reception $1$ from the broadcasting vertex $v$.
Note additionally that $v_1$ is a distance of $2$ from the broadcasting vertex $u=(0, (x,x+6y+3))$, since $u$ is the top vertex in the $4$-cycle immediately below the $4$-cycle containing $v_1$, see again Figure~\ref{fig:fig6for32} for an illustration. 
Then $v_1$ receives reception $1$ from the broadcasting vertex $v$ and reception $1$ from the broadcasting vertex $u$. As reception is additive, $v_1$ receives reception at least $2$, and is dominated.

Similarly, consider the vertex $v_3=(3,(x,x+6y+4))$. Note $v_3 \in A_6$. 
Also note that $v_3$ is a distance of $2$ from the broadcasting vertex $v=(2, (x,x+6y+2))$, since $v_3$ is the left vertex in a $4$-cycle and $v$ is the bottom vertex of the $4$-cycle immediately above the $4$-cycle containing $v_1$, see Figure~\ref{fig:fig6for32} for an illustration. 
Thus $v_3$ receives reception $1$ from the broadcasting vertex $v$.
Note additionally that $v_3$ is a distance of $2$ from the broadcasting vertex $u=(0, (x,x+6y+3))$, since $u$ is the top vertex in the $4$-cycle immediately below the $4$-cycle containing $v_3$, see again Figure~\ref{fig:fig6for32} for an illustration. 
Then $v_3$ receives reception $1$ from the broadcasting vertex $v$ and reception $1$ from the broadcasting vertex $u$. As reception is additive, $v_3$ receives reception at least $2$, and is dominated.

Thus, as $x,y$ were arbitrary, all of the vertices in $A_6$ are dominated.
\end{enumerate}
Thus, $T$ is a $(3,2)$ broadcast for $\gH$.
Next we determine the proportion of vertices in $\gH$ that are in $T$.

To compute this proportion, we note that among any row of $4$-cycles containing the vertices of $\gH$, every three $4$-cycles, there are two broadcasting vertices, see Figure~\ref{fig:rowsfor32} for an illustration. 
Since every $4$-cycle consists of $4$ vertices, we then have that among every $12$ vertices, two are selected to be broadcasting vertices. 
Thus $\delta_{3,1}(\gH)\leq \frac{2}{12}=\frac{1}{6}$, as claimed.
\end{proof}

\begin{theorem}\label{thm:density33}
The optimal density of a $(3,3)$ broadcast for $\gH$ satisfies
    \[\delta_{3,3}(H_{\infty,\infty})\leq \frac14.\]
\end{theorem}
\begin{proof}

We construct a $(3,3)$ broadcast $T$ for $H_{\infty,\infty}$ consisting of the subset of vertices 
\[T = \{(0,(x,y)):x,y\in Z\}.\]
    In Figure~\ref{fig:Delta33}, we provide a construction of a $(3,3)$ broadcast $T$ for $H_{\infty,\infty}$.

Because each broadcasting vertex in $T$ receives reception $3$ from itself, we focus only on the vertices in the set $V(H_{\infty, \infty}) \setminus T$.
This set consists of the vertices in 
\[A_1\cup A_2\cup A_3\]
where
\begin{align*}
A_1&=\{(1, (x, y)) :x,y \in \mathbb{Z}\},\\
A_2&=\{(2, (x, y)) :x,y \in \mathbb{Z}\},\mbox{ and}\\
A_3&=\{(3, (x, y)) :x,y \in \mathbb{Z}\}.\\
\end{align*}
We now proceed via a case-by-case analysis to show that for each $1\leq i\leq 3$, if $v\in A_i$, then it receives reception at least $3$.

\begin{figure}[h!]
    \centering
    \begin{subfigure}[t]{.3\textwidth}
    \centering
\includegraphics[width=1.75in, trim=2cm 2.5cm 2cm 2.5cm,clip]{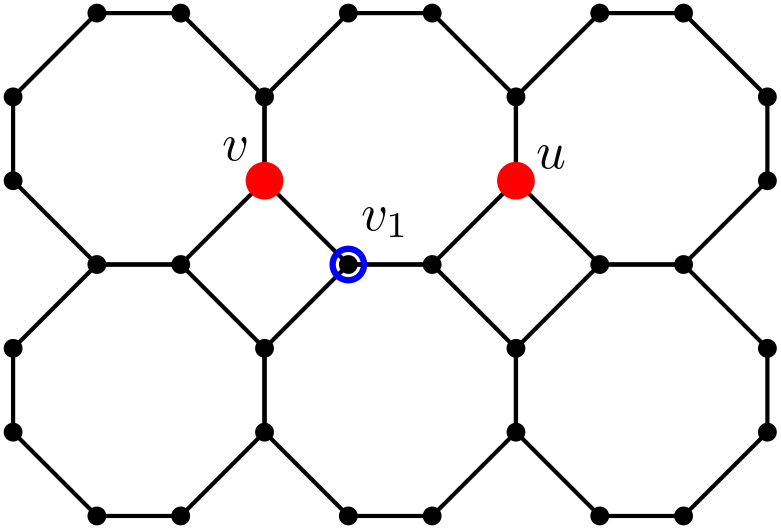}
    \caption{The vertex $v_1$ receives reception $2$ from the broadcasting vertex $v$ and reception $1$ from the broadcasting vertex $u$, hence it is dominated.}
    \label{fig:fig1for33}
\end{subfigure}    
\quad
    \begin{subfigure}[t]{.3\textwidth}
    \centering
\includegraphics[width=1in, trim=1.5cm 3cm 1.5cm 1.5cm,clip]{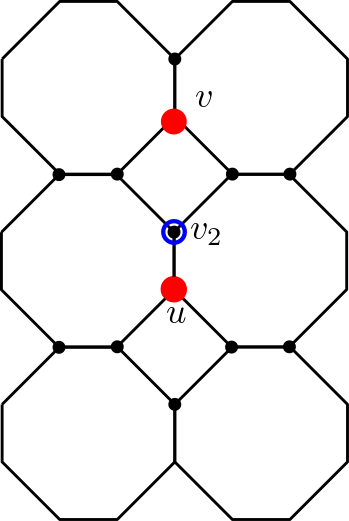}
    \caption{The vertex $v_2$ receives reception $1$ from the broadcasting vertex $v$ and reception $2$ from the broadcasting vertex $u$, hence it is dominated.}
    \label{fig:fig2for33}
\end{subfigure}    
\quad
\begin{subfigure}[t]{.3\textwidth}
    \centering
\includegraphics[width=1.75in, trim=.5cm .65cm .5cm .65cm,clip]{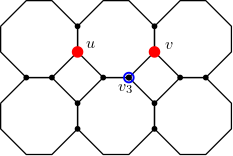}
    \caption{The vertex $v_3$ receives reception $2$ from the broadcasting vertex $v$ and reception $1$ from the broadcasting vertex $u$, hence it is dominated.}
    \label{fig3for33}
\end{subfigure}
\caption{Illustrations with $(t,r)=(3,3)$ for cases (1)-(3) in the proof of Theorem~\ref{thm:density33}.}
\end{figure}

\begin{enumerate}
    \item Fix arbitrary integers $x,y\in\Z$. Consider the vertex $v_1=(1,(x,y))\in A_1$.
The vertices $v=(0, (x, y))$ and $u=(0,(x+1,y))$ are in $T$, and vertex $v_1$ receives reception $2$ from $v$ and reception $1$ from $u$, see Figure~\ref{fig:fig1for33} for an illustration. Thus vertex $v_1$ receives reception at least $3$, as needed.  
As $x,y$ where arbitrary, all of the vertices in $A_1$ are dominated. 

    \item Fix arbitrary integers $x,y\in\Z$. Consider the vertex $v_2=(2,(x,y))\in A_2$.
The vertices $v=(0, (x, y))$ and $u=(0,(x,y-1))$ are in $T$, and vertex $v_2$ receives reception $2$ from $u$ and reception $1$ from $v$, see Figure~\ref{fig:fig2for33} for an illustration. Thus, vertex $v_1$ receives reception at least $3$, as needed.
As $x,y$ where arbitrary, all of the vertices in $A_2$ are dominated.

\item Fix arbitrary integers $x,y\in\Z$. Consider the vertex $v_3=(3,(x,y))\in A_3$.
The vertices $v=(0, (x, y))$ and $u=(0,(x-1,y))$ are in $T$, and vertex $v_3$ receives reception $2$ from $v$ and reception $1$ from $u$, see Figure~\ref{fig3for33} for an illustration. Thus, vertex $v_3$ receives reception at least $3$, as needed.
As $x,y$ where arbitrary, all of the vertices in $A_3$ are dominated. 
\end{enumerate}

Thus, $T$ is a $(3,3)$ broadcast for $\gH$.
Next we determine the proportion of vertices in $\gH$ that are in $T$.

To compute this proportion, we note that among any row of $4$-cycles containing the vertices of $\gH$, one vertex in every $4$ cycle is a broadcasting vertex, see Figure~\ref{fig:rowsfor33} for an illustration. 
Since every $4$-cycle consists of $4$ vertices, then $1$ out of every $4$ vertices in $\gH$ is a broadcasting vertex.
Thus $\delta_{3,3}(\gH)\leq \frac{1}{4}$, as claimed.
\end{proof}

\begin{figure}[h!]
    \centering
\includegraphics[width=4in,trim=1cm 1cm 1cm 1cm,clip]{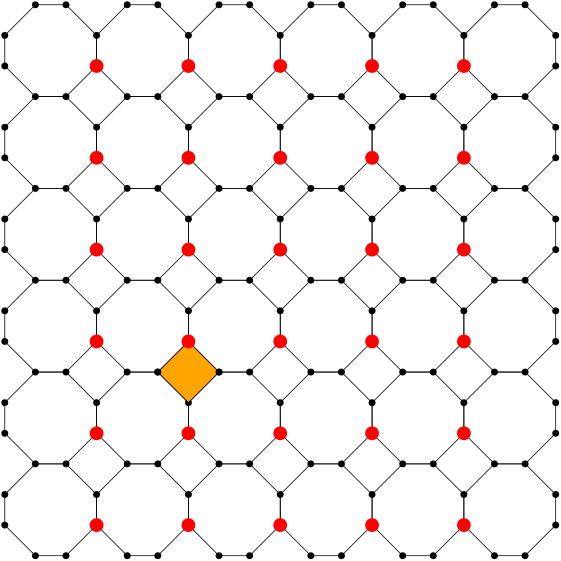}
    \caption{The position of the broadcasting vertices in the $(3,3)$ broadcast for $\gH$ used in the proof of Theorem~\ref{thm:density33}, which has density $\frac14$.}
    \label{fig:rowsfor33}
\end{figure}

\begin{remark}
    Under the $(3,3)$ broadcast presented in Theorem~\ref{thm:density33} every non-broadcasting  vertex receives signal $4$. A similar observation can be made about the $(2,2)$ broadcast in Theorem~\ref{thm:density22}, where every  non-broadcasting vertex receives signal $3$. In \Cref{dobetter} we ask if more efficient broadcasts exist.
\end{remark}

\begin{theorem}
    \label{thm:density41}
The optimal density of a $(4,1)$ broadcast for $\gH$ satisfies
    \[\delta_{4,1}(H_{\infty,\infty})\leq \frac{1}{12}.\]
\end{theorem}
\begin{proof}
Let 
\[T=\{(0,(x,3x+6y)):x,y\in \Z\}
\cup\{(2,(x+1,3x+6y+2)):x,y\in \Z\}.
\]
\begin{figure}[h!]
    \centering
\includegraphics[width=4in,trim=1cm 3cm 1cm 1cm,clip]{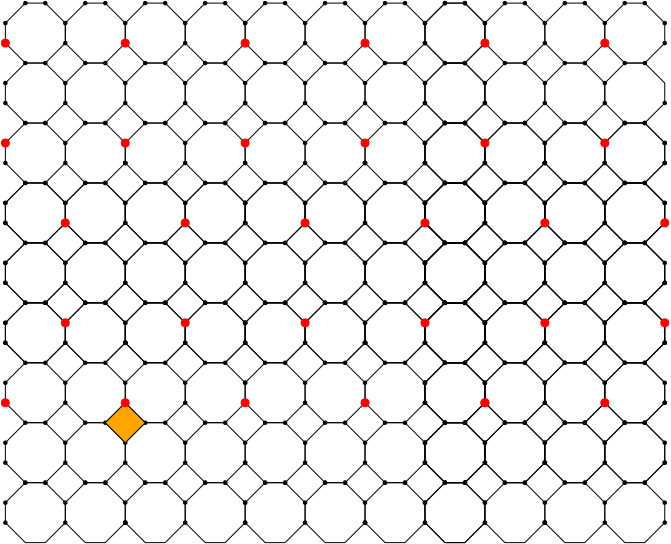}
    \caption{Illustration of the position of the broadcasting vertices in the $(4,1)$ broadcast of Theorem~\ref{thm:density41}, which has density $\frac{1}{12}$.}
    \label{fig:rowsfor41}
\end{figure}
We argue that all vertices that are not in $T$ receive enough reception from at most $2$ vertices in $T$. 

To begin we note that in the case that a $4$-cycle contains a broadcasting vertex $v=(a,(x,y))\in T$ (with $a\in\{0,2\}$), if $w$ is another vertex in that same $4$-cycle containing $v$, then the maximum distance between $v$ and $w$ is $2$. Thus, the broadcasting vertex $v$ would give at least reception $2$ to any other vertex in the same $4$-cycle. Thus, non-broadcasting vertices on a $4$-cycle containing a broadcasting vertex are all dominated. 
Thus it suffices to consider vertices in $\gH$ which are not in a $4$-cycle containing a vertex $v\in T$.
This set is given by
\[A_1\cup A_2\cup A_3\cup A_4\cup A_5\cup A_6,\]
where 
\begin{align*}
A_1&=\{(a,(x,y)):a\in\{0,1,2,3\}, x,y\in \Z, \mbox{ with } y\equiv 1\mod 6 \}\\
A_2&=\{(a,(x,y)):a\in\{0,1,2,3\}, x,y\in \Z, \mbox{ with }  y\equiv 4 \mod 6\}\\
A_3&=\{(a,(x,y)):a\in\{0,1,2,3\}, x,y\in\Z, \mbox{ with }x\equiv 1\mod 2\mbox{ and }y\equiv 0 \mod 6 \}
\\
A_4&=\{(a,(x,y)):a\in\{0,1,2,3\}, x,y\in\Z, \mbox{ with }x\equiv 0\mod 2\mbox{ and }y\equiv 3 \mod 6\}
\\
A_5&=\{(a,(x,y)):a\in\{0,1,2,3\}, x,y\in\Z,\mbox{with }x\equiv 0\mod  2\mbox{ and }y\equiv 2 \mod 6\}\mbox{, and}\\
A_6&=\{(a,(x,y)):a\in\{0,1,2,3\}, x,y\in\Z,\mbox{with }x\equiv 1\mod  2\mbox{ and }y\equiv 5 \mod 6\}.
\end{align*}

Given the symmetry of the graph, it suffices to show that the vertices in the sets $A_1$, $A_4$, and $A_{5}$ are dominated.
We work on a case-by-case analysis by
considering vertices in these sets.
\begin{enumerate}
\item Fix an arbitrary $4$-cycle of vertices in $A_1$ consisting of $(0,(x,y))$, $(1,(x,y))$, $(2,(x,y))$, and $(3,(x,y))$, where $x\in\Z$, and $y\in\Z$ satisfying $y\equiv 1\mod 6$.
In \Cref{fig:pic41} we label the $4$-cycles $\mathcal{C}_{-1}$, $\mathcal{C}_{0}$, and $\mathcal{C}_{+1}$, at positions $(x-1,y+1)$, $(x,y)$, and $(x+1,y+1)$, respectively, as well as the position of broadcasting vertices $v_1=(0,(x,y-1))$, $v_2=(2,(x-1,y+1))$, and $v_3=(2,(x+1,y+1))$.
The Figures~\ref{fig:y=1mod6}
 we illustrate the signal sent from each of the broadcasting vertices $v_1=(0,(x,y-1))$, $v_2=(2,(x-1,y+1))$, and $v_3=(2,(x+1,y+1))$. 
\begin{figure}[h!]
    \centering
\includegraphics[width=2in,trim=3cm 3cm 3cm 3cm,clip]{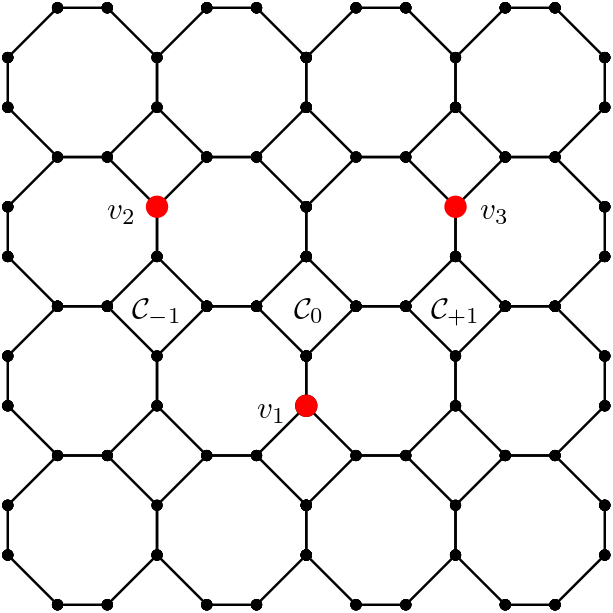}
    \caption{Positioning of three consecutive $4$-cycles and three broadcasting vertices.}
    \label{fig:pic41}
\end{figure}

As the reception is additive, all of the vertices $(a,(x-1,y))$, $(a,(x,y))$ and $(a,(x+1,y))$ with $a\in\{0,1,2,3\}$ receive reception at least $1$ from the broadcasting vertices $v_1,v_2,v_3$, and are hence dominated.
Thus, the vertices in the set $A_1$ are dominated.
\begin{figure}[h!]
\centering
\begin{subfigure}[t]{.3\textwidth}
\centering
\includegraphics[width=2in,trim=.5cm .5cm .5cm 4.5cm,clip]{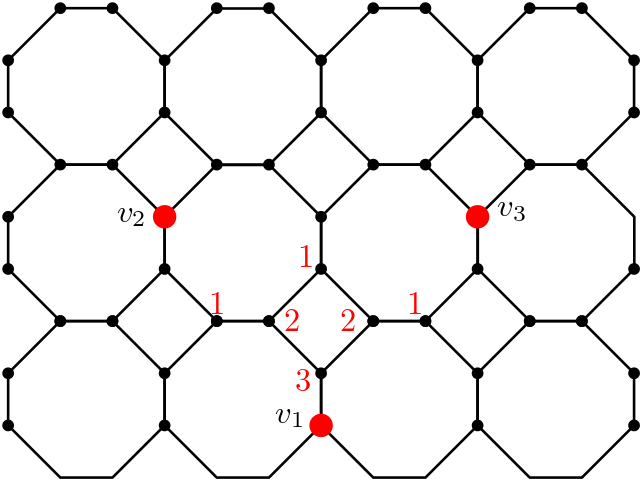}
\caption{Reception from $v_1$.}\label{fig:receptionfromv1}
\end{subfigure}
\quad
\begin{subfigure}[t]{.3\textwidth}
\centering
\includegraphics[width=2in,trim=.5cm .5cm .5cm 4.5cm,clip]{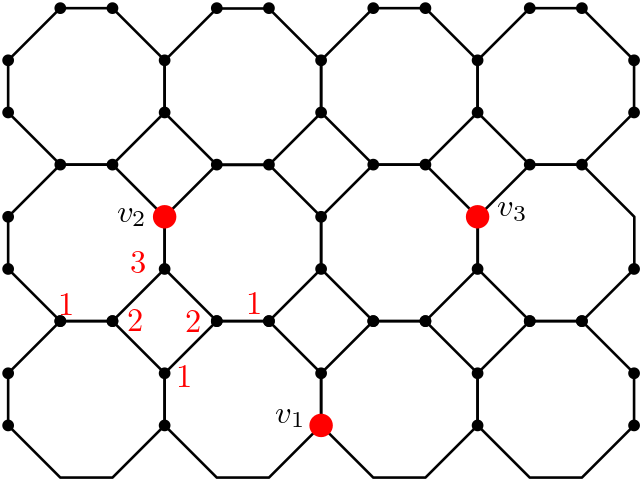}
\caption{Reception from $v_2$.}\label{fig:receptionfromv2}
\end{subfigure}
\quad
\begin{subfigure}[t]{.3\textwidth}
\centering
\includegraphics[width=2in,trim=.5cm .5cm .5cm 4.5cm,clip]{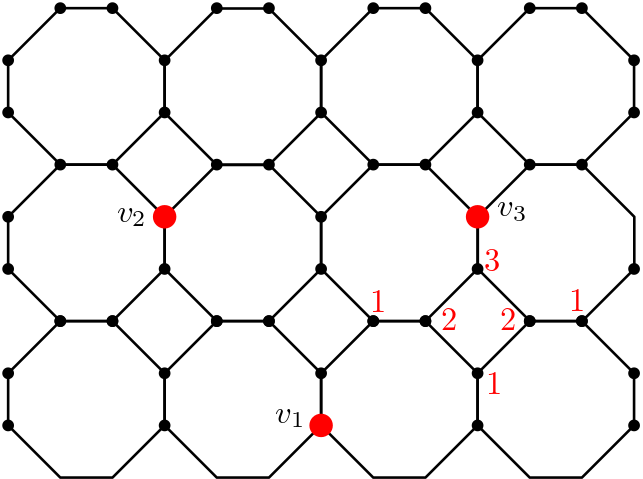}
\caption{Reception from $v_3$.}\label{fig:receptionfromv3}
\end{subfigure}
\caption{The reception at vertices $(a,(x,y))$ with $a\in\{0,1,2,3\}$, $x,y\in\Z$ with $y\equiv 1\mod6$, received from broadcasting vertices $v_1,v_2,v_3$, respectively.}\label{fig:y=1mod6}
\end{figure}

\item Fix an arbitrary $4$-cycle, denoted $\mathcal{C}_0$, in $A_4$ with vertices $(0,(x,y))$, $(1,(x,y))$, $(2,(x,y))$, and $(3,(x,y))$, with $x\in\Z$ with $x\equiv 0 \mod 2$, and $y\in\Z$ with $y\equiv 3\mod 6$.
In Figure~\ref{fig:22ax} we illustrate the $4$-cycle at position $\mathcal{C}_0=(x,y)$ and in \Cref{fig:22b} how the broadcasting vertices 
$v_1=(0,(x-1,y)),v_2=(0,(x+1,y))\in T$ send reception to the vertices in $\mathcal{C}_0$.
Note that the vertices considered are dominated as they have received at least $2$ in reception.

Thus every vertex in $A_4$ is dominated.
\begin{figure}[h!]
\centering
\begin{subfigure}[t]{.4\textwidth}
\centering
\includegraphics[width=2in,trim=.5cm 10cm .5cm 2.5cm,clip]{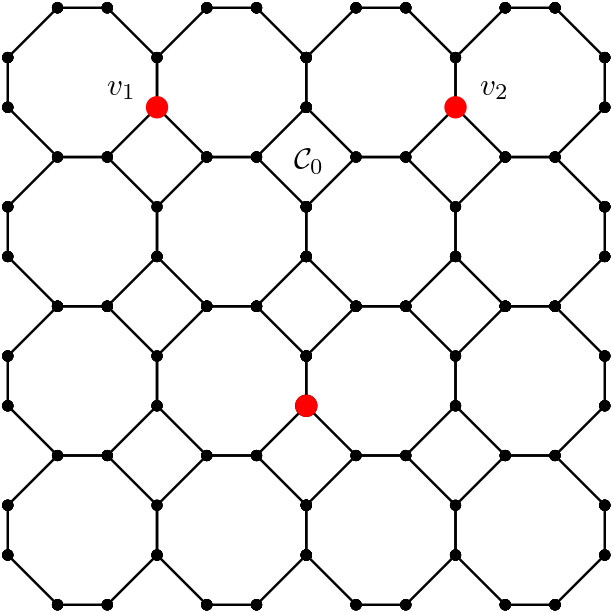}
\caption{Position of the $4$-cycle $\mathcal{C}_0$.}\label{fig:22ax}
\end{subfigure}
\quad\quad
\begin{subfigure}[t]{.4\textwidth}
\centering
\includegraphics[width=2in,trim=.5cm 4.55cm .5cm 2.5cm,clip]{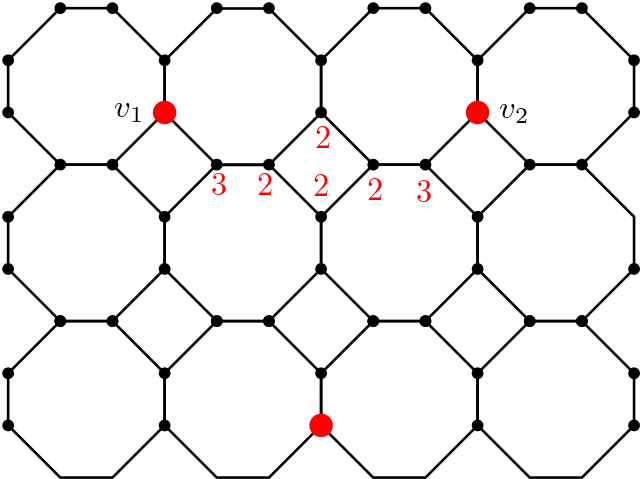}
\caption{Reception received from broadcasting vertices $v_1$ and $v_2$ combined.}\label{fig:22b}
\end{subfigure}
\caption{The position of a $4$-cycle in $A_4$ and nearby broadcasting vertices.}\label{fig:y=2mod4}
\end{figure}

\item Fix an arbitrary $4$-cycle, denoted $\mathcal{C}_0$, in $A_5$ with vertices $(0,(x,y))$, $(1,(x,y))$, $(2,(x,y))$, and $(3,(x,y))$, with $x\in\Z$ with $x\equiv 0 \mod 2$, and $y\in\Z$ with $y\equiv 2\mod 6$.
In Figure~\ref{fig:23a} we illustrate the $4$-cycle at position $\mathcal{C}_0=(x,y)$ and in \Cref{fig:23b} how the broadcasting vertices
$v_1=(2,(x-1,y)),v_2=(2,(x+1,y))\in T$ send reception to the vertices in $\mathcal{C}_0$.
Note that the vertices considered are dominated as they have received at least $2$ in reception.

Thus every vertex in $A_4$ is dominated.
\begin{figure}[h!]
\centering
\begin{subfigure}[t]{.4\textwidth}
\centering
\includegraphics[width=2in,trim=.5cm 10cm .5cm 2.5cm,clip]{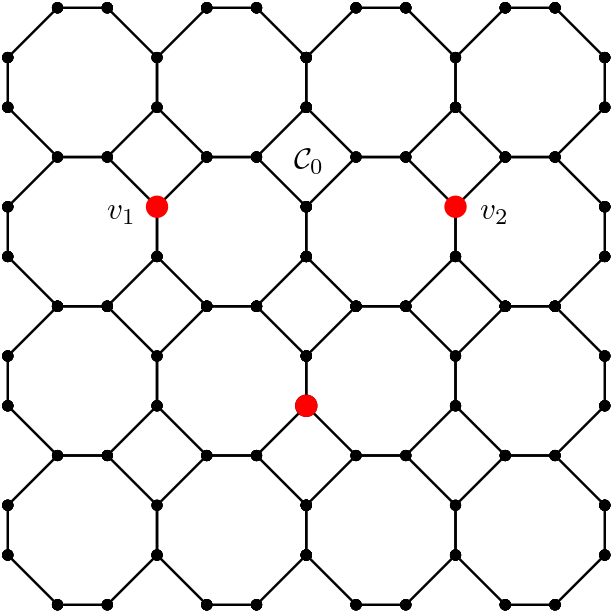}
\caption{Position of the $4$-cycle $\mathcal{C}_0$.}\label{fig:23a}
\end{subfigure}
\quad\quad
\begin{subfigure}[t]{.4\textwidth}
\centering
\includegraphics[width=2in,trim=.5cm 4.55cm .5cm 2.5cm,clip]{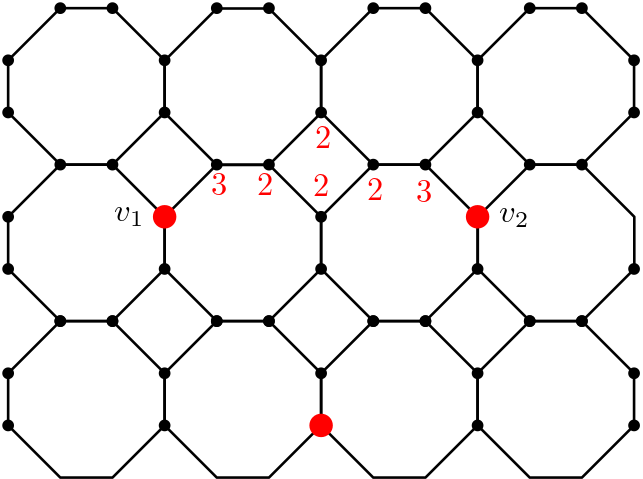}
\caption{Reception received from broadcasting vertices $v_1$ and $v_2$ combined.}\label{fig:23b}
\end{subfigure}
\caption{The position of a $4$-cycle in $A_5$ and nearby broadcasting vertices.}\label{fig:y=2mod4}
\end{figure}

\end{enumerate}

\begin{figure}
\centering
\includegraphics[width=0.5\linewidth,trim=2in 2in 3.2in 2.5in,clip]{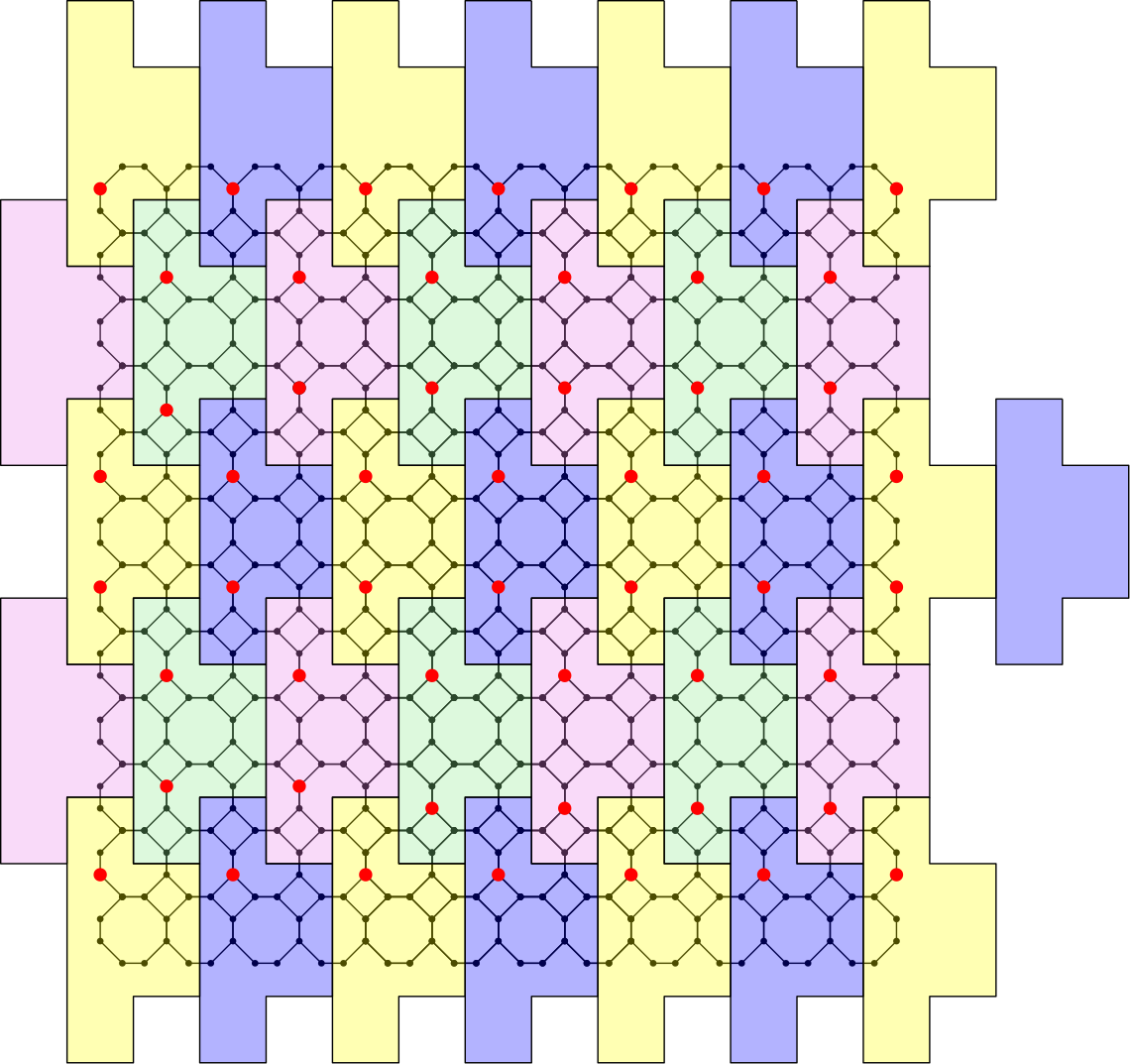}
    \caption{Grouping of six $4$-cycles used to construct the upper bound for the density in the proof of Theorem~\ref{thm:density41}.}
    \label{fig:six4cycles}
\end{figure}
This establishes that $T$ is a broadcast for $\gH$.
Next we determine the proportion of vertices in $\gH$ that are in $T$.
To compute this proportion, we note that we can group six $4$-cycles, as illustrated in Figure~\ref{fig:six4cycles}, in which there are $24$ vertices and only $2$ are dominating vertices. Thus, 
$\delta_{4,1}(\gH)\leq \frac{1}{12}$, as claimed.
\end{proof}

We conclude this section by presenting density conjectures for $(t,r)\in\{(4,2),(4,3),(4,4)\}.$
\begin{conjecture}\label{conj:densities}
The optimal densities of $H_{\infty,\infty}$ satisfy
    \[
        \delta_{4,2}(H_{\infty,\infty})\leq \frac{9}{80},\quad
        \delta_{4,3}(H_{\infty,\infty})\leq \frac{1}{7},\quad\mbox{ and }\quad
        \delta_{4,4}(H_{\infty,\infty})\leq \frac{1}{6}.\]
\end{conjecture}
In Figures~\ref{fig:Delta42},~\ref{fig:Delta43}, and~\ref{fig:Delta44} we provide broadcasts that achieve the bounds provided in Conjecture~\ref{conj:densities}. However, given the many cases needed, and short of a detailed proof, we leave the statements as conjectures and state their proof as Problem \ref{problem:prove conjectures}.

\begin{figure}[htp]
    \centering
    \begin{subfigure}{.45\textwidth}
\centering    \includegraphics[width=3in]{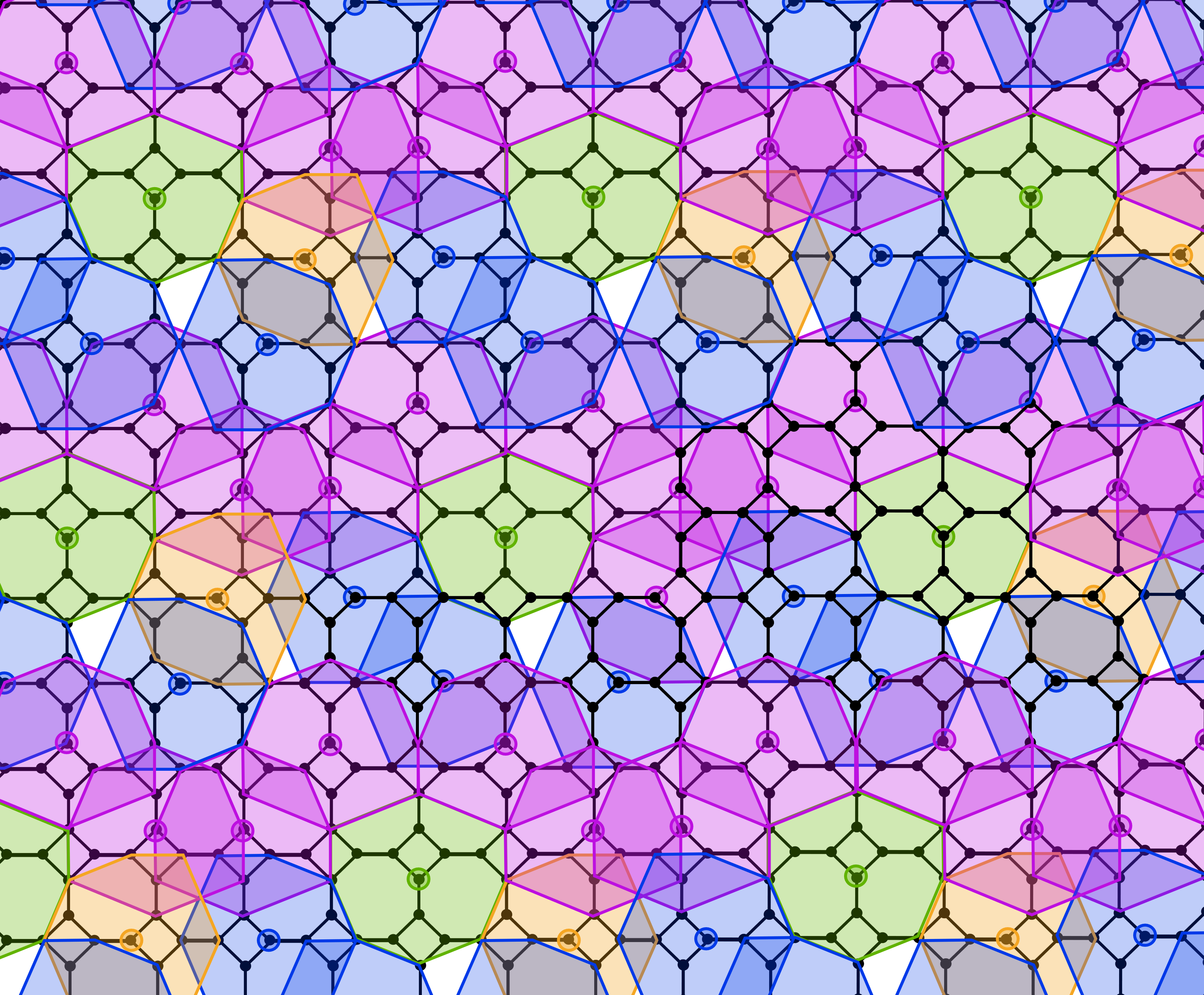}
\caption{A $(4,2)$ broadcast on $H_{\infty,\infty}$.}\label{fig:Delta42}
    \end{subfigure}
    \hfill
\begin{subfigure}{.45\textwidth}
    \centering
    \includegraphics[width=3in]{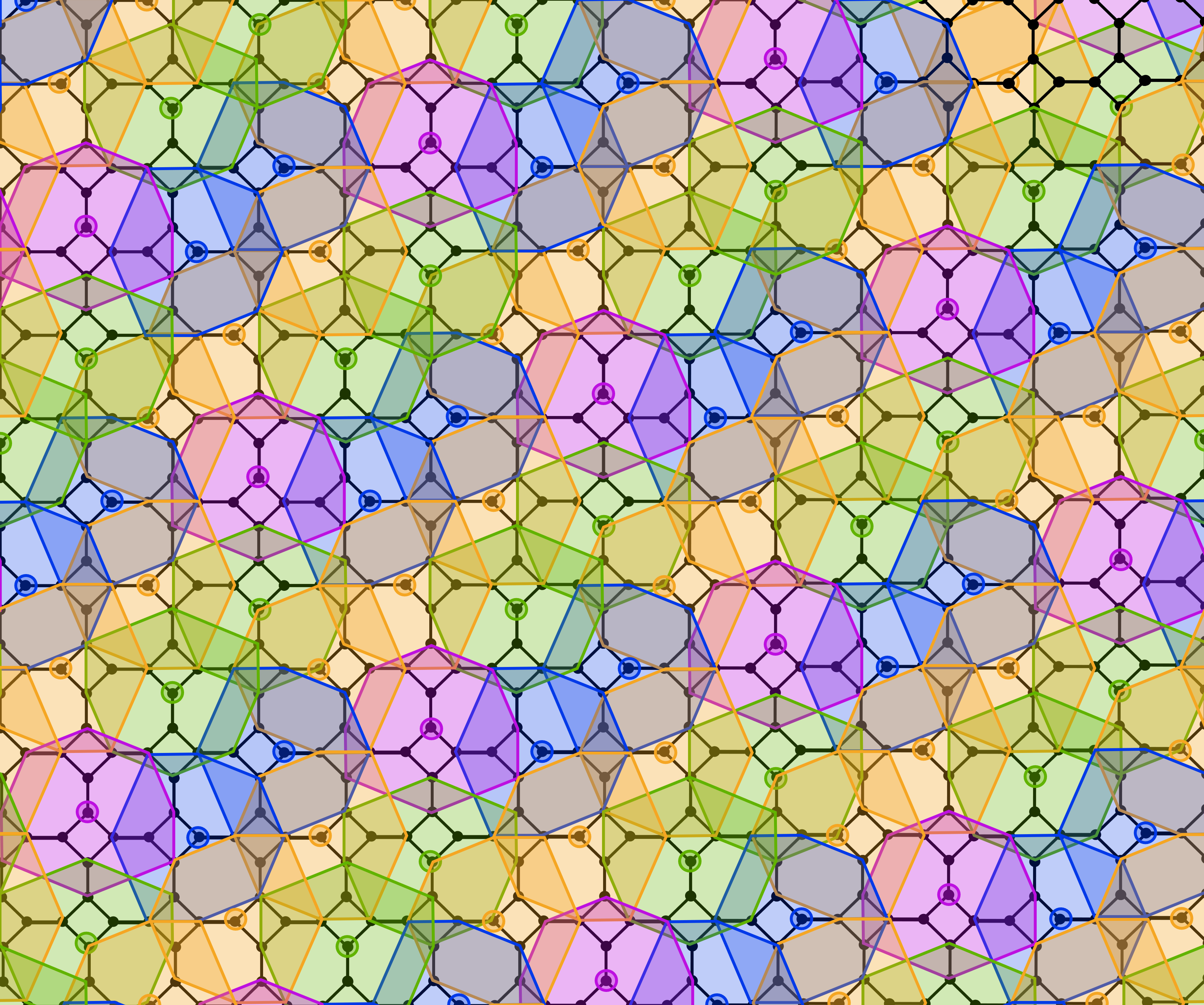}
    \caption{A $(4,3)$ broadcast on $H_{\infty,\infty}$.}
    \label{fig:Delta43}
\end{subfigure}\\

\begin{subfigure}{.35\textwidth}
    \centering
    \includegraphics[width=3in]{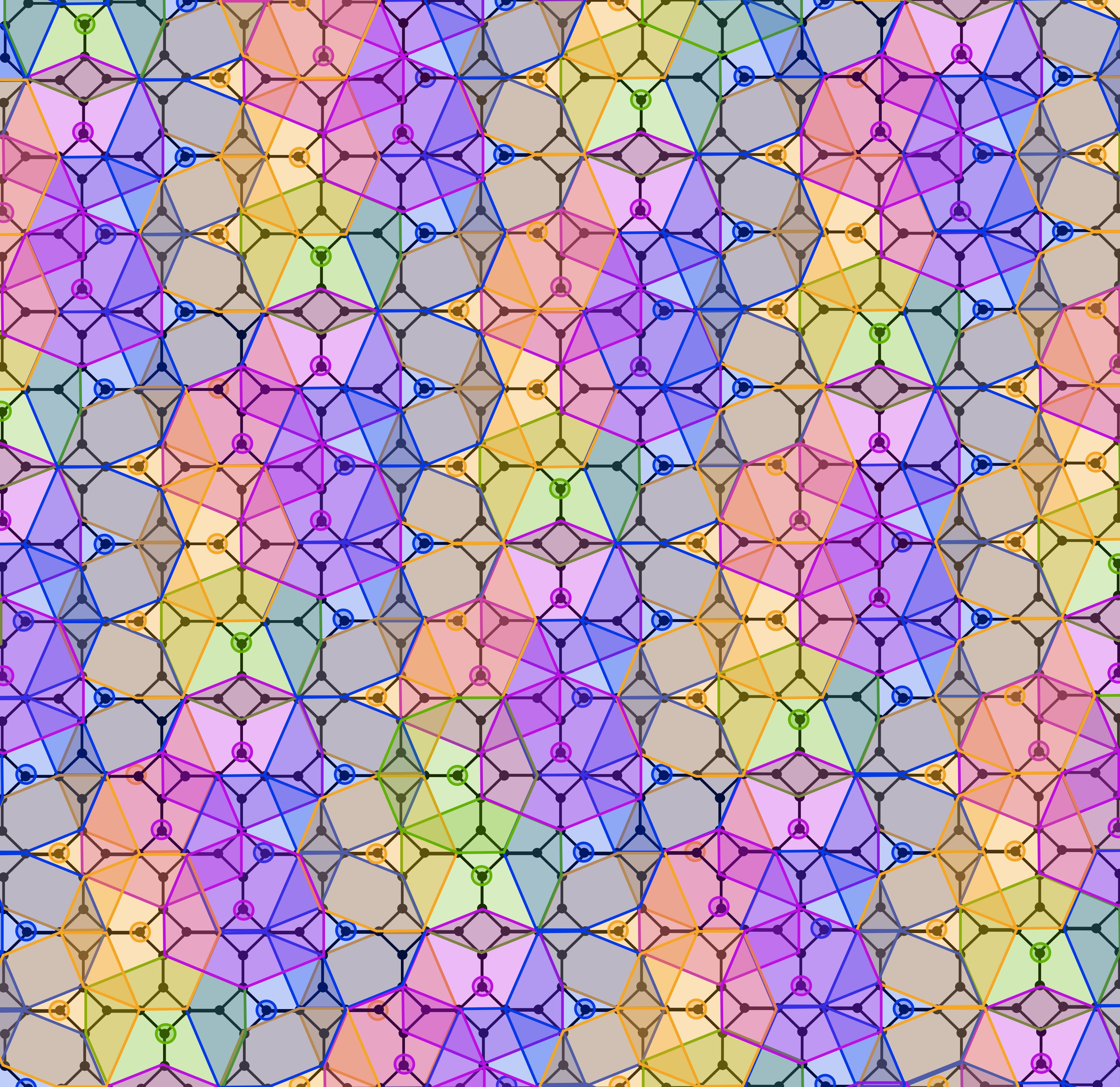}
    \caption{A $(4,4)$ broadcast on $H_{\infty,\infty}$.}
    \label{fig:Delta44}
    \end{subfigure}
    \caption{$(4,2)$, $(4,3)$, and $(4,4)$ broadcasts on $H_{\infty,\infty}$.}
\end{figure}

\section{Some bounds for the \texorpdfstring{$(t,r)$}{(t,r)} domination number of \texorpdfstring{$H_{m,n}$}{Hmn}}\label{sec:tr bounds}

Using the density results from Section~\ref{sec:densities}, we give upper bounds for the $(2,2)$ and $(3,3)$ broadcast domination number of $H_{m,n}$ for $m,n\geq 1$.

\begin{figure}[H]
    \centering
    
    \begin{subfigure}{.45\textwidth}
    \centering
\includegraphics[width=2.75in]{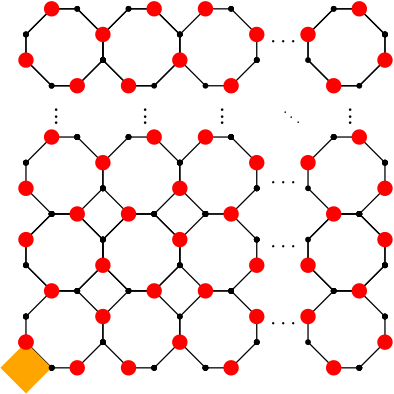}
    \caption{$(2,2)$}
    \label{fig:dominate22Hmn}
\end{subfigure}
\qquad
    \begin{subfigure}{.45\textwidth}
    \centering
\includegraphics[width=2.75in]{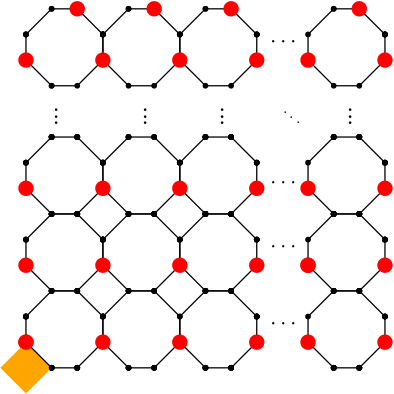}
    \caption{$(3,3)$}\label{fig:dominate33Hmn}
    \end{subfigure}
    \caption{A $(2,2)$ and a $(3,3)$ broadcast dominating set of $H_{m,n}$.}\label{fig:22 and 33}
\end{figure}

\begin{corollary}\label{cor:22bound}
    If $m,n\geq 1$, then $\gamma_{2,2}(H_{m,n})\leq 2mn+m+n$.
\end{corollary}
\begin{proof}Using the broadcast presented in the proof of
    Theorem~\ref{thm:density22}, we position the graph $H_{m,n}$ in Figure~\ref{fig:rowsfor22} so that the bottom left octagon shares its southwest most edge and vertices with the origin. The bottom left-most octagon needs $4$ vertices to be $(2,2)$ broadcast dominated, and the $m-1$ octagons above it each need $3$ additional vertices, see Figure~\ref{fig:dominate22Hmn}.  This contributes $4+3(m-1)=3m+1$ broadcasting vertices to the $(2,2)$ dominating set, see Figure~\ref{fig:dominate22Hmn}.
    Each of the $n-1$ subsequent columns need $3$ broadcasting vertices to dominate the bottom most octagon and $2$ for each of the $m-1$ octagons above it. This contributes $(n-1)(3+2(m-1))$ broadcasting vertices to the count. Thus
    \[\gamma_{2,2}(H_{m,n})\leq 3m+1+(n-1)(3+2(m-1))=2mn+m+n.\qedhere\]
\end{proof}

\begin{corollary}\label{cor:33bound}
    If $m,n\geq 1$, then $\gamma_{3,3}(H_{m,n})\leq mn+m+n$.
\end{corollary}
\begin{proof}    Using the broadcast presented in the proof Theorem~\ref{thm:density33}, we can position the graph $H_{m,n}$ in Figure~\ref{fig:rowsfor33}, so that the bottom left octagon shares its southwest most edge and vertices with the origin. 
Focusing only on the first column, note that of the $m$ octagons the bottom $m-1$ octagons need $2$ vertices to be $(3,3)$ broadcast dominated, and the top octagon needs $3$ vertices, see Figure~\ref{fig:dominate33Hmn}.  
This contributes $3+2(m-1)=2m+1$ broadcasting vertices to the $(3,3)$ dominating set. 
    Each of the $n-1$ subsequent columns need $1$ broadcasting vertex to dominate the bottom $m-1$ octagons and $2$ for the top octagon. This contributes $(n-1)(2+1(m-1))$ broadcasting vertices to the count. Thus
    \[\gamma_{3,3}(H_{m,n})\leq 2m+1+(n-1)(2+1(m-1))=mn+m+n.\qedhere\]
\end{proof}

Given the irregularity of the $(t,r) $ broadcasts presented in the proofs for the bounds on the density of the $(t,r)\in\{(3,1),(3,2),(4,1)\}$ broadcasts, 
we leave it as Problem \ref{open:irregularbroadcasts} to use those results to give similar upper bounds on the $(t,r)$ broadcast domination number of the finite graph $H_{m,n}$ with $m,n\geq 1$.

Moreover, as the reception is additive, having larger $r$ parameter requires one to determine the best arrangement of the broadcasting vertices so that the sum of the reception adds to $r$. These are integer partition problems, which are notoriously difficult to solve and even more challenging  given the structure of our graphs $H_m,n$. Thus, we pose Problem \ref{open:bounds in general} to give bounds for $\gamma_{t,r}(H_{m,n})$ in general.

\section{Future work}
One question of interest is to determine the largest graph $H_{m,m}$ for which $\gamma_{t,r}(H_{m,m})=1$. Based on our experiments the following conjecture appears to hold. 
\begin{conjecture}
    \label{thm:one_tower_h_mm}
    If $m \geq 1$ and $r \geq 1$,\, then $\gamma_{t,r}(H_{m,m})=1$ provided 
    \[ t =\begin{cases} 
       r + 2m + 1& \text{if $m$ even} \\
       r + 2m + 2& \text{if $m$ odd}. \\
   \end{cases}
\]
\end{conjecture}
In Figure~\ref{fig:onetower} we illustrate a placement of a broadcasting vertex (in red) that $(t,1)$ dominates $H_{3,3}$. The green circled vertex is the furthest vertex from the broadcasting vertex. This vertex receives the minimum allowed reception $r=1$. Note that increasing the parameter $r$ would necessitate increasing by the same amount the parameter $t$. This  
ensures that the vertex circled in green would then receive the needed reception  $r$.
This naturally leads to our first problem.

\begin{figure}[h!]
    \centering
    \includegraphics[width=2.5in,trim=0 0 0 1in,clip]{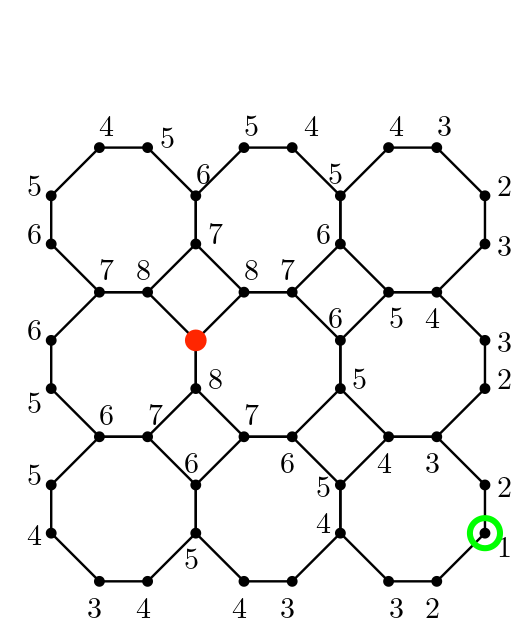}
    \caption{Placement of a broadcasting vertex (in red) with $t=9$ allowing for $H_{3,3}$ to be $(t,1)$ broadcast dominated with a single vertex. We provide the reception received at every vertex and note that the vertex circled in green is the vertex furthest from the dominating vertex in red.}
    \label{fig:onetower}
\end{figure}
\begin{problem} 
Prove or give a counterexample to  Conjecture~\ref{thm:one_tower_h_mm}.
 \end{problem}
 If a counterexample is provided, then we pose the following.
 \begin{problem}
      For fixed integers $m,r\geq $, determine the value of $t$ such that $\gamma_{t,r}(H_{m,m})=1$.
 \end{problem}
One can also generalize the statement of the previous problem to the graph $H_{m,n}$ for $m\neq n$. We also ask the following.
\begin{question}
    Can one adapt the proof technique used in Theorem~\ref{thm:initial_low_bd} to provide a lower bound for the $(t,r)$ broadcast domination number of $H_{m,n}$ for $r>1$?
\end{question}
\begin{question}
As defined, the coordination  
sequences provide a count for the number of vertices that are a fixed distance away from any broadcasting vertex.
Can one utilize the coordination sequence to give an immediate lower bound for the $(t,r)$ broadcast domination number of $H_{m,n}$ for~$r>1$?
\end{question}
The following problems are relate to our results in Section~\ref{sec:densities}.

\begin{problem}\label{problem:prove conjectures}
    Prove that the dominating sets from Conjecture~\ref{conj:densities}   dominate $\gH$ and confirm their densities.
\end{problem}

\begin{question}\label{dobetter}
    Many of the $(t,r)$ broadcast densities for $\gH$ presented (see Theorems~\ref{thm:density21}, \ref{thm:density22}, \ref{thm:density31}, \ref{thm:density32}, \ref{thm:density33}, and \ref{thm:density41}) resulted in non-broadcasting vertices receiving more reception than needed. Are there more optimal dominating sets for $\gH$ for the values $(t,r)$ considered in those results? If not, prove that the densities presented are optimal.
\end{question}
\begin{problem}\label{open:irregularbroadcasts}
    Based on the regularity of the placement of the $(2,2)$ and $(3,3)$ broadcasts on $\gH$ (Theorems \ref{thm:density22} and \ref{thm:density33}), in Corollaries \ref{cor:22bound} and  \ref{cor:33bound} we used these broadcasts to establish the following bounds: 
\[ \gamma_{2,2}(H_{m,n})\leq 2mn+m+n\qquad\mbox{ and }\qquad\gamma_{3,3}(H_{m,n})\leq mn+m+n.\] 
It remains an open problem to use the $(t,r)$ broadcasts for $\gH$ presented in Section \ref{sec:densities} to establish analogous bounds for $\gamma_{t,r}(H_{m,n})$ when $(t,r) \in\{(2,1),(3,1),(3,2),(4,1)\}$. 
\end{problem}

As we noted at the end of Section \ref{sec:tr bounds}, there is an inherent integer partition problem arising whenever we want to determine $(t,r)$ broadcast domination numbers, as the reception is additive and we must sum to the parameter $r$. Such problems are difficult to solve, which contributes to the challenge in solving the following problem.

\begin{problem}\label{open:bounds in general}
    For a fixed pair of integers $(t,r)$ give  bounds for $\gamma_{t,r}(H_{m,n})$ in general. 
\end{problem}

\section*{Acknowledgments}
 J.~Cervantes was partially supported by the SURF and SERA awards  at University of Wisconsin - Milwaukee. J. Cervantes thanks the UW Milwaukee CIberCATSS program for help in support with the programming implementation found in GitHub repository \cite{cervantes2024github}. P.~E.~Harris was supported in part through an EDGE Karen Uhlenbeck Fellowship.

\bibliographystyle{plain}
\bibliography{bibliography}

\end{document}

%% file: h_1_1_dom.tex
\tikzset{every picture/.style={line width=0.75pt}} %set default line width to 0.75pt        

\begin{tikzpicture}[x=0.75pt,y=0.75pt,yscale=-.6,xscale=.6]
%uncomment if require: \path (0,300); %set diagram left start at 0, and has height of 300

%Shape: Regular Polygon [id:dp23975975494886748] 
\draw   (225.65,114.7) -- (195.48,144.86) -- (152.83,144.86) -- (122.66,114.7) -- (122.66,72.04) -- (152.83,41.88) -- (195.48,41.88) -- (225.65,72.04) -- cycle ;
%Shape: Circle [id:dp11761017406636487] 
\draw  [fill={rgb, 255:red, 0; green, 0; blue, 0 }  ,fill opacity=1 ] (118.21,72.04) .. controls (118.21,69.58) and (120.21,67.59) .. (122.66,67.59) .. controls (125.12,67.59) and (127.11,69.58) .. (127.11,72.04) .. controls (127.11,74.5) and (125.12,76.49) .. (122.66,76.49) .. controls (120.21,76.49) and (118.21,74.5) .. (118.21,72.04) -- cycle ;
%Shape: Circle [id:dp9779656392454362] 
\draw  [fill={rgb, 255:red, 0; green, 0; blue, 0 }  ,fill opacity=1 ] (148.38,41.88) .. controls (148.38,39.42) and (150.37,37.43) .. (152.83,37.43) .. controls (155.28,37.43) and (157.28,39.42) .. (157.28,41.88) .. controls (157.28,44.33) and (155.28,46.33) .. (152.83,46.33) .. controls (150.37,46.33) and (148.38,44.33) .. (148.38,41.88) -- cycle ;
%Shape: Circle [id:dp2620722135306457] 
\draw  [fill={rgb, 255:red, 0; green, 0; blue, 0 }  ,fill opacity=1 ] (191.03,41.88) .. controls (191.03,39.42) and (193.03,37.43) .. (195.48,37.43) .. controls (197.94,37.43) and (199.93,39.42) .. (199.93,41.88) .. controls (199.93,44.33) and (197.94,46.33) .. (195.48,46.33) .. controls (193.03,46.33) and (191.03,44.33) .. (191.03,41.88) -- cycle ;
%Shape: Circle [id:dp42639029372559467] 
\draw  [fill={rgb, 255:red, 0; green, 0; blue, 0 }  ,fill opacity=1 ] (148.38,144.86) .. controls (148.38,142.4) and (150.37,140.41) .. (152.83,140.41) .. controls (155.28,140.41) and (157.28,142.4) .. (157.28,144.86) .. controls (157.28,147.32) and (155.28,149.31) .. (152.83,149.31) .. controls (150.37,149.31) and (148.38,147.32) .. (148.38,144.86) -- cycle ;
%Shape: Circle [id:dp21885088998736768] 
\draw  [fill={rgb, 255:red, 0; green, 0; blue, 0 }  ,fill opacity=1 ] (118.21,114.7) .. controls (118.21,112.24) and (120.21,110.25) .. (122.66,110.25) .. controls (125.12,110.25) and (127.11,112.24) .. (127.11,114.7) .. controls (127.11,117.15) and (125.12,119.15) .. (122.66,119.15) .. controls (120.21,119.15) and (118.21,117.15) .. (118.21,114.7) -- cycle ;
%Shape: Circle [id:dp7196190379300513] 
\draw  [fill={rgb, 255:red, 0; green, 0; blue, 0 }  ,fill opacity=1 ] (191.03,144.86) .. controls (191.03,142.4) and (193.03,140.41) .. (195.48,140.41) .. controls (197.94,140.41) and (199.93,142.4) .. (199.93,144.86) .. controls (199.93,147.32) and (197.94,149.31) .. (195.48,149.31) .. controls (193.03,149.31) and (191.03,147.32) .. (191.03,144.86) -- cycle ;
%Shape: Circle [id:dp617505296726992] 
\draw  [fill={rgb, 255:red, 0; green, 0; blue, 0 }  ,fill opacity=1 ] (221.2,114.7) .. controls (221.2,112.24) and (223.19,110.25) .. (225.65,110.25) .. controls (228.1,110.25) and (230.1,112.24) .. (230.1,114.7) .. controls (230.1,117.15) and (228.1,119.15) .. (225.65,119.15) .. controls (223.19,119.15) and (221.2,117.15) .. (221.2,114.7) -- cycle ;
%Shape: Circle [id:dp8013240969169696] 
\draw  [fill={rgb, 255:red, 0; green, 0; blue, 0 }  ,fill opacity=1 ] (221.2,72.04) .. controls (221.2,69.58) and (223.19,67.59) .. (225.65,67.59) .. controls (228.1,67.59) and (230.1,69.58) .. (230.1,72.04) .. controls (230.1,74.5) and (228.1,76.49) .. (225.65,76.49) .. controls (223.19,76.49) and (221.2,74.5) .. (221.2,72.04) -- cycle ;
%Shape: Circle [id:dp3482988392650931] 
\draw  [color={rgb, 255:red, 208; green, 2; blue, 27 }  ,draw opacity=1 ][line width=1.5]  (112.76,72.04) .. controls (112.76,66.57) and (117.2,62.14) .. (122.66,62.14) .. controls (128.13,62.14) and (132.56,66.57) .. (132.56,72.04) .. controls (132.56,77.51) and (128.13,81.94) .. (122.66,81.94) .. controls (117.2,81.94) and (112.76,77.51) .. (112.76,72.04) -- cycle ;
%Shape: Circle [id:dp052152663926080645] 
\draw  [color={rgb, 255:red, 208; green, 2; blue, 27 }  ,draw opacity=1 ][line width=1.5]  (215.75,72.04) .. controls (215.75,66.57) and (220.18,62.14) .. (225.65,62.14) .. controls (231.11,62.14) and (235.55,66.57) .. (235.55,72.04) .. controls (235.55,77.51) and (231.11,81.94) .. (225.65,81.94) .. controls (220.18,81.94) and (215.75,77.51) .. (215.75,72.04) -- cycle ;
%Shape: Circle [id:dp2683454888264245] 
\draw  [color={rgb, 255:red, 208; green, 2; blue, 27 }  ,draw opacity=1 ][line width=1.5]  (185.58,144.86) .. controls (185.58,139.39) and (190.02,134.96) .. (195.48,134.96) .. controls (200.95,134.96) and (205.38,139.39) .. (205.38,144.86) .. controls (205.38,150.33) and (200.95,154.76) .. (195.48,154.76) .. controls (190.02,154.76) and (185.58,150.33) .. (185.58,144.86) -- cycle ;

\end{tikzpicture}

%% file: h_1_2_dom.tex
\tikzset{every picture/.style={line width=0.75pt}} %set default line width to 0.75pt        

\begin{tikzpicture}[x=0.75pt,y=0.75pt,yscale=-.6,xscale=.6]
%uncomment if require: \path (0,300); %set diagram left start at 0, and has height of 300

%Shape: Regular Polygon [id:dp23975975494886748] 
\draw   (150.65,144.7) -- (120.48,174.86) -- (77.83,174.86) -- (47.66,144.7) -- (47.66,102.04) -- (77.83,71.88) -- (120.48,71.88) -- (150.65,102.04) -- cycle ;
%Shape: Circle [id:dp11761017406636487] 
\draw  [fill={rgb, 255:red, 0; green, 0; blue, 0 }  ,fill opacity=1 ] (43.21,102.04) .. controls (43.21,99.58) and (45.21,97.59) .. (47.66,97.59) .. controls (50.12,97.59) and (52.11,99.58) .. (52.11,102.04) .. controls (52.11,104.5) and (50.12,106.49) .. (47.66,106.49) .. controls (45.21,106.49) and (43.21,104.5) .. (43.21,102.04) -- cycle ;
%Shape: Circle [id:dp9779656392454362] 
\draw  [fill={rgb, 255:red, 0; green, 0; blue, 0 }  ,fill opacity=1 ] (73.38,71.88) .. controls (73.38,69.42) and (75.37,67.43) .. (77.83,67.43) .. controls (80.28,67.43) and (82.28,69.42) .. (82.28,71.88) .. controls (82.28,74.33) and (80.28,76.33) .. (77.83,76.33) .. controls (75.37,76.33) and (73.38,74.33) .. (73.38,71.88) -- cycle ;
%Shape: Circle [id:dp2620722135306457] 
\draw  [fill={rgb, 255:red, 0; green, 0; blue, 0 }  ,fill opacity=1 ] (116.03,71.88) .. controls (116.03,69.42) and (118.03,67.43) .. (120.48,67.43) .. controls (122.94,67.43) and (124.93,69.42) .. (124.93,71.88) .. controls (124.93,74.33) and (122.94,76.33) .. (120.48,76.33) .. controls (118.03,76.33) and (116.03,74.33) .. (116.03,71.88) -- cycle ;
%Shape: Circle [id:dp42639029372559467] 
\draw  [fill={rgb, 255:red, 0; green, 0; blue, 0 }  ,fill opacity=1 ] (73.38,174.86) .. controls (73.38,172.4) and (75.37,170.41) .. (77.83,170.41) .. controls (80.28,170.41) and (82.28,172.4) .. (82.28,174.86) .. controls (82.28,177.32) and (80.28,179.31) .. (77.83,179.31) .. controls (75.37,179.31) and (73.38,177.32) .. (73.38,174.86) -- cycle ;
%Shape: Circle [id:dp21885088998736768] 
\draw  [fill={rgb, 255:red, 0; green, 0; blue, 0 }  ,fill opacity=1 ] (43.21,144.7) .. controls (43.21,142.24) and (45.21,140.25) .. (47.66,140.25) .. controls (50.12,140.25) and (52.11,142.24) .. (52.11,144.7) .. controls (52.11,147.15) and (50.12,149.15) .. (47.66,149.15) .. controls (45.21,149.15) and (43.21,147.15) .. (43.21,144.7) -- cycle ;
%Shape: Circle [id:dp7196190379300513] 
\draw  [fill={rgb, 255:red, 0; green, 0; blue, 0 }  ,fill opacity=1 ] (116.03,174.86) .. controls (116.03,172.4) and (118.03,170.41) .. (120.48,170.41) .. controls (122.94,170.41) and (124.93,172.4) .. (124.93,174.86) .. controls (124.93,177.32) and (122.94,179.31) .. (120.48,179.31) .. controls (118.03,179.31) and (116.03,177.32) .. (116.03,174.86) -- cycle ;
%Shape: Circle [id:dp617505296726992] 
\draw  [fill={rgb, 255:red, 0; green, 0; blue, 0 }  ,fill opacity=1 ] (146.2,144.7) .. controls (146.2,142.24) and (148.19,140.25) .. (150.65,140.25) .. controls (153.1,140.25) and (155.1,142.24) .. (155.1,144.7) .. controls (155.1,147.15) and (153.1,149.15) .. (150.65,149.15) .. controls (148.19,149.15) and (146.2,147.15) .. (146.2,144.7) -- cycle ;
%Shape: Circle [id:dp8013240969169696] 
\draw  [fill={rgb, 255:red, 0; green, 0; blue, 0 }  ,fill opacity=1 ] (146.2,102.04) .. controls (146.2,99.58) and (148.19,97.59) .. (150.65,97.59) .. controls (153.1,97.59) and (155.1,99.58) .. (155.1,102.04) .. controls (155.1,104.5) and (153.1,106.49) .. (150.65,106.49) .. controls (148.19,106.49) and (146.2,104.5) .. (146.2,102.04) -- cycle ;
%Shape: Circle [id:dp3482988392650931] 
\draw  [color={rgb, 255:red, 208; green, 2; blue, 27 }  ,draw opacity=1 ][line width=1.5]  (37.76,102.04) .. controls (37.76,96.57) and (42.2,92.14) .. (47.66,92.14) .. controls (53.13,92.14) and (57.56,96.57) .. (57.56,102.04) .. controls (57.56,107.51) and (53.13,111.94) .. (47.66,111.94) .. controls (42.2,111.94) and (37.76,107.51) .. (37.76,102.04) -- cycle ;
%Shape: Circle [id:dp052152663926080645] 
\draw  [color={rgb, 255:red, 208; green, 2; blue, 27 }  ,draw opacity=1 ][line width=1.5]  (140.75,102.04) .. controls (140.75,96.57) and (145.18,92.14) .. (150.65,92.14) .. controls (156.11,92.14) and (160.55,96.57) .. (160.55,102.04) .. controls (160.55,107.51) and (156.11,111.94) .. (150.65,111.94) .. controls (145.18,111.94) and (140.75,107.51) .. (140.75,102.04) -- cycle ;
%Shape: Circle [id:dp2683454888264245] 
\draw  [color={rgb, 255:red, 208; green, 2; blue, 27 }  ,draw opacity=1 ][line width=1.5]  (110.58,174.86) .. controls (110.58,169.39) and (115.02,164.96) .. (120.48,164.96) .. controls (125.95,164.96) and (130.38,169.39) .. (130.38,174.86) .. controls (130.38,180.33) and (125.95,184.76) .. (120.48,184.76) .. controls (115.02,184.76) and (110.58,180.33) .. (110.58,174.86) -- cycle ;
%Shape: Regular Polygon [id:dp887599755661273] 
\draw   (253.65,144.7) -- (223.48,174.86) -- (180.83,174.86) -- (150.66,144.7) -- (150.66,102.04) -- (180.83,71.88) -- (223.48,71.88) -- (253.65,102.04) -- cycle ;
%Shape: Circle [id:dp20702660397933614] 
\draw  [fill={rgb, 255:red, 0; green, 0; blue, 0 }  ,fill opacity=1 ] (176.38,71.88) .. controls (176.38,69.42) and (178.37,67.43) .. (180.83,67.43) .. controls (183.28,67.43) and (185.28,69.42) .. (185.28,71.88) .. controls (185.28,74.33) and (183.28,76.33) .. (180.83,76.33) .. controls (178.37,76.33) and (176.38,74.33) .. (176.38,71.88) -- cycle ;
%Shape: Circle [id:dp27283332462681864] 
\draw  [fill={rgb, 255:red, 0; green, 0; blue, 0 }  ,fill opacity=1 ] (219.03,71.88) .. controls (219.03,69.42) and (221.03,67.43) .. (223.48,67.43) .. controls (225.94,67.43) and (227.93,69.42) .. (227.93,71.88) .. controls (227.93,74.33) and (225.94,76.33) .. (223.48,76.33) .. controls (221.03,76.33) and (219.03,74.33) .. (219.03,71.88) -- cycle ;
%Shape: Circle [id:dp11455648298227561] 
\draw  [fill={rgb, 255:red, 0; green, 0; blue, 0 }  ,fill opacity=1 ] (176.38,174.86) .. controls (176.38,172.4) and (178.37,170.41) .. (180.83,170.41) .. controls (183.28,170.41) and (185.28,172.4) .. (185.28,174.86) .. controls (185.28,177.32) and (183.28,179.31) .. (180.83,179.31) .. controls (178.37,179.31) and (176.38,177.32) .. (176.38,174.86) -- cycle ;
%Shape: Circle [id:dp8536624921732355] 
\draw  [fill={rgb, 255:red, 0; green, 0; blue, 0 }  ,fill opacity=1 ] (219.03,174.86) .. controls (219.03,172.4) and (221.03,170.41) .. (223.48,170.41) .. controls (225.94,170.41) and (227.93,172.4) .. (227.93,174.86) .. controls (227.93,177.32) and (225.94,179.31) .. (223.48,179.31) .. controls (221.03,179.31) and (219.03,177.32) .. (219.03,174.86) -- cycle ;
%Shape: Circle [id:dp07971455068400357] 
\draw  [fill={rgb, 255:red, 0; green, 0; blue, 0 }  ,fill opacity=1 ] (249.2,144.7) .. controls (249.2,142.24) and (251.19,140.25) .. (253.65,140.25) .. controls (256.1,140.25) and (258.1,142.24) .. (258.1,144.7) .. controls (258.1,147.15) and (256.1,149.15) .. (253.65,149.15) .. controls (251.19,149.15) and (249.2,147.15) .. (249.2,144.7) -- cycle ;
%Shape: Circle [id:dp18609421540781912] 
\draw  [fill={rgb, 255:red, 0; green, 0; blue, 0 }  ,fill opacity=1 ] (249.2,102.04) .. controls (249.2,99.58) and (251.19,97.59) .. (253.65,97.59) .. controls (256.1,97.59) and (258.1,99.58) .. (258.1,102.04) .. controls (258.1,104.5) and (256.1,106.49) .. (253.65,106.49) .. controls (251.19,106.49) and (249.2,104.5) .. (249.2,102.04) -- cycle ;
%Shape: Circle [id:dp9035424830258048] 
\draw  [color={rgb, 255:red, 208; green, 2; blue, 27 }  ,draw opacity=1 ][line width=1.5]  (243.75,102.04) .. controls (243.75,96.57) and (248.18,92.14) .. (253.65,92.14) .. controls (259.11,92.14) and (263.55,96.57) .. (263.55,102.04) .. controls (263.55,107.51) and (259.11,111.94) .. (253.65,111.94) .. controls (248.18,111.94) and (243.75,107.51) .. (243.75,102.04) -- cycle ;
%Shape: Circle [id:dp52514565121098] 
\draw  [color={rgb, 255:red, 208; green, 2; blue, 27 }  ,draw opacity=1 ][line width=1.5]  (213.58,174.86) .. controls (213.58,169.39) and (218.02,164.96) .. (223.48,164.96) .. controls (228.95,164.96) and (233.38,169.39) .. (233.38,174.86) .. controls (233.38,180.33) and (228.95,184.76) .. (223.48,184.76) .. controls (218.02,184.76) and (213.58,180.33) .. (213.58,174.86) -- cycle ;

\end{tikzpicture}

%% file: h_1_3_dom.tex
\tikzset{every picture/.style={line width=0.75pt}} %set default line width to 0.75pt        

\begin{tikzpicture}[x=0.75pt,y=0.75pt,yscale=-.6,xscale=.6]
%uncomment if require: \path (0,300); %set diagram left start at 0, and has height of 300

%Shape: Regular Polygon [id:dp23975975494886748] 
\draw   (150.65,144.7) -- (120.48,174.86) -- (77.83,174.86) -- (47.66,144.7) -- (47.66,102.04) -- (77.83,71.88) -- (120.48,71.88) -- (150.65,102.04) -- cycle ;
%Shape: Circle [id:dp11761017406636487] 
\draw  [fill={rgb, 255:red, 0; green, 0; blue, 0 }  ,fill opacity=1 ] (43.21,102.04) .. controls (43.21,99.58) and (45.21,97.59) .. (47.66,97.59) .. controls (50.12,97.59) and (52.11,99.58) .. (52.11,102.04) .. controls (52.11,104.5) and (50.12,106.49) .. (47.66,106.49) .. controls (45.21,106.49) and (43.21,104.5) .. (43.21,102.04) -- cycle ;
%Shape: Circle [id:dp9779656392454362] 
\draw  [fill={rgb, 255:red, 0; green, 0; blue, 0 }  ,fill opacity=1 ] (73.38,71.88) .. controls (73.38,69.42) and (75.37,67.43) .. (77.83,67.43) .. controls (80.28,67.43) and (82.28,69.42) .. (82.28,71.88) .. controls (82.28,74.33) and (80.28,76.33) .. (77.83,76.33) .. controls (75.37,76.33) and (73.38,74.33) .. (73.38,71.88) -- cycle ;
%Shape: Circle [id:dp2620722135306457] 
\draw  [fill={rgb, 255:red, 0; green, 0; blue, 0 }  ,fill opacity=1 ] (116.03,71.88) .. controls (116.03,69.42) and (118.03,67.43) .. (120.48,67.43) .. controls (122.94,67.43) and (124.93,69.42) .. (124.93,71.88) .. controls (124.93,74.33) and (122.94,76.33) .. (120.48,76.33) .. controls (118.03,76.33) and (116.03,74.33) .. (116.03,71.88) -- cycle ;
%Shape: Circle [id:dp42639029372559467] 
\draw  [fill={rgb, 255:red, 0; green, 0; blue, 0 }  ,fill opacity=1 ] (73.38,174.86) .. controls (73.38,172.4) and (75.37,170.41) .. (77.83,170.41) .. controls (80.28,170.41) and (82.28,172.4) .. (82.28,174.86) .. controls (82.28,177.32) and (80.28,179.31) .. (77.83,179.31) .. controls (75.37,179.31) and (73.38,177.32) .. (73.38,174.86) -- cycle ;
%Shape: Circle [id:dp21885088998736768] 
\draw  [fill={rgb, 255:red, 0; green, 0; blue, 0 }  ,fill opacity=1 ] (43.21,144.7) .. controls (43.21,142.24) and (45.21,140.25) .. (47.66,140.25) .. controls (50.12,140.25) and (52.11,142.24) .. (52.11,144.7) .. controls (52.11,147.15) and (50.12,149.15) .. (47.66,149.15) .. controls (45.21,149.15) and (43.21,147.15) .. (43.21,144.7) -- cycle ;
%Shape: Circle [id:dp7196190379300513] 
\draw  [fill={rgb, 255:red, 0; green, 0; blue, 0 }  ,fill opacity=1 ] (116.03,174.86) .. controls (116.03,172.4) and (118.03,170.41) .. (120.48,170.41) .. controls (122.94,170.41) and (124.93,172.4) .. (124.93,174.86) .. controls (124.93,177.32) and (122.94,179.31) .. (120.48,179.31) .. controls (118.03,179.31) and (116.03,177.32) .. (116.03,174.86) -- cycle ;
%Shape: Circle [id:dp617505296726992] 
\draw  [fill={rgb, 255:red, 0; green, 0; blue, 0 }  ,fill opacity=1 ] (146.2,144.7) .. controls (146.2,142.24) and (148.19,140.25) .. (150.65,140.25) .. controls (153.1,140.25) and (155.1,142.24) .. (155.1,144.7) .. controls (155.1,147.15) and (153.1,149.15) .. (150.65,149.15) .. controls (148.19,149.15) and (146.2,147.15) .. (146.2,144.7) -- cycle ;
%Shape: Circle [id:dp8013240969169696] 
\draw  [fill={rgb, 255:red, 0; green, 0; blue, 0 }  ,fill opacity=1 ] (146.2,102.04) .. controls (146.2,99.58) and (148.19,97.59) .. (150.65,97.59) .. controls (153.1,97.59) and (155.1,99.58) .. (155.1,102.04) .. controls (155.1,104.5) and (153.1,106.49) .. (150.65,106.49) .. controls (148.19,106.49) and (146.2,104.5) .. (146.2,102.04) -- cycle ;
%Shape: Circle [id:dp3482988392650931] 
\draw  [color={rgb, 255:red, 208; green, 2; blue, 27 }  ,draw opacity=1 ][line width=1.5]  (37.76,102.04) .. controls (37.76,96.57) and (42.2,92.14) .. (47.66,92.14) .. controls (53.13,92.14) and (57.56,96.57) .. (57.56,102.04) .. controls (57.56,107.51) and (53.13,111.94) .. (47.66,111.94) .. controls (42.2,111.94) and (37.76,107.51) .. (37.76,102.04) -- cycle ;
%Shape: Circle [id:dp052152663926080645] 
\draw  [color={rgb, 255:red, 208; green, 2; blue, 27 }  ,draw opacity=1 ][line width=1.5]  (140.75,102.04) .. controls (140.75,96.57) and (145.18,92.14) .. (150.65,92.14) .. controls (156.11,92.14) and (160.55,96.57) .. (160.55,102.04) .. controls (160.55,107.51) and (156.11,111.94) .. (150.65,111.94) .. controls (145.18,111.94) and (140.75,107.51) .. (140.75,102.04) -- cycle ;
%Shape: Circle [id:dp2683454888264245] 
\draw  [color={rgb, 255:red, 208; green, 2; blue, 27 }  ,draw opacity=1 ][line width=1.5]  (110.58,174.86) .. controls (110.58,169.39) and (115.02,164.96) .. (120.48,164.96) .. controls (125.95,164.96) and (130.38,169.39) .. (130.38,174.86) .. controls (130.38,180.33) and (125.95,184.76) .. (120.48,184.76) .. controls (115.02,184.76) and (110.58,180.33) .. (110.58,174.86) -- cycle ;
%Shape: Regular Polygon [id:dp887599755661273] 
\draw   (253.65,144.7) -- (223.48,174.86) -- (180.83,174.86) -- (150.66,144.7) -- (150.66,102.04) -- (180.83,71.88) -- (223.48,71.88) -- (253.65,102.04) -- cycle ;
%Shape: Circle [id:dp20702660397933614] 
\draw  [fill={rgb, 255:red, 0; green, 0; blue, 0 }  ,fill opacity=1 ] (176.38,71.88) .. controls (176.38,69.42) and (178.37,67.43) .. (180.83,67.43) .. controls (183.28,67.43) and (185.28,69.42) .. (185.28,71.88) .. controls (185.28,74.33) and (183.28,76.33) .. (180.83,76.33) .. controls (178.37,76.33) and (176.38,74.33) .. (176.38,71.88) -- cycle ;
%Shape: Circle [id:dp27283332462681864] 
\draw  [fill={rgb, 255:red, 0; green, 0; blue, 0 }  ,fill opacity=1 ] (219.03,71.88) .. controls (219.03,69.42) and (221.03,67.43) .. (223.48,67.43) .. controls (225.94,67.43) and (227.93,69.42) .. (227.93,71.88) .. controls (227.93,74.33) and (225.94,76.33) .. (223.48,76.33) .. controls (221.03,76.33) and (219.03,74.33) .. (219.03,71.88) -- cycle ;
%Shape: Circle [id:dp11455648298227561] 
\draw  [fill={rgb, 255:red, 0; green, 0; blue, 0 }  ,fill opacity=1 ] (176.38,174.86) .. controls (176.38,172.4) and (178.37,170.41) .. (180.83,170.41) .. controls (183.28,170.41) and (185.28,172.4) .. (185.28,174.86) .. controls (185.28,177.32) and (183.28,179.31) .. (180.83,179.31) .. controls (178.37,179.31) and (176.38,177.32) .. (176.38,174.86) -- cycle ;
%Shape: Circle [id:dp8536624921732355] 
\draw  [fill={rgb, 255:red, 0; green, 0; blue, 0 }  ,fill opacity=1 ] (219.03,174.86) .. controls (219.03,172.4) and (221.03,170.41) .. (223.48,170.41) .. controls (225.94,170.41) and (227.93,172.4) .. (227.93,174.86) .. controls (227.93,177.32) and (225.94,179.31) .. (223.48,179.31) .. controls (221.03,179.31) and (219.03,177.32) .. (219.03,174.86) -- cycle ;
%Shape: Circle [id:dp07971455068400357] 
\draw  [fill={rgb, 255:red, 0; green, 0; blue, 0 }  ,fill opacity=1 ] (249.2,144.7) .. controls (249.2,142.24) and (251.19,140.25) .. (253.65,140.25) .. controls (256.1,140.25) and (258.1,142.24) .. (258.1,144.7) .. controls (258.1,147.15) and (256.1,149.15) .. (253.65,149.15) .. controls (251.19,149.15) and (249.2,147.15) .. (249.2,144.7) -- cycle ;
%Shape: Circle [id:dp18609421540781912] 
\draw  [fill={rgb, 255:red, 0; green, 0; blue, 0 }  ,fill opacity=1 ] (249.2,102.04) .. controls (249.2,99.58) and (251.19,97.59) .. (253.65,97.59) .. controls (256.1,97.59) and (258.1,99.58) .. (258.1,102.04) .. controls (258.1,104.5) and (256.1,106.49) .. (253.65,106.49) .. controls (251.19,106.49) and (249.2,104.5) .. (249.2,102.04) -- cycle ;
%Shape: Circle [id:dp9035424830258048] 
\draw  [color={rgb, 255:red, 208; green, 2; blue, 27 }  ,draw opacity=1 ][line width=1.5]  (243.75,102.04) .. controls (243.75,96.57) and (248.18,92.14) .. (253.65,92.14) .. controls (259.11,92.14) and (263.55,96.57) .. (263.55,102.04) .. controls (263.55,107.51) and (259.11,111.94) .. (253.65,111.94) .. controls (248.18,111.94) and (243.75,107.51) .. (243.75,102.04) -- cycle ;
%Shape: Circle [id:dp52514565121098] 
\draw  [color={rgb, 255:red, 208; green, 2; blue, 27 }  ,draw opacity=1 ][line width=1.5]  (213.58,174.86) .. controls (213.58,169.39) and (218.02,164.96) .. (223.48,164.96) .. controls (228.95,164.96) and (233.38,169.39) .. (233.38,174.86) .. controls (233.38,180.33) and (228.95,184.76) .. (223.48,184.76) .. controls (218.02,184.76) and (213.58,180.33) .. (213.58,174.86) -- cycle ;
%Shape: Regular Polygon [id:dp35103464027613396] 
\draw   (356.65,144.7) -- (326.48,174.86) -- (283.83,174.86) -- (253.66,144.7) -- (253.66,102.04) -- (283.83,71.88) -- (326.48,71.88) -- (356.65,102.04) -- cycle ;
%Shape: Circle [id:dp7871150034782447] 
\draw  [fill={rgb, 255:red, 0; green, 0; blue, 0 }  ,fill opacity=1 ] (279.38,71.88) .. controls (279.38,69.42) and (281.37,67.43) .. (283.83,67.43) .. controls (286.28,67.43) and (288.28,69.42) .. (288.28,71.88) .. controls (288.28,74.33) and (286.28,76.33) .. (283.83,76.33) .. controls (281.37,76.33) and (279.38,74.33) .. (279.38,71.88) -- cycle ;
%Shape: Circle [id:dp41829882582572575] 
\draw  [fill={rgb, 255:red, 0; green, 0; blue, 0 }  ,fill opacity=1 ] (322.03,71.88) .. controls (322.03,69.42) and (324.03,67.43) .. (326.48,67.43) .. controls (328.94,67.43) and (330.93,69.42) .. (330.93,71.88) .. controls (330.93,74.33) and (328.94,76.33) .. (326.48,76.33) .. controls (324.03,76.33) and (322.03,74.33) .. (322.03,71.88) -- cycle ;
%Shape: Circle [id:dp9126182188230112] 
\draw  [fill={rgb, 255:red, 0; green, 0; blue, 0 }  ,fill opacity=1 ] (279.38,174.86) .. controls (279.38,172.4) and (281.37,170.41) .. (283.83,170.41) .. controls (286.28,170.41) and (288.28,172.4) .. (288.28,174.86) .. controls (288.28,177.32) and (286.28,179.31) .. (283.83,179.31) .. controls (281.37,179.31) and (279.38,177.32) .. (279.38,174.86) -- cycle ;
%Shape: Circle [id:dp3692453094121052] 
\draw  [fill={rgb, 255:red, 0; green, 0; blue, 0 }  ,fill opacity=1 ] (322.03,174.86) .. controls (322.03,172.4) and (324.03,170.41) .. (326.48,170.41) .. controls (328.94,170.41) and (330.93,172.4) .. (330.93,174.86) .. controls (330.93,177.32) and (328.94,179.31) .. (326.48,179.31) .. controls (324.03,179.31) and (322.03,177.32) .. (322.03,174.86) -- cycle ;
%Shape: Circle [id:dp9382628375957761] 
\draw  [fill={rgb, 255:red, 0; green, 0; blue, 0 }  ,fill opacity=1 ] (352.2,144.7) .. controls (352.2,142.24) and (354.19,140.25) .. (356.65,140.25) .. controls (359.1,140.25) and (361.1,142.24) .. (361.1,144.7) .. controls (361.1,147.15) and (359.1,149.15) .. (356.65,149.15) .. controls (354.19,149.15) and (352.2,147.15) .. (352.2,144.7) -- cycle ;
%Shape: Circle [id:dp28376660477087756] 
\draw  [fill={rgb, 255:red, 0; green, 0; blue, 0 }  ,fill opacity=1 ] (352.2,102.04) .. controls (352.2,99.58) and (354.19,97.59) .. (356.65,97.59) .. controls (359.1,97.59) and (361.1,99.58) .. (361.1,102.04) .. controls (361.1,104.5) and (359.1,106.49) .. (356.65,106.49) .. controls (354.19,106.49) and (352.2,104.5) .. (352.2,102.04) -- cycle ;
%Shape: Circle [id:dp49005937506094366] 
\draw  [color={rgb, 255:red, 208; green, 2; blue, 27 }  ,draw opacity=1 ][line width=1.5]  (346.75,102.04) .. controls (346.75,96.57) and (351.18,92.14) .. (356.65,92.14) .. controls (362.11,92.14) and (366.55,96.57) .. (366.55,102.04) .. controls (366.55,107.51) and (362.11,111.94) .. (356.65,111.94) .. controls (351.18,111.94) and (346.75,107.51) .. (346.75,102.04) -- cycle ;
%Shape: Circle [id:dp3902625867483819] 
\draw  [color={rgb, 255:red, 208; green, 2; blue, 27 }  ,draw opacity=1 ][line width=1.5]  (316.58,174.86) .. controls (316.58,169.39) and (321.02,164.96) .. (326.48,164.96) .. controls (331.95,164.96) and (336.38,169.39) .. (336.38,174.86) .. controls (336.38,180.33) and (331.95,184.76) .. (326.48,184.76) .. controls (321.02,184.76) and (316.58,180.33) .. (316.58,174.86) -- cycle ;

\end{tikzpicture}

%% file: h_1_4_dom.tex
\tikzset{every picture/.style={line width=0.75pt}} %set default line width to 0.75pt        

\begin{tikzpicture}[x=0.75pt,y=0.75pt,yscale=-.6,xscale=.6]
%uncomment if require: \path (0,300); %set diagram left start at 0, and has height of 300

%Shape: Regular Polygon [id:dp23975975494886748] 
\draw   (150.65,144.7) -- (120.48,174.86) -- (77.83,174.86) -- (47.66,144.7) -- (47.66,102.04) -- (77.83,71.88) -- (120.48,71.88) -- (150.65,102.04) -- cycle ;
%Shape: Circle [id:dp11761017406636487] 
\draw  [fill={rgb, 255:red, 0; green, 0; blue, 0 }  ,fill opacity=1 ] (43.21,102.04) .. controls (43.21,99.58) and (45.21,97.59) .. (47.66,97.59) .. controls (50.12,97.59) and (52.11,99.58) .. (52.11,102.04) .. controls (52.11,104.5) and (50.12,106.49) .. (47.66,106.49) .. controls (45.21,106.49) and (43.21,104.5) .. (43.21,102.04) -- cycle ;
%Shape: Circle [id:dp9779656392454362] 
\draw  [fill={rgb, 255:red, 0; green, 0; blue, 0 }  ,fill opacity=1 ] (73.38,71.88) .. controls (73.38,69.42) and (75.37,67.43) .. (77.83,67.43) .. controls (80.28,67.43) and (82.28,69.42) .. (82.28,71.88) .. controls (82.28,74.33) and (80.28,76.33) .. (77.83,76.33) .. controls (75.37,76.33) and (73.38,74.33) .. (73.38,71.88) -- cycle ;
%Shape: Circle [id:dp2620722135306457] 
\draw  [fill={rgb, 255:red, 0; green, 0; blue, 0 }  ,fill opacity=1 ] (116.03,71.88) .. controls (116.03,69.42) and (118.03,67.43) .. (120.48,67.43) .. controls (122.94,67.43) and (124.93,69.42) .. (124.93,71.88) .. controls (124.93,74.33) and (122.94,76.33) .. (120.48,76.33) .. controls (118.03,76.33) and (116.03,74.33) .. (116.03,71.88) -- cycle ;
%Shape: Circle [id:dp42639029372559467] 
\draw  [fill={rgb, 255:red, 0; green, 0; blue, 0 }  ,fill opacity=1 ] (73.38,174.86) .. controls (73.38,172.4) and (75.37,170.41) .. (77.83,170.41) .. controls (80.28,170.41) and (82.28,172.4) .. (82.28,174.86) .. controls (82.28,177.32) and (80.28,179.31) .. (77.83,179.31) .. controls (75.37,179.31) and (73.38,177.32) .. (73.38,174.86) -- cycle ;
%Shape: Circle [id:dp21885088998736768] 
\draw  [fill={rgb, 255:red, 0; green, 0; blue, 0 }  ,fill opacity=1 ] (43.21,144.7) .. controls (43.21,142.24) and (45.21,140.25) .. (47.66,140.25) .. controls (50.12,140.25) and (52.11,142.24) .. (52.11,144.7) .. controls (52.11,147.15) and (50.12,149.15) .. (47.66,149.15) .. controls (45.21,149.15) and (43.21,147.15) .. (43.21,144.7) -- cycle ;
%Shape: Circle [id:dp7196190379300513] 
\draw  [fill={rgb, 255:red, 0; green, 0; blue, 0 }  ,fill opacity=1 ] (116.03,174.86) .. controls (116.03,172.4) and (118.03,170.41) .. (120.48,170.41) .. controls (122.94,170.41) and (124.93,172.4) .. (124.93,174.86) .. controls (124.93,177.32) and (122.94,179.31) .. (120.48,179.31) .. controls (118.03,179.31) and (116.03,177.32) .. (116.03,174.86) -- cycle ;
%Shape: Circle [id:dp617505296726992] 
\draw  [fill={rgb, 255:red, 0; green, 0; blue, 0 }  ,fill opacity=1 ] (146.2,144.7) .. controls (146.2,142.24) and (148.19,140.25) .. (150.65,140.25) .. controls (153.1,140.25) and (155.1,142.24) .. (155.1,144.7) .. controls (155.1,147.15) and (153.1,149.15) .. (150.65,149.15) .. controls (148.19,149.15) and (146.2,147.15) .. (146.2,144.7) -- cycle ;
%Shape: Circle [id:dp8013240969169696] 
\draw  [fill={rgb, 255:red, 0; green, 0; blue, 0 }  ,fill opacity=1 ] (146.2,102.04) .. controls (146.2,99.58) and (148.19,97.59) .. (150.65,97.59) .. controls (153.1,97.59) and (155.1,99.58) .. (155.1,102.04) .. controls (155.1,104.5) and (153.1,106.49) .. (150.65,106.49) .. controls (148.19,106.49) and (146.2,104.5) .. (146.2,102.04) -- cycle ;
%Shape: Circle [id:dp3482988392650931] 
\draw  [color={rgb, 255:red, 208; green, 2; blue, 27 }  ,draw opacity=1 ][line width=1.5]  (37.76,102.04) .. controls (37.76,96.57) and (42.2,92.14) .. (47.66,92.14) .. controls (53.13,92.14) and (57.56,96.57) .. (57.56,102.04) .. controls (57.56,107.51) and (53.13,111.94) .. (47.66,111.94) .. controls (42.2,111.94) and (37.76,107.51) .. (37.76,102.04) -- cycle ;
%Shape: Circle [id:dp052152663926080645] 
\draw  [color={rgb, 255:red, 208; green, 2; blue, 27 }  ,draw opacity=1 ][line width=1.5]  (140.75,102.04) .. controls (140.75,96.57) and (145.18,92.14) .. (150.65,92.14) .. controls (156.11,92.14) and (160.55,96.57) .. (160.55,102.04) .. controls (160.55,107.51) and (156.11,111.94) .. (150.65,111.94) .. controls (145.18,111.94) and (140.75,107.51) .. (140.75,102.04) -- cycle ;
%Shape: Circle [id:dp2683454888264245] 
\draw  [color={rgb, 255:red, 208; green, 2; blue, 27 }  ,draw opacity=1 ][line width=1.5]  (110.58,174.86) .. controls (110.58,169.39) and (115.02,164.96) .. (120.48,164.96) .. controls (125.95,164.96) and (130.38,169.39) .. (130.38,174.86) .. controls (130.38,180.33) and (125.95,184.76) .. (120.48,184.76) .. controls (115.02,184.76) and (110.58,180.33) .. (110.58,174.86) -- cycle ;
%Shape: Regular Polygon [id:dp887599755661273] 
\draw   (253.65,144.7) -- (223.48,174.86) -- (180.83,174.86) -- (150.66,144.7) -- (150.66,102.04) -- (180.83,71.88) -- (223.48,71.88) -- (253.65,102.04) -- cycle ;
%Shape: Circle [id:dp20702660397933614] 
\draw  [fill={rgb, 255:red, 0; green, 0; blue, 0 }  ,fill opacity=1 ] (176.38,71.88) .. controls (176.38,69.42) and (178.37,67.43) .. (180.83,67.43) .. controls (183.28,67.43) and (185.28,69.42) .. (185.28,71.88) .. controls (185.28,74.33) and (183.28,76.33) .. (180.83,76.33) .. controls (178.37,76.33) and (176.38,74.33) .. (176.38,71.88) -- cycle ;
%Shape: Circle [id:dp27283332462681864] 
\draw  [fill={rgb, 255:red, 0; green, 0; blue, 0 }  ,fill opacity=1 ] (219.03,71.88) .. controls (219.03,69.42) and (221.03,67.43) .. (223.48,67.43) .. controls (225.94,67.43) and (227.93,69.42) .. (227.93,71.88) .. controls (227.93,74.33) and (225.94,76.33) .. (223.48,76.33) .. controls (221.03,76.33) and (219.03,74.33) .. (219.03,71.88) -- cycle ;
%Shape: Circle [id:dp11455648298227561] 
\draw  [fill={rgb, 255:red, 0; green, 0; blue, 0 }  ,fill opacity=1 ] (176.38,174.86) .. controls (176.38,172.4) and (178.37,170.41) .. (180.83,170.41) .. controls (183.28,170.41) and (185.28,172.4) .. (185.28,174.86) .. controls (185.28,177.32) and (183.28,179.31) .. (180.83,179.31) .. controls (178.37,179.31) and (176.38,177.32) .. (176.38,174.86) -- cycle ;
%Shape: Circle [id:dp8536624921732355] 
\draw  [fill={rgb, 255:red, 0; green, 0; blue, 0 }  ,fill opacity=1 ] (219.03,174.86) .. controls (219.03,172.4) and (221.03,170.41) .. (223.48,170.41) .. controls (225.94,170.41) and (227.93,172.4) .. (227.93,174.86) .. controls (227.93,177.32) and (225.94,179.31) .. (223.48,179.31) .. controls (221.03,179.31) and (219.03,177.32) .. (219.03,174.86) -- cycle ;
%Shape: Circle [id:dp07971455068400357] 
\draw  [fill={rgb, 255:red, 0; green, 0; blue, 0 }  ,fill opacity=1 ] (249.2,144.7) .. controls (249.2,142.24) and (251.19,140.25) .. (253.65,140.25) .. controls (256.1,140.25) and (258.1,142.24) .. (258.1,144.7) .. controls (258.1,147.15) and (256.1,149.15) .. (253.65,149.15) .. controls (251.19,149.15) and (249.2,147.15) .. (249.2,144.7) -- cycle ;
%Shape: Circle [id:dp18609421540781912] 
\draw  [fill={rgb, 255:red, 0; green, 0; blue, 0 }  ,fill opacity=1 ] (249.2,102.04) .. controls (249.2,99.58) and (251.19,97.59) .. (253.65,97.59) .. controls (256.1,97.59) and (258.1,99.58) .. (258.1,102.04) .. controls (258.1,104.5) and (256.1,106.49) .. (253.65,106.49) .. controls (251.19,106.49) and (249.2,104.5) .. (249.2,102.04) -- cycle ;
%Shape: Circle [id:dp9035424830258048] 
\draw  [color={rgb, 255:red, 208; green, 2; blue, 27 }  ,draw opacity=1 ][line width=1.5]  (243.75,102.04) .. controls (243.75,96.57) and (248.18,92.14) .. (253.65,92.14) .. controls (259.11,92.14) and (263.55,96.57) .. (263.55,102.04) .. controls (263.55,107.51) and (259.11,111.94) .. (253.65,111.94) .. controls (248.18,111.94) and (243.75,107.51) .. (243.75,102.04) -- cycle ;
%Shape: Circle [id:dp52514565121098] 
\draw  [color={rgb, 255:red, 208; green, 2; blue, 27 }  ,draw opacity=1 ][line width=1.5]  (213.58,174.86) .. controls (213.58,169.39) and (218.02,164.96) .. (223.48,164.96) .. controls (228.95,164.96) and (233.38,169.39) .. (233.38,174.86) .. controls (233.38,180.33) and (228.95,184.76) .. (223.48,184.76) .. controls (218.02,184.76) and (213.58,180.33) .. (213.58,174.86) -- cycle ;
%Shape: Regular Polygon [id:dp35103464027613396] 
\draw   (356.65,144.7) -- (326.48,174.86) -- (283.83,174.86) -- (253.66,144.7) -- (253.66,102.04) -- (283.83,71.88) -- (326.48,71.88) -- (356.65,102.04) -- cycle ;
%Shape: Circle [id:dp7871150034782447] 
\draw  [fill={rgb, 255:red, 0; green, 0; blue, 0 }  ,fill opacity=1 ] (279.38,71.88) .. controls (279.38,69.42) and (281.37,67.43) .. (283.83,67.43) .. controls (286.28,67.43) and (288.28,69.42) .. (288.28,71.88) .. controls (288.28,74.33) and (286.28,76.33) .. (283.83,76.33) .. controls (281.37,76.33) and (279.38,74.33) .. (279.38,71.88) -- cycle ;
%Shape: Circle [id:dp41829882582572575] 
\draw  [fill={rgb, 255:red, 0; green, 0; blue, 0 }  ,fill opacity=1 ] (322.03,71.88) .. controls (322.03,69.42) and (324.03,67.43) .. (326.48,67.43) .. controls (328.94,67.43) and (330.93,69.42) .. (330.93,71.88) .. controls (330.93,74.33) and (328.94,76.33) .. (326.48,76.33) .. controls (324.03,76.33) and (322.03,74.33) .. (322.03,71.88) -- cycle ;
%Shape: Circle [id:dp9126182188230112] 
\draw  [fill={rgb, 255:red, 0; green, 0; blue, 0 }  ,fill opacity=1 ] (279.38,174.86) .. controls (279.38,172.4) and (281.37,170.41) .. (283.83,170.41) .. controls (286.28,170.41) and (288.28,172.4) .. (288.28,174.86) .. controls (288.28,177.32) and (286.28,179.31) .. (283.83,179.31) .. controls (281.37,179.31) and (279.38,177.32) .. (279.38,174.86) -- cycle ;
%Shape: Circle [id:dp3692453094121052] 
\draw  [fill={rgb, 255:red, 0; green, 0; blue, 0 }  ,fill opacity=1 ] (322.03,174.86) .. controls (322.03,172.4) and (324.03,170.41) .. (326.48,170.41) .. controls (328.94,170.41) and (330.93,172.4) .. (330.93,174.86) .. controls (330.93,177.32) and (328.94,179.31) .. (326.48,179.31) .. controls (324.03,179.31) and (322.03,177.32) .. (322.03,174.86) -- cycle ;
%Shape: Circle [id:dp9382628375957761] 
\draw  [fill={rgb, 255:red, 0; green, 0; blue, 0 }  ,fill opacity=1 ] (352.2,144.7) .. controls (352.2,142.24) and (354.19,140.25) .. (356.65,140.25) .. controls (359.1,140.25) and (361.1,142.24) .. (361.1,144.7) .. controls (361.1,147.15) and (359.1,149.15) .. (356.65,149.15) .. controls (354.19,149.15) and (352.2,147.15) .. (352.2,144.7) -- cycle ;
%Shape: Circle [id:dp28376660477087756] 
\draw  [fill={rgb, 255:red, 0; green, 0; blue, 0 }  ,fill opacity=1 ] (352.2,102.04) .. controls (352.2,99.58) and (354.19,97.59) .. (356.65,97.59) .. controls (359.1,97.59) and (361.1,99.58) .. (361.1,102.04) .. controls (361.1,104.5) and (359.1,106.49) .. (356.65,106.49) .. controls (354.19,106.49) and (352.2,104.5) .. (352.2,102.04) -- cycle ;
%Shape: Circle [id:dp49005937506094366] 
\draw  [color={rgb, 255:red, 208; green, 2; blue, 27 }  ,draw opacity=1 ][line width=1.5]  (346.75,102.04) .. controls (346.75,96.57) and (351.18,92.14) .. (356.65,92.14) .. controls (362.11,92.14) and (366.55,96.57) .. (366.55,102.04) .. controls (366.55,107.51) and (362.11,111.94) .. (356.65,111.94) .. controls (351.18,111.94) and (346.75,107.51) .. (346.75,102.04) -- cycle ;
%Shape: Circle [id:dp3902625867483819] 
\draw  [color={rgb, 255:red, 208; green, 2; blue, 27 }  ,draw opacity=1 ][line width=1.5]  (316.58,174.86) .. controls (316.58,169.39) and (321.02,164.96) .. (326.48,164.96) .. controls (331.95,164.96) and (336.38,169.39) .. (336.38,174.86) .. controls (336.38,180.33) and (331.95,184.76) .. (326.48,184.76) .. controls (321.02,184.76) and (316.58,180.33) .. (316.58,174.86) -- cycle ;
%Shape: Regular Polygon [id:dp2131447190826733] 
\draw   (459.65,144.7) -- (429.48,174.86) -- (386.83,174.86) -- (356.66,144.7) -- (356.66,102.04) -- (386.83,71.88) -- (429.48,71.88) -- (459.65,102.04) -- cycle ;
%Shape: Circle [id:dp27901870705143905] 
\draw  [fill={rgb, 255:red, 0; green, 0; blue, 0 }  ,fill opacity=1 ] (382.38,71.88) .. controls (382.38,69.42) and (384.37,67.43) .. (386.83,67.43) .. controls (389.28,67.43) and (391.28,69.42) .. (391.28,71.88) .. controls (391.28,74.33) and (389.28,76.33) .. (386.83,76.33) .. controls (384.37,76.33) and (382.38,74.33) .. (382.38,71.88) -- cycle ;
%Shape: Circle [id:dp5639839984713698] 
\draw  [fill={rgb, 255:red, 0; green, 0; blue, 0 }  ,fill opacity=1 ] (425.03,71.88) .. controls (425.03,69.42) and (427.03,67.43) .. (429.48,67.43) .. controls (431.94,67.43) and (433.93,69.42) .. (433.93,71.88) .. controls (433.93,74.33) and (431.94,76.33) .. (429.48,76.33) .. controls (427.03,76.33) and (425.03,74.33) .. (425.03,71.88) -- cycle ;
%Shape: Circle [id:dp8536509185994338] 
\draw  [fill={rgb, 255:red, 0; green, 0; blue, 0 }  ,fill opacity=1 ] (382.38,174.86) .. controls (382.38,172.4) and (384.37,170.41) .. (386.83,170.41) .. controls (389.28,170.41) and (391.28,172.4) .. (391.28,174.86) .. controls (391.28,177.32) and (389.28,179.31) .. (386.83,179.31) .. controls (384.37,179.31) and (382.38,177.32) .. (382.38,174.86) -- cycle ;
%Shape: Circle [id:dp4721083117509851] 
\draw  [fill={rgb, 255:red, 0; green, 0; blue, 0 }  ,fill opacity=1 ] (425.03,174.86) .. controls (425.03,172.4) and (427.03,170.41) .. (429.48,170.41) .. controls (431.94,170.41) and (433.93,172.4) .. (433.93,174.86) .. controls (433.93,177.32) and (431.94,179.31) .. (429.48,179.31) .. controls (427.03,179.31) and (425.03,177.32) .. (425.03,174.86) -- cycle ;
%Shape: Circle [id:dp7657054443168595] 
\draw  [fill={rgb, 255:red, 0; green, 0; blue, 0 }  ,fill opacity=1 ] (455.2,144.7) .. controls (455.2,142.24) and (457.19,140.25) .. (459.65,140.25) .. controls (462.1,140.25) and (464.1,142.24) .. (464.1,144.7) .. controls (464.1,147.15) and (462.1,149.15) .. (459.65,149.15) .. controls (457.19,149.15) and (455.2,147.15) .. (455.2,144.7) -- cycle ;
%Shape: Circle [id:dp7435520562882836] 
\draw  [fill={rgb, 255:red, 0; green, 0; blue, 0 }  ,fill opacity=1 ] (455.2,102.04) .. controls (455.2,99.58) and (457.19,97.59) .. (459.65,97.59) .. controls (462.1,97.59) and (464.1,99.58) .. (464.1,102.04) .. controls (464.1,104.5) and (462.1,106.49) .. (459.65,106.49) .. controls (457.19,106.49) and (455.2,104.5) .. (455.2,102.04) -- cycle ;
%Shape: Circle [id:dp6631828465261103] 
\draw  [color={rgb, 255:red, 208; green, 2; blue, 27 }  ,draw opacity=1 ][line width=1.5]  (449.75,102.04) .. controls (449.75,96.57) and (454.18,92.14) .. (459.65,92.14) .. controls (465.11,92.14) and (469.55,96.57) .. (469.55,102.04) .. controls (469.55,107.51) and (465.11,111.94) .. (459.65,111.94) .. controls (454.18,111.94) and (449.75,107.51) .. (449.75,102.04) -- cycle ;
%Shape: Circle [id:dp39398642166716946] 
\draw  [color={rgb, 255:red, 208; green, 2; blue, 27 }  ,draw opacity=1 ][line width=1.5]  (419.58,174.86) .. controls (419.58,169.39) and (424.02,164.96) .. (429.48,164.96) .. controls (434.95,164.96) and (439.38,169.39) .. (439.38,174.86) .. controls (439.38,180.33) and (434.95,184.76) .. (429.48,184.76) .. controls (424.02,184.76) and (419.58,180.33) .. (419.58,174.86) -- cycle ;

\end{tikzpicture}

%% file: h_2_2_dom.tex
\tikzset{every picture/.style={line width=0.75pt}} %set default line width to 0.75pt        

\begin{tikzpicture}[x=0.75pt,y=0.75pt,yscale=-.75,xscale=.75]
%uncomment if require: \path (0,300); %set diagram left start at 0, and has height of 300

%Shape: Regular Polygon [id:dp23975975494886748] 
\draw   (235.65,111.7) -- (205.48,141.86) -- (162.83,141.86) -- (132.66,111.7) -- (132.66,69.04) -- (162.83,38.88) -- (205.48,38.88) -- (235.65,69.04) -- cycle ;
%Shape: Circle [id:dp11761017406636487] 
\draw  [fill={rgb, 255:red, 0; green, 0; blue, 0 }  ,fill opacity=1 ] (128.21,69.04) .. controls (128.21,66.58) and (130.21,64.59) .. (132.66,64.59) .. controls (135.12,64.59) and (137.11,66.58) .. (137.11,69.04) .. controls (137.11,71.5) and (135.12,73.49) .. (132.66,73.49) .. controls (130.21,73.49) and (128.21,71.5) .. (128.21,69.04) -- cycle ;
%Shape: Circle [id:dp9779656392454362] 
\draw  [fill={rgb, 255:red, 0; green, 0; blue, 0 }  ,fill opacity=1 ] (158.38,38.88) .. controls (158.38,36.42) and (160.37,34.43) .. (162.83,34.43) .. controls (165.28,34.43) and (167.28,36.42) .. (167.28,38.88) .. controls (167.28,41.33) and (165.28,43.33) .. (162.83,43.33) .. controls (160.37,43.33) and (158.38,41.33) .. (158.38,38.88) -- cycle ;
%Shape: Circle [id:dp2620722135306457] 
\draw  [fill={rgb, 255:red, 0; green, 0; blue, 0 }  ,fill opacity=1 ] (201.03,38.88) .. controls (201.03,36.42) and (203.03,34.43) .. (205.48,34.43) .. controls (207.94,34.43) and (209.93,36.42) .. (209.93,38.88) .. controls (209.93,41.33) and (207.94,43.33) .. (205.48,43.33) .. controls (203.03,43.33) and (201.03,41.33) .. (201.03,38.88) -- cycle ;
%Shape: Circle [id:dp42639029372559467] 
\draw  [fill={rgb, 255:red, 0; green, 0; blue, 0 }  ,fill opacity=1 ] (158.38,141.86) .. controls (158.38,139.4) and (160.37,137.41) .. (162.83,137.41) .. controls (165.28,137.41) and (167.28,139.4) .. (167.28,141.86) .. controls (167.28,144.32) and (165.28,146.31) .. (162.83,146.31) .. controls (160.37,146.31) and (158.38,144.32) .. (158.38,141.86) -- cycle ;
%Shape: Circle [id:dp21885088998736768] 
\draw  [fill={rgb, 255:red, 0; green, 0; blue, 0 }  ,fill opacity=1 ] (128.21,111.7) .. controls (128.21,109.24) and (130.21,107.25) .. (132.66,107.25) .. controls (135.12,107.25) and (137.11,109.24) .. (137.11,111.7) .. controls (137.11,114.15) and (135.12,116.15) .. (132.66,116.15) .. controls (130.21,116.15) and (128.21,114.15) .. (128.21,111.7) -- cycle ;
%Shape: Circle [id:dp7196190379300513] 
\draw  [fill={rgb, 255:red, 0; green, 0; blue, 0 }  ,fill opacity=1 ] (201.03,141.86) .. controls (201.03,139.4) and (203.03,137.41) .. (205.48,137.41) .. controls (207.94,137.41) and (209.93,139.4) .. (209.93,141.86) .. controls (209.93,144.32) and (207.94,146.31) .. (205.48,146.31) .. controls (203.03,146.31) and (201.03,144.32) .. (201.03,141.86) -- cycle ;
%Shape: Circle [id:dp617505296726992] 
\draw  [fill={rgb, 255:red, 0; green, 0; blue, 0 }  ,fill opacity=1 ] (231.2,111.7) .. controls (231.2,109.24) and (233.19,107.25) .. (235.65,107.25) .. controls (238.1,107.25) and (240.1,109.24) .. (240.1,111.7) .. controls (240.1,114.15) and (238.1,116.15) .. (235.65,116.15) .. controls (233.19,116.15) and (231.2,114.15) .. (231.2,111.7) -- cycle ;
%Shape: Circle [id:dp8013240969169696] 
\draw  [fill={rgb, 255:red, 0; green, 0; blue, 0 }  ,fill opacity=1 ] (231.2,69.04) .. controls (231.2,66.58) and (233.19,64.59) .. (235.65,64.59) .. controls (238.1,64.59) and (240.1,66.58) .. (240.1,69.04) .. controls (240.1,71.5) and (238.1,73.49) .. (235.65,73.49) .. controls (233.19,73.49) and (231.2,71.5) .. (231.2,69.04) -- cycle ;
%Shape: Circle [id:dp3482988392650931] 
\draw  [color={rgb, 255:red, 208; green, 2; blue, 27 }  ,draw opacity=1 ][line width=1.5]  (122.76,69.04) .. controls (122.76,63.57) and (127.2,59.14) .. (132.66,59.14) .. controls (138.13,59.14) and (142.56,63.57) .. (142.56,69.04) .. controls (142.56,74.51) and (138.13,78.94) .. (132.66,78.94) .. controls (127.2,78.94) and (122.76,74.51) .. (122.76,69.04) -- cycle ;
%Shape: Circle [id:dp052152663926080645] 
\draw  [color={rgb, 255:red, 208; green, 2; blue, 27 }  ,draw opacity=1 ][line width=1.5]  (225.75,69.04) .. controls (225.75,63.57) and (230.18,59.14) .. (235.65,59.14) .. controls (241.11,59.14) and (245.55,63.57) .. (245.55,69.04) .. controls (245.55,74.51) and (241.11,78.94) .. (235.65,78.94) .. controls (230.18,78.94) and (225.75,74.51) .. (225.75,69.04) -- cycle ;
%Shape: Circle [id:dp2683454888264245] 
\draw  [color={rgb, 255:red, 208; green, 2; blue, 27 }  ,draw opacity=1 ][line width=1.5]  (195.58,141.86) .. controls (195.58,136.39) and (200.02,131.96) .. (205.48,131.96) .. controls (210.95,131.96) and (215.38,136.39) .. (215.38,141.86) .. controls (215.38,147.33) and (210.95,151.76) .. (205.48,151.76) .. controls (200.02,151.76) and (195.58,147.33) .. (195.58,141.86) -- cycle ;
%Shape: Regular Polygon [id:dp887599755661273] 
\draw   (338.65,111.7) -- (308.48,141.86) -- (265.83,141.86) -- (235.66,111.7) -- (235.66,69.04) -- (265.83,38.88) -- (308.48,38.88) -- (338.65,69.04) -- cycle ;
%Shape: Circle [id:dp20702660397933614] 
\draw  [fill={rgb, 255:red, 0; green, 0; blue, 0 }  ,fill opacity=1 ] (261.38,38.88) .. controls (261.38,36.42) and (263.37,34.43) .. (265.83,34.43) .. controls (268.28,34.43) and (270.28,36.42) .. (270.28,38.88) .. controls (270.28,41.33) and (268.28,43.33) .. (265.83,43.33) .. controls (263.37,43.33) and (261.38,41.33) .. (261.38,38.88) -- cycle ;
%Shape: Circle [id:dp27283332462681864] 
\draw  [fill={rgb, 255:red, 0; green, 0; blue, 0 }  ,fill opacity=1 ] (304.03,38.88) .. controls (304.03,36.42) and (306.03,34.43) .. (308.48,34.43) .. controls (310.94,34.43) and (312.93,36.42) .. (312.93,38.88) .. controls (312.93,41.33) and (310.94,43.33) .. (308.48,43.33) .. controls (306.03,43.33) and (304.03,41.33) .. (304.03,38.88) -- cycle ;
%Shape: Circle [id:dp11455648298227561] 
\draw  [fill={rgb, 255:red, 0; green, 0; blue, 0 }  ,fill opacity=1 ] (261.38,141.86) .. controls (261.38,139.4) and (263.37,137.41) .. (265.83,137.41) .. controls (268.28,137.41) and (270.28,139.4) .. (270.28,141.86) .. controls (270.28,144.32) and (268.28,146.31) .. (265.83,146.31) .. controls (263.37,146.31) and (261.38,144.32) .. (261.38,141.86) -- cycle ;
%Shape: Circle [id:dp8536624921732355] 
\draw  [fill={rgb, 255:red, 0; green, 0; blue, 0 }  ,fill opacity=1 ] (304.03,141.86) .. controls (304.03,139.4) and (306.03,137.41) .. (308.48,137.41) .. controls (310.94,137.41) and (312.93,139.4) .. (312.93,141.86) .. controls (312.93,144.32) and (310.94,146.31) .. (308.48,146.31) .. controls (306.03,146.31) and (304.03,144.32) .. (304.03,141.86) -- cycle ;
%Shape: Circle [id:dp07971455068400357] 
\draw  [fill={rgb, 255:red, 0; green, 0; blue, 0 }  ,fill opacity=1 ] (334.2,111.7) .. controls (334.2,109.24) and (336.19,107.25) .. (338.65,107.25) .. controls (341.1,107.25) and (343.1,109.24) .. (343.1,111.7) .. controls (343.1,114.15) and (341.1,116.15) .. (338.65,116.15) .. controls (336.19,116.15) and (334.2,114.15) .. (334.2,111.7) -- cycle ;
%Shape: Circle [id:dp18609421540781912] 
\draw  [fill={rgb, 255:red, 0; green, 0; blue, 0 }  ,fill opacity=1 ] (334.2,69.04) .. controls (334.2,66.58) and (336.19,64.59) .. (338.65,64.59) .. controls (341.1,64.59) and (343.1,66.58) .. (343.1,69.04) .. controls (343.1,71.5) and (341.1,73.49) .. (338.65,73.49) .. controls (336.19,73.49) and (334.2,71.5) .. (334.2,69.04) -- cycle ;
%Shape: Circle [id:dp9035424830258048] 
\draw  [color={rgb, 255:red, 208; green, 2; blue, 27 }  ,draw opacity=1 ][line width=1.5]  (328.75,69.04) .. controls (328.75,63.57) and (333.18,59.14) .. (338.65,59.14) .. controls (344.11,59.14) and (348.55,63.57) .. (348.55,69.04) .. controls (348.55,74.51) and (344.11,78.94) .. (338.65,78.94) .. controls (333.18,78.94) and (328.75,74.51) .. (328.75,69.04) -- cycle ;
%Shape: Circle [id:dp52514565121098] 
\draw  [color={rgb, 255:red, 208; green, 2; blue, 27 }  ,draw opacity=1 ][line width=1.5]  (298.58,141.86) .. controls (298.58,136.39) and (303.02,131.96) .. (308.48,131.96) .. controls (313.95,131.96) and (318.38,136.39) .. (318.38,141.86) .. controls (318.38,147.33) and (313.95,151.76) .. (308.48,151.76) .. controls (303.02,151.76) and (298.58,147.33) .. (298.58,141.86) -- cycle ;
%Shape: Regular Polygon [id:dp35103464027613396] 
\draw   (235.65,214.7) -- (205.48,244.86) -- (162.83,244.86) -- (132.66,214.7) -- (132.66,172.04) -- (162.83,141.88) -- (205.48,141.88) -- (235.65,172.04) -- cycle ;
%Shape: Circle [id:dp7871150034782447] 
\draw  [fill={rgb, 255:red, 0; green, 0; blue, 0 }  ,fill opacity=1 ] (128.21,172.04) .. controls (128.21,169.58) and (130.21,167.59) .. (132.66,167.59) .. controls (135.12,167.59) and (137.11,169.58) .. (137.11,172.04) .. controls (137.11,174.5) and (135.12,176.49) .. (132.66,176.49) .. controls (130.21,176.49) and (128.21,174.5) .. (128.21,172.04) -- cycle ;
%Shape: Circle [id:dp9126182188230112] 
\draw  [fill={rgb, 255:red, 0; green, 0; blue, 0 }  ,fill opacity=1 ] (158.38,244.86) .. controls (158.38,242.4) and (160.37,240.41) .. (162.83,240.41) .. controls (165.28,240.41) and (167.28,242.4) .. (167.28,244.86) .. controls (167.28,247.32) and (165.28,249.31) .. (162.83,249.31) .. controls (160.37,249.31) and (158.38,247.32) .. (158.38,244.86) -- cycle ;
%Shape: Circle [id:dp3692453094121052] 
\draw  [fill={rgb, 255:red, 0; green, 0; blue, 0 }  ,fill opacity=1 ] (201.03,244.86) .. controls (201.03,242.4) and (203.03,240.41) .. (205.48,240.41) .. controls (207.94,240.41) and (209.93,242.4) .. (209.93,244.86) .. controls (209.93,247.32) and (207.94,249.31) .. (205.48,249.31) .. controls (203.03,249.31) and (201.03,247.32) .. (201.03,244.86) -- cycle ;
%Shape: Circle [id:dp9382628375957761] 
\draw  [fill={rgb, 255:red, 0; green, 0; blue, 0 }  ,fill opacity=1 ] (231.2,214.7) .. controls (231.2,212.24) and (233.19,210.25) .. (235.65,210.25) .. controls (238.1,210.25) and (240.1,212.24) .. (240.1,214.7) .. controls (240.1,217.15) and (238.1,219.15) .. (235.65,219.15) .. controls (233.19,219.15) and (231.2,217.15) .. (231.2,214.7) -- cycle ;
%Shape: Circle [id:dp28376660477087756] 
\draw  [fill={rgb, 255:red, 0; green, 0; blue, 0 }  ,fill opacity=1 ] (231.2,172.04) .. controls (231.2,169.58) and (233.19,167.59) .. (235.65,167.59) .. controls (238.1,167.59) and (240.1,169.58) .. (240.1,172.04) .. controls (240.1,174.5) and (238.1,176.49) .. (235.65,176.49) .. controls (233.19,176.49) and (231.2,174.5) .. (231.2,172.04) -- cycle ;
%Shape: Circle [id:dp3902625867483819] 
\draw  [color={rgb, 255:red, 208; green, 2; blue, 27 }  ,draw opacity=1 ][line width=1.5]  (122.76,214.7) .. controls (122.76,209.23) and (127.2,204.8) .. (132.66,204.8) .. controls (138.13,204.8) and (142.56,209.23) .. (142.56,214.7) .. controls (142.56,220.16) and (138.13,224.6) .. (132.66,224.6) .. controls (127.2,224.6) and (122.76,220.16) .. (122.76,214.7) -- cycle ;
%Shape: Regular Polygon [id:dp2131447190826733] 
\draw   (338.63,214.7) -- (308.46,244.86) -- (265.81,244.86) -- (235.65,214.7) -- (235.65,172.04) -- (265.81,141.88) -- (308.46,141.88) -- (338.63,172.04) -- cycle ;
%Shape: Circle [id:dp8536509185994338] 
\draw  [fill={rgb, 255:red, 0; green, 0; blue, 0 }  ,fill opacity=1 ] (261.38,244.86) .. controls (261.38,242.4) and (263.37,240.41) .. (265.83,240.41) .. controls (268.28,240.41) and (270.28,242.4) .. (270.28,244.86) .. controls (270.28,247.32) and (268.28,249.31) .. (265.83,249.31) .. controls (263.37,249.31) and (261.38,247.32) .. (261.38,244.86) -- cycle ;
%Shape: Circle [id:dp4721083117509851] 
\draw  [fill={rgb, 255:red, 0; green, 0; blue, 0 }  ,fill opacity=1 ] (304.03,244.86) .. controls (304.03,242.4) and (306.03,240.41) .. (308.48,240.41) .. controls (310.94,240.41) and (312.93,242.4) .. (312.93,244.86) .. controls (312.93,247.32) and (310.94,249.31) .. (308.48,249.31) .. controls (306.03,249.31) and (304.03,247.32) .. (304.03,244.86) -- cycle ;
%Shape: Circle [id:dp7657054443168595] 
\draw  [fill={rgb, 255:red, 0; green, 0; blue, 0 }  ,fill opacity=1 ] (334.2,214.7) .. controls (334.2,212.24) and (336.19,210.25) .. (338.65,210.25) .. controls (341.1,210.25) and (343.1,212.24) .. (343.1,214.7) .. controls (343.1,217.15) and (341.1,219.15) .. (338.65,219.15) .. controls (336.19,219.15) and (334.2,217.15) .. (334.2,214.7) -- cycle ;
%Shape: Circle [id:dp7435520562882836] 
\draw  [fill={rgb, 255:red, 0; green, 0; blue, 0 }  ,fill opacity=1 ] (334.2,172.04) .. controls (334.2,169.58) and (336.19,167.59) .. (338.65,167.59) .. controls (341.1,167.59) and (343.1,169.58) .. (343.1,172.04) .. controls (343.1,174.5) and (341.1,176.49) .. (338.65,176.49) .. controls (336.19,176.49) and (334.2,174.5) .. (334.2,172.04) -- cycle ;
%Shape: Circle [id:dp39398642166716946] 
\draw  [color={rgb, 255:red, 208; green, 2; blue, 27 }  ,draw opacity=1 ][line width=1.5]  (328.73,214.7) .. controls (328.73,209.23) and (333.16,204.8) .. (338.63,204.8) .. controls (344.09,204.8) and (348.53,209.23) .. (348.53,214.7) .. controls (348.53,220.16) and (344.09,224.6) .. (338.63,224.6) .. controls (333.16,224.6) and (328.73,220.16) .. (328.73,214.7) -- cycle ;
%Shape: Circle [id:dp28111790778634727] 
\draw  [fill={rgb, 255:red, 0; green, 0; blue, 0 }  ,fill opacity=1 ] (128.21,214.7) .. controls (128.21,212.24) and (130.21,210.25) .. (132.66,210.25) .. controls (135.12,210.25) and (137.11,212.24) .. (137.11,214.7) .. controls (137.11,217.15) and (135.12,219.15) .. (132.66,219.15) .. controls (130.21,219.15) and (128.21,217.15) .. (128.21,214.7) -- cycle ;
%Shape: Circle [id:dp6381694869953835] 
\draw  [color={rgb, 255:red, 208; green, 2; blue, 27 }  ,draw opacity=1 ][line width=1.5]  (225.75,214.7) .. controls (225.75,209.23) and (230.18,204.8) .. (235.65,204.8) .. controls (241.11,204.8) and (245.55,209.23) .. (245.55,214.7) .. controls (245.55,220.16) and (241.11,224.6) .. (235.65,224.6) .. controls (230.18,224.6) and (225.75,220.16) .. (225.75,214.7) -- cycle ;

\end{tikzpicture}